\documentclass{amsart}
\usepackage{amssymb,amsmath,amscd,amsthm}
\usepackage{amsfonts,epsfig,latexsym,graphicx,amssymb, pict2e}

\newtheorem{Lm}{Lemma}[section]
\newtheorem{Thm}[Lm]{Theorem}
\newtheorem{Prop}[Lm]{Proposition}
\newtheorem{Def}[Lm]{Definition}
\newtheorem{Rem}[Lm]{Remark}
\newtheorem{Cor}{Corollary}
\newtheorem*{GL}{Gap Lemma}

\def\R{\mathbb{R}}

\def\Z{\mathbb{Z}}
\def\N{\mathbb{N}}

\def\T{\mathbb{T}}

\def\bdef{\begin{Def}}
\def\endef{\end{Def}}
\def\bthm{\begin{Thm}}
\def\ethm{\end{Thm}}
\def\bprop{\begin{Prop}}
\def\enprop{\end{Prop}}
\def\blm{\begin{Lm}}
\def\elm{\end{Lm}}
\def\bcor{\begin{Cor}}
\def\ecor{\end{Cor}}
\def\brm{\begin{Rem}}
\def\erm{\end{Rem}}
\def\bfig{\begin{picture}}
\def\efig{\end{picture}}
\def\beq{\begin{eqnarray}}
\def\eneq{\end{eqnarray}}
\def\beal{\begin{aligned}}
\def\enal{\end{aligned}}

\title[The Weakly Coupled Fibonacci Hamiltonian]{Spectral and Quantum Dynamical Properties of the Weakly Coupled Fibonacci Hamiltonian}

\author{David Damanik}

\address{Department of Mathematics, Rice University, Houston, TX~77005, USA}

\email{damanik@rice.edu}

\thanks{D.\ D.\ was supported in part by NSF grants DMS--0653720 and DMS--0800100.}

\author{Anton Gorodetski}

\address{Department of Mathematics, University of California, Irvine CA 92697, USA}

\email{asgor@math.uci.edu}

\thanks{A.\ G.\ was supported in part by NSF grant DMS--0901627.}

\date{\today}

\begin{document}

\begin{abstract}
We consider the spectrum of the Fibonacci Hamiltonian for small
values of the coupling constant. It is known that this set is a
Cantor set of zero Lebesgue measure. Here we study the limit, as the
value of the coupling constant approaches zero, of its thickness and
its Hausdorff dimension. We prove that the thickness
tends to infinity and, consequently, the Hausdorff dimension of the
spectrum tends to one. We also show that at small coupling, all gaps allowed by
the gap labeling theorem are open and the length of every gap tends to
zero linearly. Moreover, for sufficiently small coupling, the sum of
the spectrum with itself is an interval. This last result provides a
rigorous explanation of a phenomenon for the Fibonacci square
lattice discovered numerically by Even-Dar Mandel and Lifshitz.
Finally, we provide explicit upper and lower bounds for the solutions
to the difference equation and use them to study the spectral measures
and the transport exponents.
\end{abstract}

\maketitle

\tableofcontents

\section{Introduction}

\subsection{Background and Motivation}

It is always exciting to obtain a new connection between two
different areas of mathematics. Here we prove several new results
concerning the spectral properties of the discrete Schr\"odinger operator with
Fibonacci potential, the so-called Fibonacci Hamiltonian, using
methods from the modern theory of dynamical systems (uniformly
hyperbolic and normally hyperbolic dynamics).

The Fibonacci Hamiltonian is a central model in the study of
electronic properties of one-dimensional quasicrystals. It is
given by the following bounded self-adjoint operator in
$\ell^2(\Z)$,
$$
[H_{V,\omega} \psi] (n) = \psi(n+1) + \psi(n-1) + V
\chi_{[1-\alpha , 1)}(n\alpha + \omega \!\!\! \mod 1) \psi(n),
$$
where $V > 0$, $\alpha = \frac{\sqrt{5}-1}{2}$, and $\omega \in \T
= \R / \Z$.

This operator family has been studied in many papers since the
early 1980's and numerous fundamental results are known. Let us
recall some of them and refer the reader to the survey articles
\cite{D00,D07a,S95} for additional information.

The spectrum is easily seen to be independent of $\omega$ and may
therefore be denoted by $\Sigma_V$. That is, $\sigma(H_{V,\omega})
= \Sigma_V$ for every $\omega \in \T$. Indeed, this follows
quickly from the minimality of the irrational rotation by $\alpha$
and strong operator convergence. It was shown by S\"ut\H{o} that
$\Sigma_V$ has zero Lebesgue measure for every $V > 0$; see
\cite{S89}. Moreover, it is compact (since it is the spectrum of a
bounded operator) and perfect (because the irrational rotation by
$\alpha$ is ergodic). Thus, $\Sigma_V$ is a zero-measure Cantor
set. This result was recently strengthened by Cantat \cite{Can}
who showed that the Hausdorff dimension of $\Sigma_V$ lies
strictly between zero and one.

Naturally, one is interested in fractal properties of $\Sigma_V$,
such as its dimension, thickness, and denseness.  While such a study
is well-motivated from a purely mathematical perspective, we want to
point out that there is significant additional interest in these
quantities. In particular, it has recently been realized that the
fractal dimension of the spectrum is intimately related with the
long-time asymptotics of the solution to the associated
time-dependent Schr\"odinger equation, that is, $i \partial_t \phi =
H_{V,\omega} \phi$; see \cite{DEGT}.

Fractal properties of $\Sigma_V$ are by now well understood for large values of
$V$. Work of Casdagli \cite{Cas} and S\"ut\H{o} \cite{S87} shows
that for $V \ge 16$, $\Sigma_V$ is a dynamically defined Cantor
set. It follows from this result that the Hausdorff dimension and
the upper and lower box counting dimension of $\Sigma_V$ all
coincide; let us denote this common value by $\dim
\Sigma_V$. Using this result, Damanik, Embree, Gorodetski, and
Tcheremchantsev have shown upper and lower bounds for the
dimension; see \cite{DEGT}. A particular consequence of these
bounds is the identification of the asymptotic behavior of the
dimension as $V$ tends to infinity:
$$
\lim_{V \to \infty} \dim \Sigma_V \cdot \log V = \log (1 + \sqrt{2}).
$$
The paper \cite{DEGT} also discusses some of the implications for
the dynamics of the Schr\"odinger equation; let us mention
\cite{DT07,DT08} for further recent advances in this direction for
the strongly coupled Fibonacci Hamiltonian.

By contrast, hardly anything about $\Sigma_V$ (beyond it having
 Hausdorff dimension strictly between zero and one) is known for small values of $V$.
The largeness of $V$ enters the proofs of the existing results in
critical ways. Consequently, these proofs indeed break down once
the largeness assumption is dropped. The purpose of this paper is to fill out this gap by completely different methods.

We would like to emphasize that quantitative properties of regular
Cantor sets such as thickness and denseness are widely used in
dynamical systems (see \cite{N79, N70}, \cite{PT}, \cite{Me}) and they have found
an application in number theory (see \cite{As00, As01, As02}, \cite{Cus}, \cite{Ha}, \cite{Hl}), but to the best of
our knowledge, these kinds of techniques have never been used before
in the context of mathematical physics.

\subsection{Statement of the Main Results}

In this subsection we describe our results for small coupling $V$.
Clearly, as $V$ approaches zero, $H_{V,\omega}$ approaches the
free Schr\"odinger operator
$$
[H_0 \psi] (n) = \psi(n+1) + \psi(n-1),
$$
which is a well-studied object whose spectral properties are
completely understood. In particular, the spectrum of $H_0$ is
given by the interval $[-2,2]$. It is natural to ask which
spectral features of $H_{V,\omega}$ approach those of $H_0$. It
follows from S\"ut\H{o}'s 1989 result \cite{S89} that the Lebesgue measure of the
spectrum does not extend continuously to the case $V=0$. Given
this situation, one would at least hope that the dimension of the
spectrum is continuous at $V = 0$.

It was shown by us in \cite{DG09a} (and independently by Cantat \cite{Can}) that $\Sigma_V$ is a dynamically
defined Cantor set for $V>0$ sufficiently small (i.e., the small
coupling counterpart to Casdagli's result at large coupling). A
consequence of this is the equality of Hausdorff dimension and
upper and lower box counting dimensions of $\Sigma_V$ in this
coupling constant regime.  Our first result shows that the dimension of the
spectrum indeed extends continuously to $V=0$.

\begin{Thm}\label{t.1}
We have
$$
\lim_{V \to 0} \dim \Sigma_V = 1.
$$
More precisely, there are constants $C_1, C_2 > 0$ such that
$$
1 - C_1 V \le \dim \Sigma_V \le 1 - C_2 V
$$
for $V > 0$ sufficiently small.
\end{Thm}

We get Theorem~\ref{t.1} as a consequence of a connection between the
Hausdorff dimension of a Cantor set and its denseness and
thickness, along with estimates for the latter quantities. Since
these notions and connections may be less familiar to at least a
part of our intended audience, let us recall the definitions and
some of the main results; an excellent general reference in this
context is \cite{PT}.

Let $C \subset \mathbb{R}$ be a Cantor set and denote by $I$ its
convex hull. Any connected component of $I\backslash C$ is called a
\emph{gap} of $C$. A \emph{presentation} of $C$ is given by an
ordering $\mathcal{U} = \{U_n\}_{n \ge 1}$ of the gaps of $C$. If $u
\in C$ is a boundary point of a gap $U$ of $C$, we denote by $K$ the
connected component of $I\backslash (U_1\cup U_2\cup\ldots \cup
U_n)$ (with $n$ chosen so that $U_n = U$) that contains $u$ and
write
$$
\tau(C, \mathcal{U}, u)=\frac{|K|}{|U|}.
$$

With this notation, the \emph{thickness} $\tau(C)$ and the
\emph{denseness} $\theta(C)$ of $C$ are given by
$$
\tau(C) = \sup_{\mathcal{U}} \inf_{u} \tau(C, \mathcal{U}, u),
\qquad \theta(C) = \inf_{\mathcal{U}} \sup_{u} \tau(C,
\mathcal{U}, u),
$$
and they are related to the Hausdorff dimension of $C$ by the
following inequalities (cf.~\cite[Section~4.2]{PT}),
$$
\frac{\log
2}{\log(2+ \frac{1}{\tau(C)})}\le \dim_\mathrm{H} C \le
\frac{\log 2}{\log(2 + \frac{1}{\theta(C)})}.
$$

Due to these inequalities, Theorem \ref{t.1} is a consequence of the following result:

\begin{Thm}\label{t.2}
We have
$$
\lim_{V \to 0}\tau(\Sigma_V) = \infty.
$$
More precisely, there are constants $C_3, C_4 > 0$ such that
$$
C_3 V^{-1} \le \tau(\Sigma_V)\le \theta(\Sigma_V) \le  C_4 V^{-1}
$$
for $V > 0$ sufficiently small.
\end{Thm}

Bovier and Ghez described in their 1995 paper \cite{BG} the
then-state of the art concerning mathematically rigorous results
for Schr\"odinger operators in $\ell^2(\Z)$ with potentials
generated by primitive substitutions. The Fibonacci Hamiltonian
belongs to this class; more precisely, it is in many ways the most
important example within this class of models. One of the most
spectacular discoveries is that, in this class of models, the
spectrum jumps from being an interval for coupling $V = 0$ to
being a zero-measure Cantor set for coupling $V > 0$. That is, as
the potential is turned on, a dense set of gaps opens immediately
(and the complement of these gaps has zero Lebesgue measure). It
is natural to ask about the size of these gaps, which can in fact
be parametrized by a canonical countable set of gap labels; see
\cite{BBG92}. 
These gap openings were studied
in \cite{B} for a Thue-Morse potential and in \cite{BBG91} for period doubling potential. However, for the important Fibonacci
case, the problem remained open. In fact, Bovier and Ghez write on
p.~2321 of \cite{BG}: \textit{It is a quite perplexing feature
that even in the simplest case of all, the golden Fibonacci
sequence, the opening of the gaps at small coupling is not known!}\footnote{There is a perturbative approach to this problem for a class of
models that includes the Fibonacci Hamiltonian by Sire and
Mosseri; see \cite{SM89} and \cite{OK, Si89, SM90, SMS} for
related work. While their work is non-rigorous, it gives quite
convincing arguments in favor of linear gap opening; see
especially \cite[Section~5]{SM89}. It would be interesting to make
their approach mathematically rigorous.}

Our next result resolves this issue completely and shows that, in
the Fibonacci case, all gaps open linearly:

\begin{Thm}\label{t.3}
For $V > 0$ sufficiently small, the boundary points of a gap in
the spectrum $\Sigma_V$ depend smoothly on the coupling constant
$V$. Moreover, given any one-parameter continuous family $\{U_V\}_{V >
0}$ of gaps of $\Sigma_V$, we have that
$$
\lim_{V\to 0} \frac{|U_V|}{|V|}
$$
exists and belongs to $(0,\infty)$.
\end{Thm}

Figure~\ref{fig:specplot20} shows a plot of the spectrum for small coupling:

\begin{figure}[htp]
\includegraphics[width=1\textwidth]{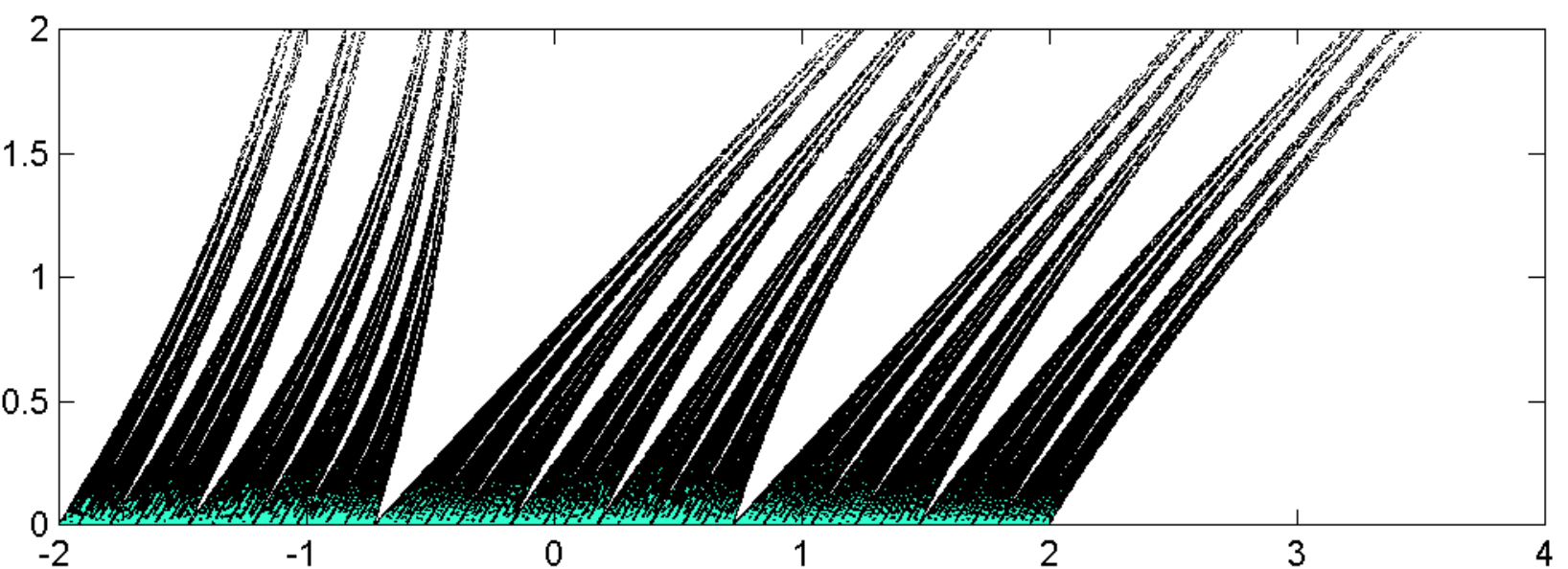} \par
\caption{The set $\{ (E,V) : E \in \Sigma_V , \; 0 \le V \le 2 \}$.} \label{fig:specplot20}
\end{figure}

The plot illustrates the results contained in Theorems~\ref{t.1}--\ref{t.3}. It also suggests that the limit
$\lim_{V\to 0}\frac{|U_V|}{|V|}$ depends on the chosen
family of gaps. We have more to say about the value of the limit
in Theorem~\ref{t.gaplimit} below. It will turn out that its size
is related to the label assigned to it by the gap labeling
theorem.

Our next result concerns the sum set
$$
\Sigma_V + \Sigma_V = \{ E_1 + E_2 : E_1 , E_2 \in \Sigma_V\}.
$$
This set is equal to the spectrum of the so-called square Fibonacci Hamiltonian. Here, one
considers the Schr\"odinger operator
\begin{align*}
[H^{(2)}_V \psi] (m,n) = & \psi(m+1,n) + \psi(m-1,n) + \psi(m,n+1)
+ \psi(m,n-1) + \\
& + V \left( \chi_{[1-\alpha , 1)}(m\alpha \!\!\! \mod 1) +
\chi_{[1-\alpha , 1)}(n\alpha \!\!\! \mod 1) \right) \psi(m,n)
\end{align*}
in $\ell^2(\Z^2)$. The theory of tensor products of Hilbert spaces
and operators then implies that $\sigma(H^{(2)}_V) = \Sigma_V +
\Sigma_V$, see Section \ref{s.chdm}. This operator and its spectrum have been studied
numerically and heuristically by Even-Dar Mandel and Lifshitz in a
series of papers \cite{EL06, EL07, EL08} (a similar model was studied by Sire in \cite{Si89}). Their study suggested that at small coupling, the spectrum of $\Sigma_V + \Sigma_V$ is \text{not} a Cantor set; quite on the contrary, it has no gaps at
all.

Our next theorem confirms this observation:

\begin{Thm}\label{t.4}
For $V > 0$ sufficiently small, we have that $\sigma(H^{(2)}_V) =
\Sigma_V + \Sigma_V$ is an interval.
\end{Thm}

Certainly, the same statement holds for the cubic Fibonacci Hamiltonian (i.e., the analogously defined Schr\"odinger operator in $\ell^2(\Z^3)$ with spectrum $\Sigma_V + \Sigma_V+ \Sigma_V$).

Notice that Theorem \ref{t.4} is a consequence of Theorem
\ref{t.2} and the famous Gap Lemma, which was used by Newhouse to
construct persistent tangencies and generic diffeomorphisms with an
infinite number of attractors (the so-called ``Newhouse
phenomenon"), see Subsection \ref{ss.gaplemma} for details:

\begin{GL}[Newhouse \cite{N79, N70}]\label{t.gaplemma}
If $C_1, C_2\subset \mathbb{R}^1$ are Cantor sets such that
$$
\tau(C_1) \cdot \tau(C_2)>1,
$$
then either one of these sets is contained entirely in a
gap\footnote{For the purpose of this lemma, we also consider the
two unbounded gaps in addition to the bounded gaps considered
above, that is, the connected components of the complement of the
convex hull of the Cantor set in question.} of the other set, or
$C_1\cap C_2\ne \emptyset$.
\end{GL}

Let us turn to the formulation of results involving the integrated
density of states, which is a quantity of fundamental importance
associated with an ergodic family of Schr\"odinger operators. We
first recall the definition of the integrated density of states.
Denote the restriction of $H_{V , \omega}$ to some finite interval
$\Lambda \subset \Z$ with Dirichlet boundary conditions by $H_{V ,
\omega}^\Lambda$. We denote by $N(E,\omega,V,\Lambda)$ the number
of eigenvalues of $H_{V , \omega}^\Lambda$ that are less than or
equal $E$. The integrated density of states is given by
\begin{equation}\label{e.idsdef}
N(E,V) = \lim_{n \to \infty} \frac{1}{n} N(E, \omega , V , [1,n] ).
\end{equation}
We will comment on the existence of the limit and some of its
basic properties in Section~\ref{s.ids}. One of the most important
applications of the integrated density of states is the so-called
gap labeling. That is, one can identify a canonical set of gap
labels, that is only associated with the underlying dynamics (in
this case, an irrational rotation of the circle or the
shift-transformation on a substitution-generated subshift over two
symbols), in such a way that the value of $N(E,V)$ for $E \in \R
\setminus \Sigma_V$ must belong to this canonical set. In the
Fibonacci case, this set is well-known (see, e.g.,
\cite[Eq.~(6.7)]{BBG92}) and the general gap labeling theorem
specializes to the following statement:
\begin{equation}\label{f.fibgaplabels}
\{ N(E,V) : E \in \R \setminus \Sigma_V \} \subseteq \{ \{ m \alpha \} : m \in \Z \} \cup \{ 1 \}
\end{equation}
for every $V \not= 0$. Here $\{ m \alpha \}$ denotes the fractional part
of $m \alpha$, that is, $\{ m \alpha \} = m \alpha - \lfloor m \alpha \rfloor$.
Notice that the set of gap labels is indeed $V$-independent and
only depends on the value of $\alpha$ from the underlying circle
rotation. Since $\alpha$ is irrational, the set of gap labels is
dense. In general, a dense set of gap labels is indicative of a
Cantor spectrum and hence a common (and attractive) stronger
version of proving Cantor spectrum is to show that the operator
``has all its gaps open.'' For example, the Ten Martini Problem
for the almost Mathieu operator is to show Cantor spectrum, while
the Dry Ten Martini Problem is to show that all labels correspond
to gaps in the spectrum. The former problem has been completely
solved \cite{AJ}, while the latter has not yet been completely
settled. Indeed, it is in general a hard problem to show that all
labels given by the gap labeling theorem correspond to gaps and
there are only few results of this kind. Here we show the stronger
(or ``dry'') form of Cantor spectrum for the weakly coupled
Fibonacci Hamiltonian and establish complete gap labeling:

\begin{Thm}\label{t.completegaplabeling}
There is $V_0 > 0$ such that for every $V \in (0,V_0]$, all gaps allowed by the gap labeling theorem are open. That is,
\begin{equation}\label{f.completelabeling}
\{ N(E,V) : E \in \R \setminus \Sigma_V \} = \{ \{ m \alpha \} : m \in \Z \} \cup \{ 1 \}.
\end{equation}
\end{Thm}

Complete gap labeling for the strongly coupled Fibonacci
Hamiltonian was shown by Raymond in \cite{Ra}, where he proves
\eqref{f.completelabeling} for $V > 4$. We conjecture that
\eqref{f.completelabeling} holds for every $V > 0$.

Let us return to the existence of the limit in Theorem~\ref{t.3}.
As was pointed out there, the value of the limit will depend on
the family of gaps chosen. Now that the gap labeling has been
introduced, we can refine the statement. For $m \in \Z \setminus
\{ 0 \}$, denote by $U_m(V)$ the gap of $\Sigma_V$ where the
integrated density of states takes the value $\{ m \alpha \}$.

\begin{Thm}\label{t.gaplimit}
There is a finite constant $C^*$ such that for every $m \in \Z \setminus \{ 0 \}$, we have
$$
\lim_{V \to 0} \frac{|U_m(V)|}{|V|} = \frac{C_m}{|m|}
$$
for a suitable $C_m \in (0, C^*)$.
\end{Thm}

Our final set of results concerns the spectral measures and
transport exponents associated with the operator family. We will
give precise definitions and statements of our results in Section \ref{s.smte} and
limit ourselves to a brief description here. The ultimate goal of
any analysis of a given Schr\"odinger operator is always an
understanding of the associated unitary group, which then allows
one to understand the dynamics of the associated time-dependent
Schr\"odinger equation. The standard transport exponents capture
the spreading of the quantum state in space. Most approaches to a
study of these transport exponents proceed via (time-independent)
spectral theory and link continuity properties of the spectral
measure, associated to the initial state of the time evolution via
the spectral theorem, to lower bounds for the transport exponents.
In one space dimension, these continuity properties can in turn be
investigated via an analysis of the solutions of the
time-independent Schr\"odinger equation. Our goal is to carry this
out for the weakly coupled Fibonacci Hamiltonian. Indeed, results
of this kind are known, but the dependence of the quantities
entering the estimates on the coupling constant had not been
optimized. We revisit these approaches here and improve them to
yield the best possible quantitative estimates at small coupling
that can be obtained with current technology. Our results in Section \ref{s.smte} are
likely not optimal and in particular do not approach the (known)
zero-coupling values. We regard it as an interesting open problem
to either prove or disprove that the dimension estimates for the
spectral measures and the lower bounds for the transport exponents
approach the values in the free (zero coupling) case as the
coupling approaches zero.

Some of the results of this paper were announced in \cite{DG09b}.

\subsection{Overview of the Paper}

Let us outline the remaining parts of this paper. In
Section~\ref{s.prelim} we give some necessary background
information and recall how the trace map arises in the context of
the Fibonacci Hamiltonian and some of its basic properties.
Moreover, since our aim is an understanding of weak coupling
phenomena, we discuss the case of zero coupling.
Section~\ref{s.saas} is the heart of the paper. Here we regard the
weak coupling scenario as a perturbation of zero coupling and
study the dynamics of the trace map and consequences thereof for
the structure of the spectrum as a set. Section~\ref{s.ids}
considers the integrated density of states and proves complete gap
labeling at weak coupling and our quantitative version of the
linear gap opening result. Spectral measures and transport
exponents are studied by means of solution estimates in
Section~\ref{s.smte}. Higher-dimensional models generated by a
product construction are discussed in Section~\ref{s.chdm}, where
we confirm some predictions of Even-Dar Mandel and Lifshitz. Since
their model is based on the off-diagonal Fibonacci Hamiltonian and
the main body of this paper (and most of the other mathematical
works on the Fibonacci Hamiltonian) considers the diagonal
Fibonacci Hamiltonian, we develop all the basic results for the
off-diagonal model in Appendix~\ref{a.odfh} and explain there how
our work indeed confirms the predictions for the original Even-Dar
Mandel-Lifshitz product model.

To assist the reader in locating the proofs of the theorems from
the previous subsection, here is where they may be found: We prove
Theorem~\ref{t.2} (which implies Theorem~\ref{t.1}) in Subsection~\ref{ss.egs}, Theorem~\ref{t.3} in Subsection~\ref{ss.gapsizes}, Theorem~\ref{t.4} in Subsection~\ref{ss.sfh}, Theorem~\ref{t.completegaplabeling} in Subsection~\ref{ss.cgl}, and finally Theorem~\ref{t.gaplimit} in Subsection~\ref{ss.agl}.

\bigskip

\noindent\textit{Acknowledgment.} We are grateful to Mark Embree for generating the plots shown in Figures~\ref{f.specplot} and \ref{f.specplot2}.

\section{Preliminaries}\label{s.prelim}

\subsection{Description of the Trace Map and Previous Results}

The main tool that we are using here is the so-called
\textit{trace map}. It was originally introduced in \cite{Ka,KKT};
further useful references include \cite{BGJ,BR,HM,Ro}. Let us
quickly recall how it arises from the substitution invariance of
the Fibonacci potential; see \cite{S87} for detailed proofs of
some of the statements below.

The one step transfer matrices associated with the difference
equation $H_{V,\omega} u = E u$ are given by
$$
T_{V,\omega}(m,E) = \begin{pmatrix} E - V \chi_{[1-\alpha,1)}(m
\alpha + \omega \!\!\!\! \mod 1) & -1 \\
1 & 0 \end{pmatrix}.
$$
Denote the Fibonacci numbers by $\{F_k\}$, that is, $F_0 = F_1 =
1$ and $F_{k+1} = F_k + F_{k-1}$ for $k \ge 1$. Then, one can show
that the matrices
$$
M_{-1}(E) = \begin{pmatrix} 1 & -V \\ 0 & 1
\end{pmatrix} , \quad M_0(E) =
\begin{pmatrix} E & -1 \\ 1 & 0 \end{pmatrix}
$$
and
$$
M_k(E) = T_{V,0}(F_k,E) \times \cdots \times T_{V,0}(1,E) \quad
\text{ for } k \ge 1
$$
obey the recursive relations
$$
M_{k+1}(E) = M_{k-1}(E) M_k(E)
$$
for $k \ge 0$. Passing to the variables
$$
x_k(E) = \frac12 \mathrm{Tr} M_k(E),
$$
this in turn implies
$$
x_{k+1}(E) = 2 x_k(E) x_{k-1}(E) - x_{k-2}(E).
$$
These recursion relations exhibit a conserved quantity; namely, we
have
$$
x_{k+1}(E)^2+x_k(E)^2+x_{k-1}(E)^2 - 2 x_{k+1}(E) x_k(E)
x_{k-1}(E) -1 = \frac{V^2}{4}
$$
for every $k \ge 0$.

Given these observations, it is then convenient to introduce the
\textit{trace map}
$$
T : \Bbb{R}^3 \to \Bbb{R}^3, \; T(x,y,z)=(2xy-z,x,y).
$$
The following function\footnote{The function
$G(x,y,z)$ is called the \textit{Fricke character}, or sometimes the \textit{Fricke-Vogt invariant}.}
$$
G(x,y,z) = x^2+y^2+z^2-2xyz-1
$$
is invariant under the action of $T$, and hence $T$
preserves the family of cubic surfaces\footnote{The surface $S_0$
is called the \textit{Cayley cubic}.}
$$
S_V = \left\{(x,y,z)\in \Bbb{R}^3 : x^2+y^2+z^2-2xyz=1+
\frac{V^2}{4} \right\}.
$$
Plots of the surfaces $S_{0.01}$ and $S_{0.5}$ are given in Figures~\ref{fig:s0.01} and \ref{fig:s0.5}, respectively.

It is of course natural to consider the restriction $T_{V}$ of the
trace map $T$ to the invariant surface $S_V$. That is, $T_{V}:S_V
\to S_V$, $T_{V}=T|_{S_V}$. Denote by $\Omega_{V}$ the set of
points in $S_V$ whose full orbits under $T_{V}$ are bounded. {\it
A priori} the set of bounded orbits of $T_V$ could be different
from the non-wandering set\footnote{A point $p\in M$ of a
diffeomorphism $f:M\to M$ is wandering if there exists a
neighborhood $O(p)\subset M$ such that $f^k(O)\cap O = \emptyset$
for any $k\in \mathbb{Z}\backslash 0$. The non-wandering set of
$f$ is the set of points that are not wandering.} of $T_V$, but
our construction of the Markov partition and our analysis of the
behavior of $T_V$ near singularities show that here these two sets
do coincide. Notice that this is parallel to the construction of
the symbolic coding in \cite{Cas}.

Let us recall that an invariant closed set $\Lambda$ of a
diffeomorphism $f : M \to M$ is \textit{hyperbolic} if there
exists a splitting of the tangent space $T_xM=E^u_x\oplus E^u_x$ at
every point $x\in \Lambda$ such that this splitting is invariant
under $Df$, the differential $Df$ exponentially contracts vectors
from the stable subspaces $\{E^s_x\}$, and the differential of the
inverse, $Df^{-1}$, exponentially contracts vectors from the
unstable subspaces $\{E^u_x\}$. A hyperbolic set $\Lambda$ of a
diffeomorphism $f : M \to M$ is \textit{locally maximal} if there
exists a neighborhood $U$ of $\Lambda$ such that
$$
\Lambda=\bigcap_{n\in\Bbb{Z}}f^n(U).
$$

We want to recall the following central result.

\bthm[\cite{Cas}, \cite{DG09a}, \cite{Can}]\label{Casdagli} For $V\ne 0$, the set
$\Omega_{V}$ is a locally maximal hyperbolic set of
$T_{V}:S_V \to S_V$. It is homeomorphic to a Cantor set. \ethm

\begin{figure}[htp]\begin{minipage}{5cm} \includegraphics[width=1.2\textwidth]{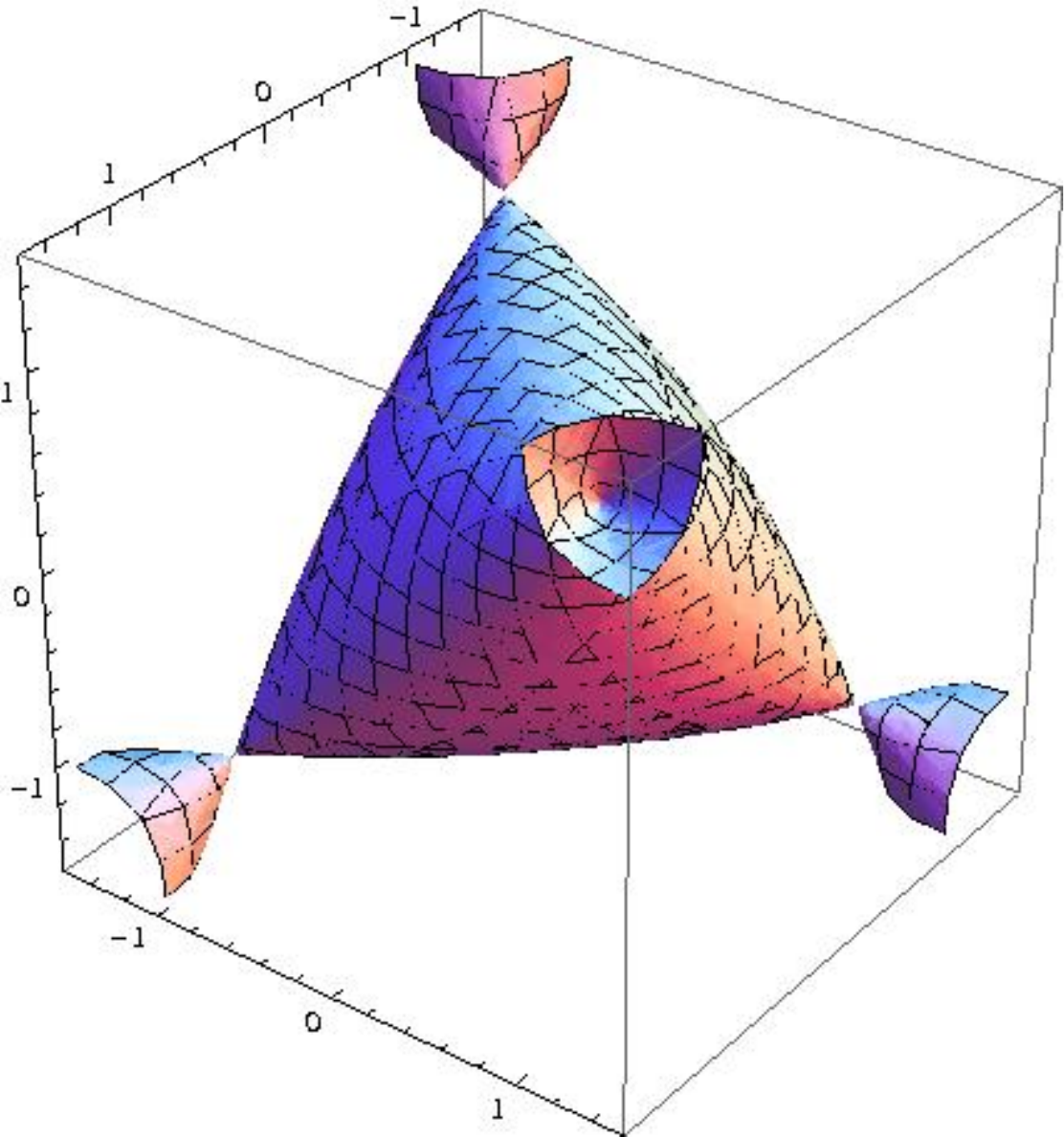} \par
\caption{The surface $S_{0.01}$.} \label{fig:s0.01}
\end{minipage} \hfill \begin{minipage}{5cm} \includegraphics[width=1.2\textwidth]{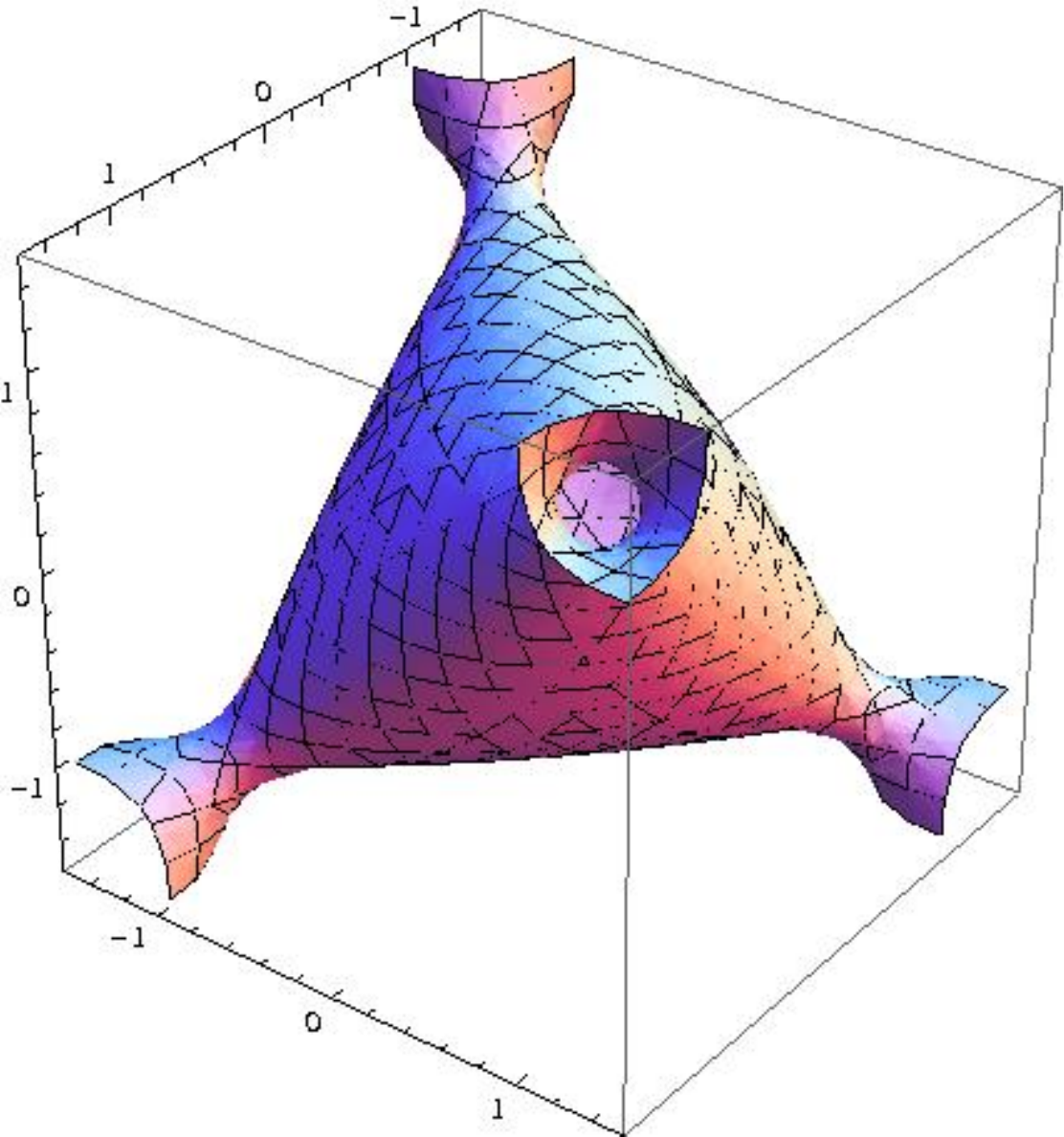}
\par \caption{The surface $S_{0.5}$.} \label{fig:s0.5} \end{minipage}
\end{figure}

Denote by $\ell_V$ the line
$$
\ell_V = \left\{ \left(\frac{E-V}{2}, \frac{E}{2}, 1 \right) : E
\in \Bbb{R} \right\}.
$$
It is easy to check that $\ell_V \subset S_V$.

The second central result about the trace map we wish to recall was proven by S\"ut\H{o} in \cite{S87}.
%
%

\bthm[S\"ut\H{o} 1987]\label{spectrum}  An energy $E$ belongs to
the spectrum of $H_{V,\omega}$ if and only if the positive
semiorbit of the point $(\frac{E-V}{2}, \frac{E}{2}, 1)$ under
iterates of the trace map $T$ is bounded. \ethm

In fact, as also shown by S\"ut\H{o} in \cite{S87}, the trace map can be used to generate canonical approximations of the spectrum, $\Sigma_V$. Namely, consider the following sets:
$$
\Sigma_V^{(n)} = \{ E \in \R : \text{ for } (x_n , y_n , z_n) = T^n(\tfrac{E-V}{2}, \tfrac{E}{2}, 1), \text{ we have } \min\{ |x_n| , |y_n| \} \le 1 \}.
$$
Then, we have $\Sigma_V^{(n)} \supseteq \Sigma_V^{(n+1)} \to \Sigma_V$,
that is,
$$
\Sigma_V = \bigcap_{n \in \Z_+} \Sigma_V^{(n)}.
$$


\begin{figure}[htp]
\includegraphics[width=1\textwidth]{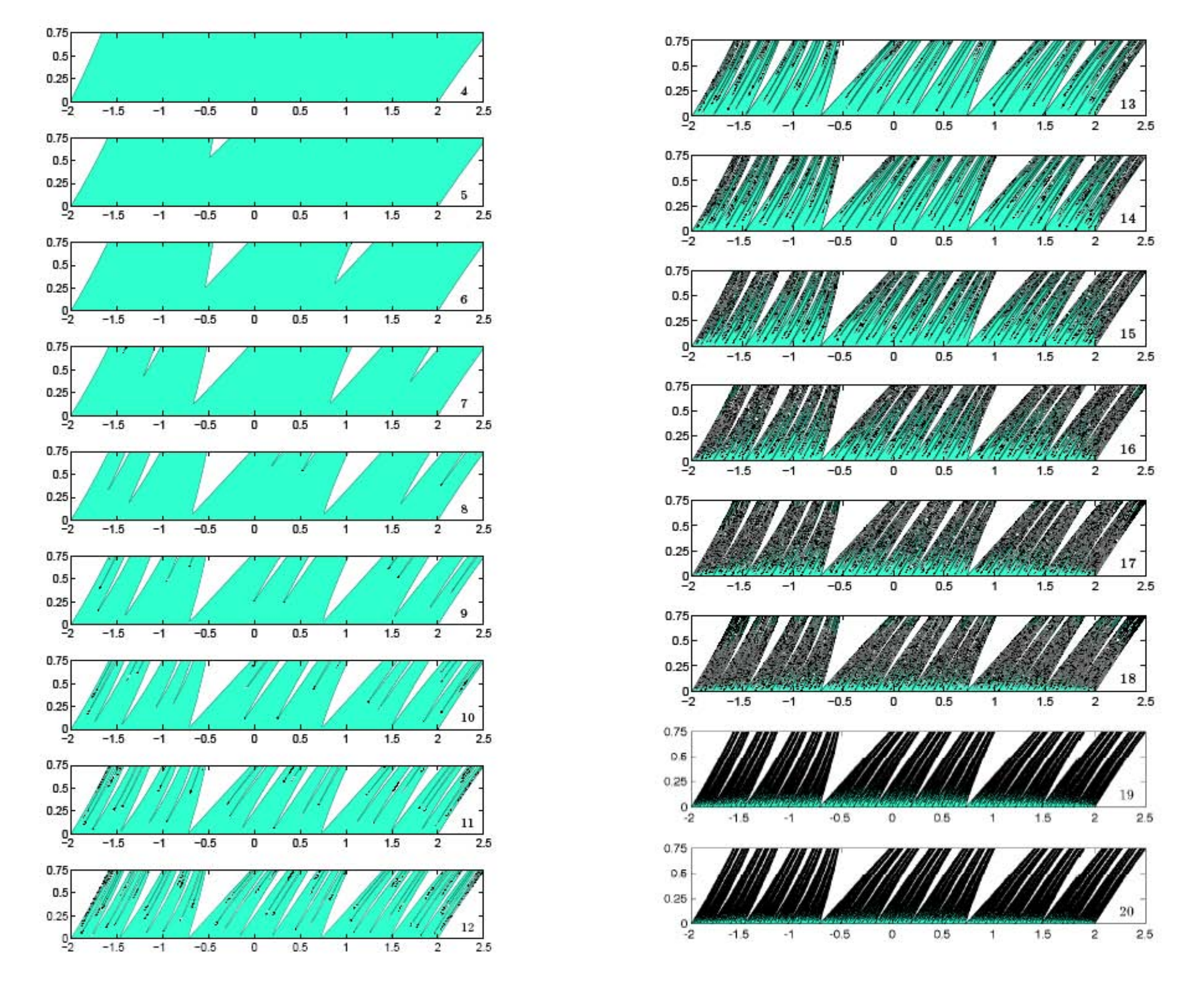} \par
\caption{The sets $\{ (E,V) : E \in \Sigma_V^{(n)} , \; 0 \le V \le \tfrac34 \}$ through $n = 20$.} \label{f.specplot}
\end{figure}

Figure~\ref{f.specplot} shows a plot of the sets $\{ (E,V) : E \in \Sigma_V^{(n)} , \; 0 \le V \le \tfrac34 \}$ for values of $n$ up to $20$.\footnote{This is also how the plot in Figure~\ref{fig:specplot20} was obtained. In fact, what is shown there is the set $\{ (E,V) : E \in \Sigma_V^{(20)} , \; 0 \le V \le 2 \}$.} These plots illustrate nicely both the linear gap opening and the fact that the size of a gap depends on its label; compare Theorems~\ref{t.3} and \ref{t.gaplimit}. To further document linear gap opening through numerics, Figure~\ref{f.specplot2} zooms into a portion of Figure~\ref{f.specplot} near a point $(E,0)$ for an energy $E$ where a gap opens; we have chosen $E \approx 0.7248$.


\begin{figure}[htp]
\includegraphics[width=1\textwidth]{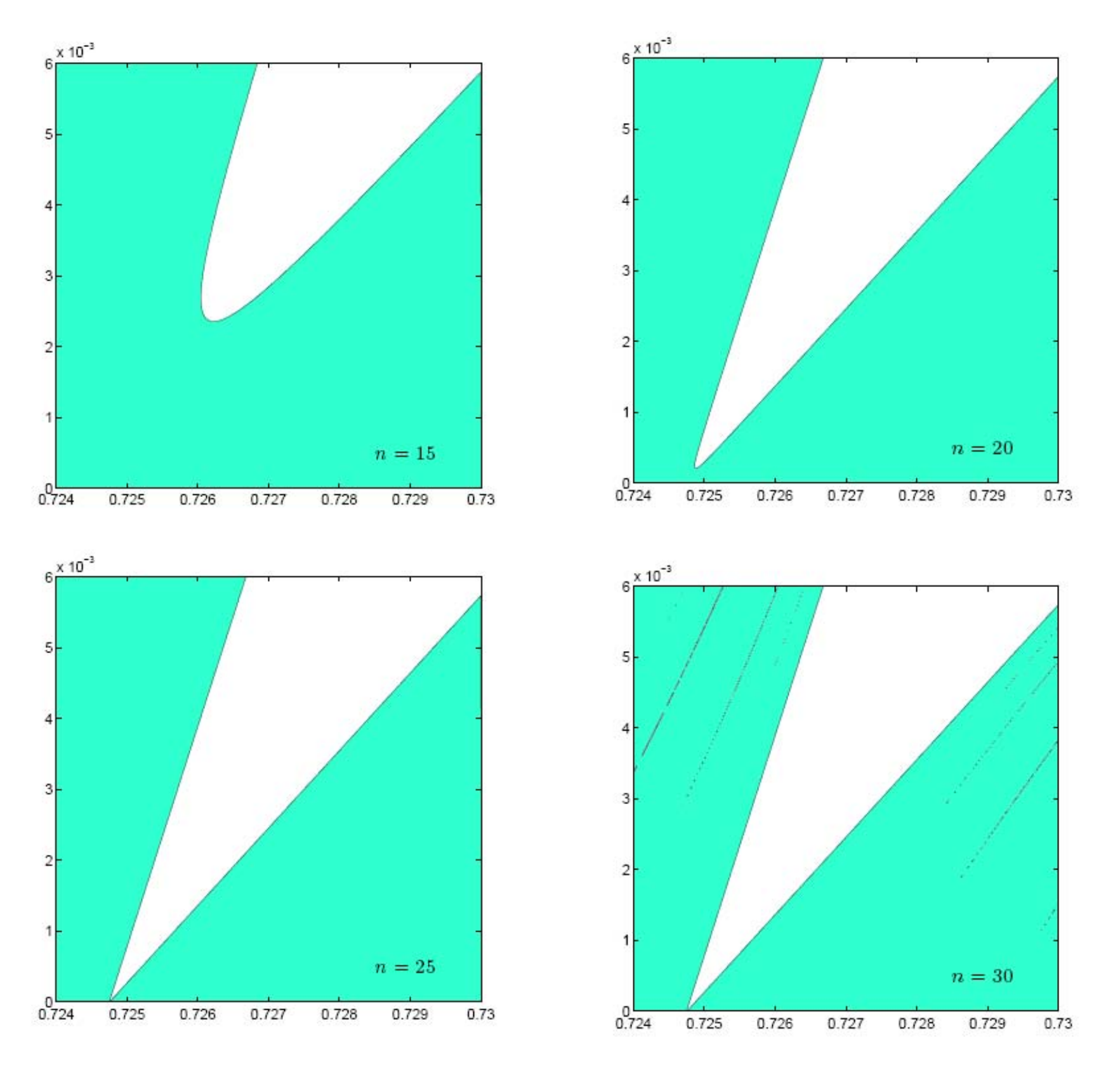} \par
\caption{The sets $\{ (E,V) : E \in \Sigma_V^{(n)} \}$ near $(0.7248,0)$ for $n = 15, 20, 25, 30$.} \label{f.specplot2}
\end{figure}

\subsection{Properties of the Trace Map for $V=0$}\label{s.vequzero}

We will regard the case of small $V$ as a small perturbation of
the case $V = 0$.  This subsection is devoted to the study of this
``unperturbed case.''

Denote by $\mathbb{S}$ the part of the surface $S_0$ inside the
cube $\{ |x|\le 1, |y|\le 1, |z|\le 1\}$. The surface $\mathbb{S}$
is homeomorphic to $S^2$, invariant, smooth everywhere except at
the four points $P_1=(1,1,1)$, $P_2=(-1,-1,1)$, $P_3=(1,-1,-1)$,
and $P_4=(-1,1,-1)$, where $\mathbb{S}$ has conic singularities,
and the trace map $T$ restricted to $\mathbb{S}$ is a factor of the
hyperbolic automorphism of $\T^2 = \R^2 / \Z^2$ given by
$$
\mathcal{A}(\theta, \varphi)=(\theta + \varphi, \theta)\
(\text{\rm mod}\ 1).
$$
The semiconjugacy is given by the map
$$
F: (\theta, \varphi) \mapsto (\cos 2\pi(\theta + \varphi), \cos
2\pi \theta, \cos 2\pi \varphi).
$$
The map $\mathcal{A}$ is hyperbolic, and is given by the matrix
$A=\begin{pmatrix} 1 & 1 \\ 1 & 0 \end{pmatrix}$, which has
eigenvalues
$$
\mu=\frac{1+\sqrt{5}}{2}\ \ \text{\rm and} \ \ \
-\mu^{-1}=\frac{1-\sqrt{5}}{2}.
$$
Let us denote by $\mathbf{v}^u, \mathbf{v}^u \in \mathbb{R}^2 $
the unstable and stable eigenvectors of $A$:
$$
A\mathbf{v}^u=\mu\mathbf{v}^u, \
A\mathbf{v}^s=-\mu^{-1}\mathbf{v}^s, \
\|\mathbf{v}^u\|=\|\mathbf{v}^s\|=1.
$$
Fix some small $\zeta > 0$ and define the stable (resp., unstable)
cone fields on $\mathbb{R}^2$ in the following way:
\begin{align}\label{e.cones}
    K^s_p & =\{\mathbf{v}\in T_p\mathbb{R}^2 :
\mathbf{v}=v^u\mathbf{v}^u+v^s\mathbf{v}^s, \ |v^s| > \zeta^{-1}
|v^u|\},\\
\nonumber    K^u_p & =\{\mathbf{v}\in T_p\mathbb{R}^2 :
\mathbf{v}=v^u\mathbf{v}^u+v^s\mathbf{v}^s, \ |v^u|
> \zeta^{-1} |v^s|\}.
\end{align}
These cone fields are invariant:
$$
\forall\ \mathbf{v}\in K^u_p\ \ \ \ A\mathbf{v}\in K^u_{A(p)},
$$
$$
\forall\ \mathbf{v}\in K^s_p\ \ \ \ A^{-1}\mathbf{v}\in
K^s_{A^{-1}(p)}.
$$
Also, the iterates of the map $A$ expand vectors from the unstable
cones, and the iterates of the map $A^{-1}$ expand vectors from
the stable cones:
$$
\forall\ \mathbf{v}\in K^u_p\ \ \ \ \forall\ n\in \mathbb{N}\ \ \
\ |A^n\mathbf{v}| > \frac{1}{\sqrt{1 + \zeta^2}}
\mu^n|\mathbf{v}|,
$$
$$
\forall\ \mathbf{v}\in K^s_p\ \ \ \ \forall\ n\in \mathbb{N}\ \ \
\ |A^{-n}\mathbf{v}| > \frac{1}{\sqrt{1 + \zeta^2}}
\mu^n|\mathbf{v}|.
$$
The families of cones $\{K^s\}$ and $\{K^u\}$ invariant under
$\mathcal{A}$ can be also considered on $\mathbb{T}^2$.

The differential of the semiconjugacy $F$ sends these cone
families to stable and unstable cone families on
$\mathbb{S}\backslash \{P_1, P_2, P_3, P_4\}$. Let us denote these
images by $\{\mathcal{K}^s\}$ and $\{\mathcal{K}^u\}$.

\blm[Lemma 3.1 from \cite{DG09a}]\label{l.differential} The
differential of the semiconjugacy $DF$ induces a map of the unit
bundle of $\mathbb{T}^2$ to the unit bundle of
$\mathbb{S}\backslash \{P_1, P_2, P_3, P_4\}$. The derivatives of
the restrictions of this map to fibers are uniformly bounded. In
particular, the sizes of the cones in the families
$\{\mathcal{K}^s\}$ and $\{\mathcal{K}^u\}$ are uniformly bounded
away from zero. \elm

Finally, consider the Markov partition for the map
$\mathcal{A}:\mathbb{T}^2\to \mathbb{T}^2$ 
that is shown in
Figure~\ref{fig.Casdagli-Markov} (and which had already appeared
in \cite{Cas}; for more details on Markov partitions for
two-dimensional hyperbolic maps see \cite[Appendix 2]{PT}). Its
image under the map $F:\mathbb{T}^2\to \mathbb{S}$ is a Markov
partition for the pseudo-Anosov map $T:\mathbb{S}\to \mathbb{S}$.

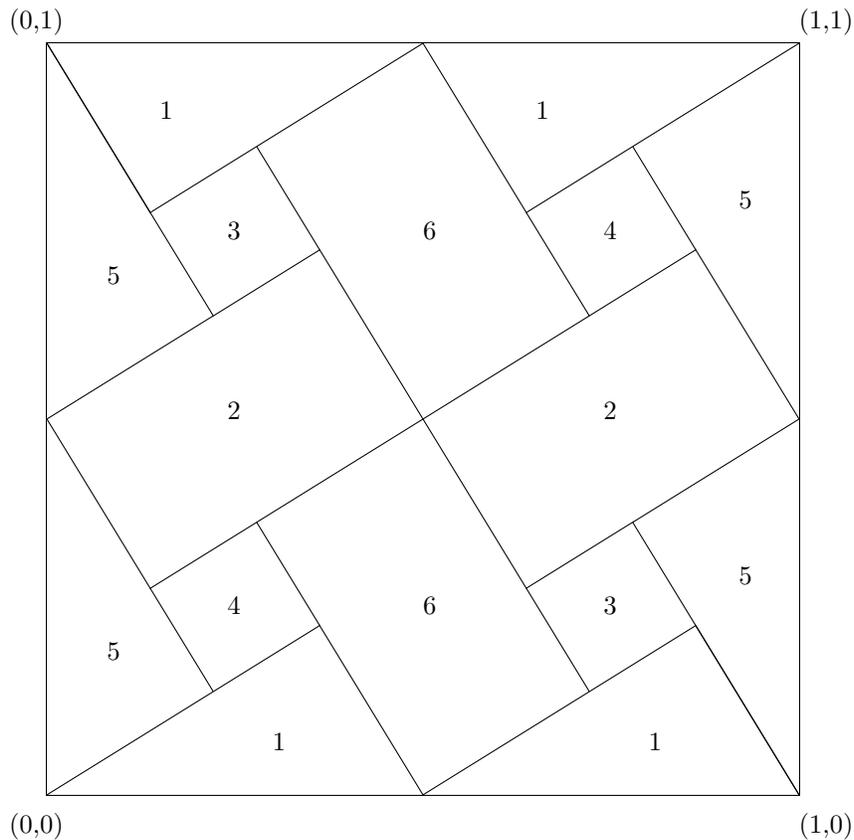
\begin{figure}[htb]
\setlength{\unitlength}{1mm}
\begin{picture}(120,120)

\put(10,10){\framebox(100,100)}

\put(5,5){(0,0)}

\put(110,5){(1,0)}

\put(5,112){(0,1)}

\put(110,112){(1,1)}

\put(10,10){\line(161,100){36.2}}

\put(60,10){\line(-61,100){13.8}}

\put(40,16){$1$}

\put(60,10){\line(161,100){36.2}}

\put(110,10){\line(-61,100){13.8}}

\put(90,16){$1$}

\put(60,110){\line(-161,-100){36.2}}

\put(10,110){\line(61,-100){13.8}}

\put(25,100){$1$}

\put(110,110){\line(-161,-100){36.2}}

\put(60,110){\line(61,-100){13.8}}

\put(75,100){$1$}

\put(10,60){\line(61,-100){22.1}}

\put(18,28){$5$}

\put(10,60){\line(161,100){22.1}}

\put(10,110){\line(61,-100){22.1}}

\put(18,78){$5$}

\put(110,60){\line(-61,100){22.1}}

\put(102,38){$5$}

\put(110,60){\line(-161,-100){22.1}}

\put(110,10){\line(-61,100){22.1}}

\put(102,88){$5$}

\put(60,60){\line(-161,-100){36.3}}

\put(60,60){\line(161,100){36.3}}

\put(60,60){\line(-61,100){22.1}}

\put(60,60){\line(61,-100){22.1}}

\put(88,46.35){\line(-161,-100){14.3}}

\put(32,73.65){\line(161,100){14.3}}

\put(37.88,46.28){\line(61,-100){8.4}}

\put(82.12,73.72){\line(-61,100){8.4}}

\put(34,34){$4$}

\put(84,84){$4$}

\put(34,84){$3$}

\put(84,34){$3$}

\put(34,60){$2$}

\put(84,60){$2$}

\put(60,34){$6$}

\put(60,84){$6$}

\end{picture}
\caption{The Markov partition for the map
$\mathcal{A}$.}\label{fig.Casdagli-Markov}
\end{figure}

\section{The Spectrum as a Set}\label{s.saas}

It is not hard to see that the line $\ell_V$ is transversal to the
stable manifolds of the hyperbolic set $\Omega_V$ for small values
of $V$ (see, e.g., \cite[Lemma~5.5]{DG09a}). Therefore the
intersection of $\ell_V$ and $W^s(\Omega_V)$ (and, hence,
$\Sigma_V$) is a dynamically defined Cantor set (see, for example, \cite{T}). In this section
we study the properties of this one-parameter family of Cantor sets. Namely, in Subsection \ref{ss.gapsizes} we prove that the size of a given gap in the Cantor set tends to zero linearly as the coupling constant (the parameter) tends to zero. In Subsection \ref{s.s04} we use normally hyperbolic theory to introduce a normalizing coordinate system in a neighborhood of a singularity. Then in Subsection \ref{ss.ordering} the order of the gaps that is related to the dynamics is chosen. Roughly speaking, the longer it takes for a gap to leave the union of the elements of the Markov partition, the higher is the order of the gap.  Next, in Subsection \ref{ss.dynnearsing} this normalizing coordinate system is used to study the distortion properties of the transitions through a neighborhood of a singularity. Finally, in the last three subsections we bring all the pieces together and prove the distortion property that immediately implies Theorem \ref{t.2}.

\subsection{Linear Gap Opening as the Potential is Turned On}\label{ss.gapsizes}

Here we prove Theorem \ref{t.3}.

Consider
the dynamics of $T$ in a neighborhood of $P_1=(1,1,1)$. Due to the symmetries of the trace map this will also provide information on the dynamics near the other singularities. Take $r_0>0$ small and let $O_{r_0}(P_1)$ be an $r_0$-neighborhood
of the point $P_1=(1, 1, 1)$ in
$\mathbb{R}^3$. Let us consider the set $Per_2(T)$ of periodic
points of $T$ of period 2; compare Figures~\ref{fig:sing0.0} and \ref{fig:sing0.2}.

\begin{figure}[htb] \begin{minipage}{5cm} \includegraphics[width=1.2\textwidth]{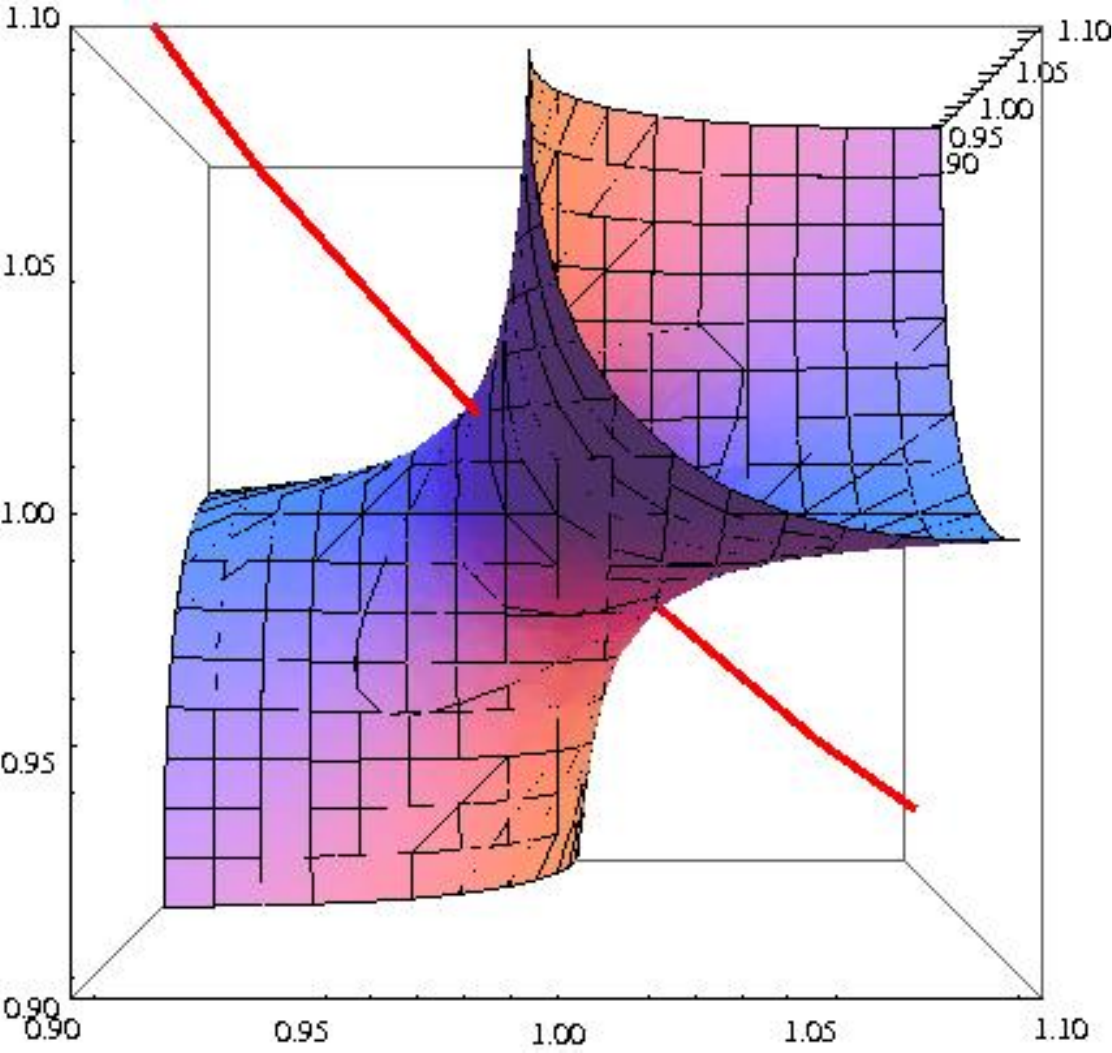} \par
\caption{$S_{0.1}$ and $Per_2(T)$ near $(1,1,1)$.}
\label{fig:sing0.0} \end{minipage} \hfill
\begin{minipage}{5cm}
\includegraphics[width=1.2\textwidth]{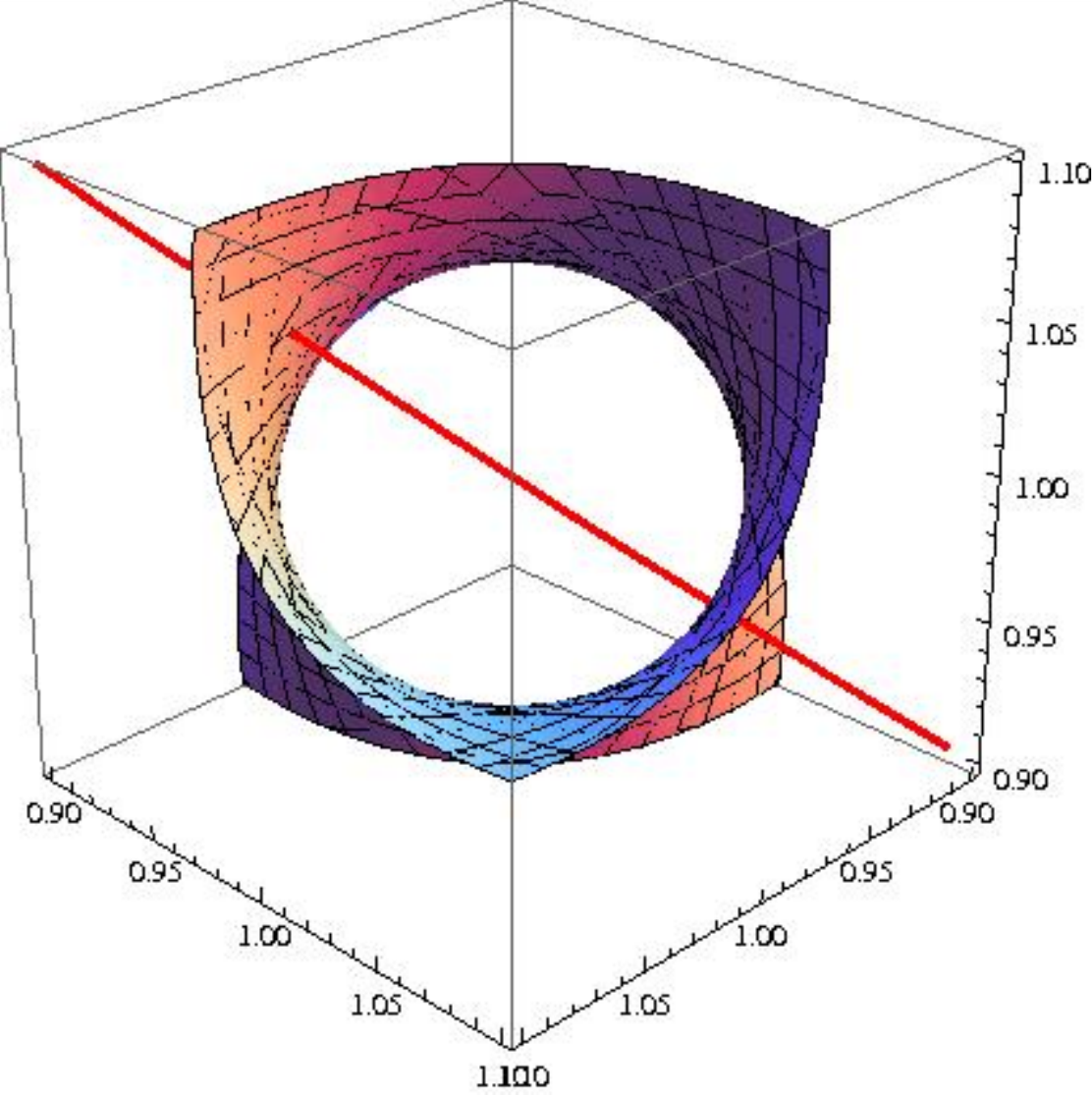}
\par \caption{$S_{0.2}$ and $Per_2(T)$ near $(1,1,1)$.} \label{fig:sing0.2}
\end{minipage} \end{figure}

\blm\label{l.periodtwo} We have
$$
Per_2(T)=\left\{(x,y,z) : x \in (-\infty, \tfrac{1}{2})\cup
(\tfrac{1}{2}, \infty), \ y=\frac{x}{2x-1}, \ z=x\right\}.
$$
\elm
\begin{proof}
Direct calculation.
\end{proof}
Notice that in a neighborhood $U_1$ of $P_1$, the intersection $I\equiv
Per_2(T)\cap U_1$ is a smooth curve that is a normally hyperbolic
with respect to $T$ (see, e.g., Appendix~1 in \cite{PT} for
the formal definition of normal hyperbolicity). Therefore, the
local center-stable manifold  $W_{loc}^{cs}(I)$ and the local
center-unstable manifold $W^{cu}_{loc}(I)$ defined by
$$
W_{loc}^{cs}(I)=\left\{p\in O_{r_0}(P_1) : T^n(p) \in O_{r_0}(P_1) \ \text{\rm for all}\
n\in \mathbb{N}\right\},
$$
$$
W_{loc}^{cu}(I)=\left\{p\in O_{r_0}(P_1) : T^{-n}(p) \in O_{r_0}(P_1) \ \text{\rm for
all}\ n\in \mathbb{N}\right\}
$$
are smooth two-dimensional surfaces. Also, the local strong stable
manifold $W^{ss}_{loc}(P_1)$ and the local strong unstable
manifold $W^{uu}_{loc}(P_1)$ of the fixed point $P_1$, defined by
$$
W^{ss}_{loc}(P_1)=\left\{p\in W^{cs}_{loc}(I) : T^n(p)\to P_1 \
\text{\rm as}\ n\to \infty\right\},
$$
$$
W^{uu}_{loc}(P_1)=\left\{p\in W^{cu}_{loc}(I) : T^{-n}(p)\to P_1 \
\text{\rm as}\ n\to \infty\right\},
$$
are smooth curves.

The Markov partition for the pseudo-Anosov map $T:\mathbb{S}\to \mathbb{S}$ can be extended to a Markov partition for the map $T_V:S_V\to S_V$ for small values of $V$. Namely, there are
four singular points $P_1=(1,1,1)$, $P_2=(-1,-1,1)$, $P_3=(1,-1,-1)$,
and $P_4=(-1,1,-1)$ of $\mathbb{S}$. The point
$P_1$ is a fixed point of $T$, and the points $P_2$, $P_3$, $P_4$
form a periodic orbit of period 3. For small $V$, on the surface
$S_V$ near $P_1$ there is a hyperbolic orbit of the map
$T_V=T|_{S_V}$ of period 2, and near the orbit $\{P_2, P_3, P_4\}$
there is a hyperbolic periodic orbit of period 6. Pieces of stable
and unstable manifolds of these 8 periodic points form a Markov
partition for $T_V:S_V\to S_V$. For $V\ne 0$, the elements of this
Markov partition are disjoint. Let us denote these six rectangles
(the elements of the Markov partition) by $R_V^1, R_V^2, \ldots,
R_V^6$. Let us also denote $R_V=\cup_{i=1}^6R_V^i$.

It is convenient now to consider $T_V^6:S_V\to S_V$ since for
$T_V^6$, each of the eight periodic points that were born from
singularities becomes a fixed point. Due to the symmetries of the
trace map, the dynamics of $T^6$ is the same in a neighborhood of
each of the singularities $P_1, P_2, P_3,$ and $P_4$.
Pick any bounded gap $O$ in $\ell_V\backslash W^s(\Omega_V)$. The boundary points
of $O$ belong to stable manifolds of two fixed points near one of
the singularities $P_1, P_2, P_3,$ or $P_4$. Without loss of
generality assume that those fixed points are in a neighborhood of
$P_1$.

The set of fixed points of $T^6$ in a neighborhood $U_1$ of $P_1$
is a smooth curve $\text{\rm Fix}(T^6, O_{r_0}(P_1))=Per_2(T)\cap O_{r_0}(P_1)$; see Lemma \ref{l.periodtwo} above.  
Each of the fixed points has one of the eigenvalues equal to 1,
one greater than 1, and one smaller than 1 in absolute value.
Therefore the curve $\text{\rm Fix}(T^6, O_{r_0}(P_1))$ is a normally
hyperbolic manifold, and its stable set $W^s(\text{\rm Fix}(T^6,
O_{r_0}(P_1)))$ is a smooth two dimensional surface; see \cite{HPS}. The strong
stable manifolds form a $C^1$-foliation of $W^s(\text{\rm
Fix}(T^6, O_{r_0}(P_1)))$; see \cite[Theorem B]{PSW}.

If ${\bf{p}}_V$ and ${\bf{q}}_V$ are two fixed points of $T^6_V$
in $O_{r_0}(P_1)$, then these points form the intersection of the curve
$\text{\rm Fix}(T^6, O_{r_0}(P_1))$ with $S_V$ and can be found from the
system
$$
\left\{
  \begin{array}{ll}
   y=\frac{x}{2x-1} \\
   z=x \\
   x^2+y^2+z^2-2xyz=1+\frac{V^2}{4}.
  \end{array}
\right.
$$
If we parameterize  $\text{\rm Fix}(T^6, O_{r_0}(P_1))$ as $\{x=t+1, z=t+1,
y=\frac{t+1}{2t+1}\}$, we get
$$
(t+1)^2+\frac{t+1}{2t+1}+(t+1)^2-2(t+1)^2\frac{t+1}{2t+1}=1+\frac{V^2}{4},
$$
or
$$
\frac{4t^4+10t^3+9t^2+4t+1}{(2t+1)^2}=1+\frac{V^2}{4}.
$$
Since for the function
$f(t)=\frac{4t^4+10t^3+9t^2+4t+1}{(2t+1)^2}$, we have $f(0)=1,
f'(0)=0, f''(0)>0$, the distance between the points ${\bf{p}}_V$ and
${\bf{q}}_V$ is of order $|V|$ for small values of $V$.

\blm\label{l.limit} Let $W\subset \mathbb{R}^3$ be a smooth
surface with a $C^1$-foliation on it. Let $\xi\subset W$ be a smooth
curve transversal to the foliation. Fix a leaf $L\subset W$ of the
foliation, and denote $P=L\cap \xi$. Take a point $Q\ne P$, $Q\in
L$, and a line $\ell_0\subset \mathbb{R}^3$, $Q\in \ell_0$, tangent to
$W$ at $Q$, but not tangent to the leaf $L$. Suppose that a family
of lines $\{\ell_V\}_{V\in (0, V_0)}$ is given such that $\ell_V
\to \ell_0$ as $V \to 0$, each line $\ell_V$, $V>0$,  intersects $W$
at two points $p_V$ and $q_V$, and $p_V\to Q$, $q_V\to Q$ as $V\to
0$.

Denote by $L_{p_V}$ and $L_{q_V}$ the leaves of the foliation that
contain $p_V$ and $q_V$, respectively. Denote
${\bf{p}}_V=L_{p_V}\cap \xi$ and ${\bf{q}}_V=L_{q_V}\cap \xi$.
Then there exists a finite non-zero limit
 $$
 \lim_{V\to 0}\frac{\mathrm{dist}(p_V, q_V)}{\mathrm{dist}({\bf{p}}_V, {\bf{q}}_V)}.
 $$
\elm
\begin{proof}
Since $\ell_0$ is not tangent to the leaf $L$, there exists a
curve $\widetilde{\xi}\subset W$ tangent to $\ell_0$, transversal
to the foliation, and such that $Q\in \widetilde{\xi}$.

Set $\widetilde{\mathbf{p}}=L_{p_V}\cap\widetilde{\xi},
\widetilde{\mathbf{q}}=L_{q_V}\cap \widetilde{\xi}$. Since the
foliation is $C^1$, there exists a finite non-zero limit
\beq\label{e.limit1} \lim_{V\to 0}\frac{\text{\rm
dist}({\bf{p}}_V, {\bf{q}}_V)}{\text{\rm
dist}(\widetilde{{\bf{p}}}_V, \widetilde{{\bf{q}}}_V)}\ne 0. \eneq
Let us consider a plane $\Pi$ tangent to $W$ at $Q$, and let
$\pi:W\to \Pi$ be an orthogonal projection (well defined and
smooth in a neighborhood of $Q$). It is clear that
\beq\label{e.limit2} \lim_{V\to 0}\frac{\text{\rm
dist}(\pi(\widetilde{{\bf{p}}}_V),
\pi(\widetilde{{\bf{q}}}_V))}{\text{\rm
dist}(\widetilde{{\bf{p}}}_V, \widetilde{{\bf{q}}}_V)}=1. \eneq
Also, since $\ell_V\to \ell_0$ as $V\to 0$, we have
\beq\label{e.limit3} \lim_{V\to 0}\frac{\text{\rm
dist}(\pi({{{p}}}_V), \pi({{{q}}}_V))}{\text{\rm dist}({{{p}}}_V,
{{{q}}}_V)}=1. \eneq Finally, since $\pi$ sends the
$C^1$-foliation of $W$ to a $C^1$-foliation on $\Pi$, the
projection along this foliation from $\pi(\ell_V)$ to $\ell_0$ is
$C^1$-close  to isometry. In particular, \beq\label{e.limit4}
\lim_{V\to 0}\frac{\text{\rm dist}(\pi({{{p}}}_V),
\pi({{{q}}}_V))}{\text{\rm dist}(\pi(\widetilde{{\bf{p}}}_V),
\pi(\widetilde{{\bf{q}}}_V))}=1. \eneq The statement of Lemma
\ref{l.limit} follows now from (\ref{e.limit1})--(\ref{e.limit4}).
\end{proof}

\begin{proof}[Proof of Theorem~\ref{t.3}]
Apply Lemma~\ref{l.limit} to the surface $W=W^s(\text{\rm
Fix}(T^6, O_{r_0}(P_1)))$.
\end{proof}

\subsection{Choice of a Coordinate System in a Neighborhood of a Singular
Point}\label{s.s04}

 Due to the smoothness of the
invariant manifolds of the curve of periodic points of period two
described in Section \ref{ss.gapsizes}, there exists a smooth
change of coordinates $\Phi:O_{r_0}(P_1)\to \mathbb{R}^3$ such
that $\Phi(P_1)=(0,0,0)$ and
\begin{itemize}
\item
$\Phi(I)$ is a part of the line $\{x=0, z=0\}$;
\item
$\Phi(W^{cs}_{loc}(I))$ is a part of the plane $\{z=0\}$;
\item
$\Phi(W^{cu}_{loc}(I))$ is a part of the plane $\{x=0\}$;
\item
$\Phi(W^{ss}_{loc}(P_1))$ is a part of the line $\{y=0, z=0\}$;
\item
$\Phi(W^{uu}_{loc}(P_1))$ is a part of the line $\{x=0, y=0\}$.
\end{itemize}
Denote $f=\Phi\circ T\circ \Phi^{-1}$.

In this case,
$$
A\equiv Df(0,0,0)=D(\Phi\circ T\circ \Phi^{-1})(0,0,0)=\begin{pmatrix}
                                       \lambda^{-1} & 0 & 0 \\
                                       0 & -1 & 0 \\
                                       0 & 0 & \lambda \\
                                     \end{pmatrix},
$$
where $\lambda$ is the largest eigenvalue of the differential
$DT(P_1):T_{P_1}\mathbb{R}^3\to T_{P_1}\mathbb{R}^3$,
$$
DT(P_1)=\begin{pmatrix}
                                       2 & 2 & -1 \\
                                       1 & 0 & 0 \\
                                       0 & 1 & 0 \\
                                     \end{pmatrix}, \ \ \ \lambda=\frac{3+\sqrt{5}}{2}=\mu^2.
$$
Let us denote $\frak{S}_V=\Phi(S_V)$. Then, away from $(0,0,0)$,
the family $\{\frak{S}_V\}$ is a smooth family of surfaces,
$\frak{S}_0$ is diffeomorphic to a cone, contains the lines
$\{y=0, z=0\}$ and $\{x=0, y=0\}$, and at each non-zero point on
these lines, it has a quadratic tangency with a horizontal or
vertical plane.

Due to the symmetries of the trace map, similar changes of
coordinates exist in a neighborhood of each of the other
singularities. Denote $O_{r_0}=O_{r_0}(P_1)\cup O_{r_0}(P_2)\cup
O_{r_0}(P_3)\cup O_{r_0}(P_4)$.

Fix a small constant $\mathbf{C} > 0$ and introduce the following
cone fields in $\mathbb{R}^3$: \beq\label{e.coneformula2}
\mathbf{K}_p^u=\{\mathbf{v}\in T_p\mathbb{R}^3, \
\mathbf{v}=\mathbf{v}_{xy}+\mathbf{v}_z : |\mathbf{v}_z|>
\mathbf{C} \sqrt{|z_p|}|\mathbf{v}_{xy}|\}, \eneq
\beq\label{e.cuconeformula}
\widetilde{\mathbf{K}}_p^{u}=\{\mathbf{v}\in T_p\mathbb{R}^3, \
\mathbf{v}=\mathbf{v}_{x}+\mathbf{v}_{yz} : |\mathbf{v}_z|>
\mathbf{C}^{-1}|\mathbf{v}_{xy}| \}, \eneq
\beq\label{e.coneformulas} \mathbf{K}_p^s=\{\mathbf{v}\in
T_p\mathbb{R}^3, \ \mathbf{v}=\mathbf{v}_{x}+\mathbf{v}_{yz} :
|\mathbf{v}_x|> \mathbf{C} \sqrt{|x_p|}|\mathbf{v}_{yz}|\}, \eneq
\beq\label{e.csconeformula}
\widetilde{\mathbf{K}}_p^{s}=\{\mathbf{v}\in T_p\mathbb{R}^3, \
\mathbf{v}=\mathbf{v}_{z}+\mathbf{v}_{xy} : |\mathbf{v}_x|>
\mathbf{C}^{-1}|\mathbf{v}_{yz}| \}. \eneq

\blm\label{l.intersection}
There are $r_1\in (0, r_0)$ and  $m_0\in \mathbb{N}$ such that the following holds.

1. $T^{m_0}(\ell_0)\cap O_{r_1}$ is a union of two connected
curves $\gamma_1$ and $\gamma_2$, and $\Phi(\gamma_i), i=1,2,$ is
tangent to the cone field $\widetilde{\mathbf{K}}^{u}$;

2. $F^{-1}(T^{m_0}(\ell_0\cap \mathbb{S}))$ is tangent to the cone
field $K^{u}$ {\rm (}defined by \eqref{e.cones}{\rm )}. \elm

\bdef We will call the rectangle ${R}^6$ {\rm (}the element of the
Markov partition{\rm )} {\it the opposite} to singularities $P_1$ and
$P_3$, and we will call the rectangle ${R}^5$ {\it the opposite} to
singularities $P_2$ and $P_4$
\endef

Notice that $T^{m_0}(\ell_0\cap \mathbb{S})$ consists of a curve
that connects $P_1$ with an unstable boundary of $R_6$, a finite
number of curves $\xi_i$ that connect two unstable sides of an
element of the Markov partition and such that $F^{-1}(\xi_i)$ is
tangent to the cone field $K^{u}$, and a curve that connects some
other singularity with an unstable boundary of its opposite
rectangle.

Now let us take $r_2\in (0, r_1)$ so small that $\cup_{i=-3}^{3}T^i(O_{r_2})\subset O_{r_1}$.

For small $V$, denote by $\mathbb{S}_{V, O_{r_2}}$ the bounded component
of $S_V\backslash O_{r_2}$. The family $\{\mathbb{S}_{V, O_{r_2}}\}_{V\in
[0,V_0)}$ of surfaces with boundary depends smoothly on the
parameter and has uniformly bounded curvature. For small $V$, a
projection $\pi_V:\mathbb{S}_{V, O_{r_2}} \to \mathbb{S}$ is defined.
The map $\pi_V$ is smooth, and if $p\in \mathbb{S}$, $q\in
\mathbb{S}_{V, O_{r_2}}$, and $\pi_V(q)=p$, then $T_p \mathbb{S}$ and
$T_q S_V$ are close. Denote by $\mathcal{K}_V^u$ (resp.,
$\mathcal{K}_V^s$) the image of the cone $\mathcal{K}^u$ (resp.,
$\mathcal{K}^s$) under the differential of $\pi_V^{-1}$.

Denote $F_V=\pi^{-1}_V\circ F, \ \ F_V:F^{-1}(\mathbb{S}\backslash
O_{r_2})\to S_V$. Denote also $\Psi_V=\Phi \circ F_V$. Compactness
and Mean Value Theorem type arguments imply the following
statement.

\blm\label{l.inequalities} There are $V_0>0$ and $\widetilde{C}>0$
such that the following holds. Suppose that $a, b\in
\mathbb{T}^2\backslash F^{-1}(O_{r_2})$, $v_a\in T_a\mathbb{T}^2$,
$v_b\in T_b\mathbb{T}^2$. Then the following inequalities hold for
all $V\in [0, V_0]$:
\begin{align*}
\text{\rm dist}(F_V(a),F_V(b)) & \le \widetilde{C} \, \text{\rm
dist}(a,b), \\
\text{\rm dist}(a,b) & \le \widetilde{C} \, \text{\rm
dist}(F_V(a),F_V(b)), \\
\angle(DF_V(v_a), DF_V(v_b)) & \le \widetilde{C} (\angle(v_a,
v_b)+\text{\rm dist}(a,b)), \\
\angle(v_a, v_b) & \le \widetilde{C} (\angle(DF_V(v_a),
DF_V(v_b))+\text{\rm dist}(F_V(a),F_V(b))).
\end{align*}
Moreover, if $\Psi_V(a)$ and $\Psi_V(b)$ are defined, then
\begin{align*}
\text{\rm dist}(\Psi_V(a),\Psi_V(b)) & \le \widetilde{C} \,
\text{\rm
dist}(a,b), \\
\text{\rm dist}(a,b) & \le \widetilde{C} \, \text{\rm
dist}(\Psi_V(a),\Psi_V(b)), \\
\angle(D\Psi_V(v_a), D\Psi_V(v_b)) & \le \widetilde{C}(\angle(v_a,
v_b)+\text{\rm dist}(a,b)), \\
\angle(v_a, v_b) & \le \widetilde{C}(\angle(D\Psi_V(v_a),
D\Psi_V(v_b))+\text{\rm dist}(\Psi_V(a),\Psi_V(b))).
\end{align*}
\elm

Finally, notice that if $\mathbf{C}$ and $V_0$ are taken
sufficiently small, then the cone fields $K^u$ on $\mathbb{T}^2$
and $\mathbf{K}^u, \widetilde{\mathbf{K}}^u$ respect each other in
the following sense. Suppose that $a\in \mathbb{T}^2\backslash
F^{-1}(O_{r_2})$ is such that $\Psi_V(a)$ is defined for $V\in [0,
V_0]$, and $v_a\in K^u_a$. Lemma~\ref{l.differential} implies that
$$
D\Psi_V(v_a)\in \mathbf{K}_{\Psi_V(a)}^u.
$$
On the other hand, if $b\in \mathbb{T}^2\backslash F^{-1}(O_{r_2})$ is such that
$\Psi_V(b)$ is defined and $D\Psi_V(v_b)\in
\widetilde{\mathbf{K}}^u_{\Psi_V(b)}$, then
$$
v_b\in K^u_b.
$$

\subsection{Ordering of the Gaps}\label{ss.ordering}

To estimate the thickness of a Cantor set from below (or the denseness from above), it is enough to consider one particular ordering of its gaps. Here we choose a convenient ordering of gaps in $\ell_V\cap W^s(\Omega_V)$ (which is affine equivalent to $\Sigma_V$).

The trace map $T_V$, $V\ne 0$,  has two periodic points of period
2, denote them by $P_1(V)$ and $P_1'(V)$, and six periodic points
of period 6, denote them by $P_2(V)$, $P_2'(V)$, $P_3(V)$,
$P_3'(V)$, $P_4(V)$, and $P_4'(V)$. In Section \ref{ss.gapsizes}
we showed that the distance between $P_i(V)$ and $P_i'(V)$ is of
order $|V|$.

We can choose the notation (swapping the notation for $P_1(V)$ and
$P_1'(V)$, and/or for $P_2(V)$ and $P_2'(V)$ if necessary) in such
a way that the following lemma holds.

\blm If $V$ is small enough, the line $\ell_V$ contains points
$B_1(V)\in \ell_V\cap W^{ss}(P_1(V))$ and $B_2(V)\in \ell_V\cap
W^{ss}(P_2(V))$ such that every point of the line which is not
between $B_1(V)$ and $B_2(V)$ tends to infinity under iterates of
$T_V$. \elm

Denote by $\mathbf{l}_V$ the closed interval on $\ell_V$ between
the points $B_1(V)$ and $B_2(V)$. It is known that the set of
points on $\mathbf{l}_V$ with bounded positive semiorbits is a
dynamically defined Cantor set; see \cite{DG09a, Can}. We would like
to estimate the thickness of this Cantor set.

\blm\label{l.firststep}
There are $m_0\in \mathbb{N}$, $0< C_1< C_2$ and $V_0>0$ such that for all $V\in [0, V_0]$ the following holds.

1. $T_V^{m_0}(\mathbf{l}_V)\cap O_{r_1}$ is a union of two connected curves $\gamma_1(V)$ and
$\gamma_2(V)$, and $\Phi(\gamma_i(V)), i=1,2,$ is tangent to the
cone field $\widetilde{\mathbf{K}}^{u}$;

2. $F^{-1}_V(T^{m_0}(\mathbf{l}_V)\backslash O_{r_2})$ is tangent
to the cone field $K^{u}$;

3. $T^{m_0}_V(\mathbf{l}_V)$ consists of a curve that connects
$W^{ss}_{loc}(P_1)$ with a stable boundary of $R^6$, a finite number
of curves $\xi_i(V)$, each of which connects two stable sides of
an element of the Markov partition and is such that
$F^{-1}_V(\xi_i)$ is tangent to the cone field $K^{u}$, a curve
that connects $W^{ss}_{loc}(P_i),\ i\in\{2, 3, 4\},$  with a stable boundary of
an opposite rectangle to $P_i$, and some ``gaps'' between the curves
mentioned above. The length of these ``gaps" is between $C_1 V$
and $C_2 V$ for all small enough $V$. \elm

\begin{proof}
The statement holds for $V=0$ (see Lemma \ref{l.intersection}). For $V$ positive but small enough properties 1. and 2. hold by continuity, and property 3. follows from the fact that the distances between finite pieces of strong stable manifolds of points $\mathbf{p}_V$ and $\mathbf{q}_V$ is of order $V$, and these strong stable manifolds form the stable parts of boundary of the Markov partition for $T_V$.
\end{proof}

We will call the preimages (under $T_V^{m_0}$) of the gaps defined in Lemma \ref{l.firststep} {\it the gaps of order 1}.

\bdef\label{d.tangenttounstable}
A smooth curve $\gamma\subset S_V$ is tangent to an unstable cone field if
$F_V^{-1}(\gamma\backslash O_{r_2})$ is tangent to $K^u$, and
$\Phi(\gamma\cap O_{r_1})$ is tangent to ${\mathbf{K}}^u$.
\endef

\bdef A  curve is of {\bf type one} if it is tangent to the
unstable cone field and connects opposite sides of unstable
boundaries of some ${R}^i$.

A  curve is of {\bf type two} if it is tangent to the unstable
cone field and connects a point from $W^{ss}_{loc}(P_i), i\in\{1, 2, 3, 4\}, $
with an unstable boundary of an opposite element of the Markov
partition.
\endef

In this terminology, Lemma~\ref{l.firststep} claims that
$T_V^{m_0}(\mathbf{l}_V)$ consists of two curves of type two, some
curves of type one, and some gaps between them of size of order
$V$.

\blm\label{l.curvetypes} An image of a curve of type one under $T^6$ is a union of a finite number
of curves of type one, and of a finite number of gaps of length
between $C_1V$ and $C_2V$. An image of a curve of type two under $T^6$ is a
union of a curve of type two, a finite number of curves of type
one, and a finite number of gaps of length between $C_1V$ and
$C_2V$. \elm

\begin{proof}
The first part follows from the properties of the Markov partition and the fact that the distance between strong stable (strong unstable) manifolds that form the Markov partition is of order $V$, see Subsection \ref{ss.gapsizes}. An image of a curve tangent to an unstable cone field is a curve tangent to an unstable cone field. Also, $T^6(W^{ss}_{loc}(P_i))\subset  W^{ss}_{loc}(P_i)$. Therefore the image of a curve of type two under $T^6$ is a curve which is close to a finite piece of a strong unstable manifold of $P_i$, so the second part follows.
\end{proof}

Suppose that the gaps of order $k$ have already been defined.
Consider the complement of all gaps of order not greater than $k$
on $\mathbf{l}_V$. It consists of a finite number of closed
intervals. Let $J$ be one of them. Consider the curve
$T_V^{m_0+6(k-1)}(J)$. By construction, it is either a curve of type
one, or of type two. In either case, the image
$T_V^6(T_V^{m_0+6(k-1)}(J))=T_V^{m_0+6k}(J)$ consists of some curves of
type one or two, and some gaps of size  $\sim V$. Let us say
that the preimages of these gaps (under $T_V^{m_0+6k}$) are gaps of
order $k+1$. It is clear that every gap in $\mathbf{l}_V\cap
W^s(\Omega_V)$ has some finite order. Therefore we have ordered
all the gaps.

\subsection{Distortion Property: Estimate of the Gap Sizes}\label{ss.egs}

Let us consider $\mathbf{l}_V$ and some gap $\gamma_G\subset \mathbf{l}_V$ of order
$n$. {\it A bridge} that corresponds to this gap is a connected
component of the complement of the union of all gaps of order $\le
n$ next to the gap. There are two bridges that correspond to the
chosen gap, take one of them, and denote it by $\gamma_B$. Denote
also $\gamma=\gamma_G\cup \gamma_B$. Now let us consider
$\Gamma_G\equiv T_V^{m_0+6(n-1)}(\gamma_G)$ and $\Gamma_B\equiv
T_V^{m_0+6(n-1)}(\gamma_B)$. By definition of the order $n$ of the gap we
know that
$$
C_3V\le \frac{|\Gamma_G|}{|\Gamma_B|}\le C_4V
$$
 for some constants $C_3$ and $C_4$ independent of $V$.

\bprop\label{p.distortion} There is a constant $K>1$ independent of
the choice of the gap and of $V$ such that
$$
K^{-1}\frac{|\gamma_G|}{|\gamma_B|}\le
\frac{|\Gamma_G|}{|\Gamma_B|}\le  K\frac{|\gamma_G|}{|\gamma_B|}.
$$
\enprop

Notice that Theorem \ref{t.2} immediately follows from Proposition
\ref{p.distortion}.

The rest of this section is devoted to the proof of Proposition
\ref{p.distortion}, which is completed in Subsection \ref{ss.proofdist}.

\subsection{Dynamics Near Singularities}\label{ss.dynnearsing}

Here we prove several technical propositions on the properties of
the trace map in the coordinate system constructed in Subsection \ref{s.s04}. The first two propositions are reformulations of
\cite[Proposition~1]{DG09a}. The first one claims that a certain
unstable cone field is invariant.

We will use the variables $(x,y,z)$ for coordinates in
$\mathbb{R}^3$. For a point $p\in \mathbb{R}^3$, we will denote
its coordinates by $(x_p, y_p, z_p)$.

\bprop\label{p.coneinvariance} Given $C_1>0, C_2>0, \lambda>1$,
there exists $\delta_0=\delta_0(C_1, C_2,  \lambda)$ such that for
any $\delta\in (0, \delta_0)$, the following holds.

Let $f:\mathbb{R}^3\to \mathbb{R}^3$ be a $C^2$-diffeomorphism
such that
\begin{itemize}
\item[{\rm (i)}] $\|f\|_{C^2}\le C_1$; \item[{\rm (ii)}] The plane
$\{z=0\}$ is invariant under iterates of $f$; \item[{\rm (iii)}]
$\|Df(p)-A\|<\delta$ for every $p\in \mathbb{R}^3$, where
$$
A=\begin{pmatrix}
    \lambda^{-1} & 0 & 0 \\
    0 & 1 & 0 \\
    0 & 0 & \lambda \\
  \end{pmatrix}
$$
is a constant matrix.
\end{itemize}
Introduce the following cone field in $\mathbb{R}^3$:
\beq\label{e.coneformula} K_p^u=\{\mathbf{v}\in T_p\mathbb{R}^3, \
\mathbf{v}=\mathbf{v}_{xy}+\mathbf{v}_z : |\mathbf{v}_z|\ge C_2
\sqrt{|z_p|}|\mathbf{v}_{xy}|\}. \eneq

Then for any point $p=(x_p, y_p, z_p)$, $|z_p|\le 1$ we have $$
Df(K_p)\subseteq K^u_{f(p)}.
$$
 \enprop

 Notice that the choice of the cone field $K_p^u$ here (in (\ref{e.coneformula})) and below (in (\ref{e.coneformula1})) corresponds to the choice of the cone field $\mathbf{K}_p^u$ in (\ref{e.coneformula2}).

The next proposition establishes expansion of vectors from the
introduced unstable cones under the differential of the map.

\bprop\label{p.expanding} Given $C_1>0, C_2>0, \lambda>1$,
$\varepsilon \in (0,\frac{1}{4})$,  $\eta>0$ there
exists $\delta_0=\delta_0(C_1, C_2, \lambda, \varepsilon)$,
$N_0\in \mathbb{N}, N_0=N_0(C_1, C_2, \lambda, \varepsilon,
\delta_0)$, and $C=C(\eta)>0$ such that for
any $\delta\in (0, \delta_0)$, the following holds.

Under the conditions of and with the notation from Proposition
\ref{p.coneinvariance}, suppose that for the points $p=(x_p, y_p,
z_p)$ and $q=(x_q, y_q, z_q)$, the following holds:

\vspace{3pt}

1.  $0<z_p<1$ and $0<z_q<1$;

\vspace{3pt}

2.  For some $N\ge N_0$ both $f^N(p)$ and $f^N(q)$  have
$z$-coordinates larger than 1, and both $f^{N-1}(p)$ and
$f^{N-1}(q)$ have $z$-coordinates not greater than 1;

\vspace{3pt}

3. There is a smooth curve $\gamma:[0,1]\to \mathbb{R}^3$ such
that $\gamma(0)=p$, $\gamma(1)=q$, and for each $t\in [0,1]$ we
have $\gamma'(t)\in K^u_{\gamma(t)}$;

\vspace{3pt}

If $N\ge N_0$ {\rm (}i.e., if $z_p$ is small enough{\rm )}, then
\beq\label{e.expansionforN} |Df^N(\mathbf{v})|\ge
\lambda^{\frac{N}{2}(1-4\varepsilon)}|\mathbf{v}| \ \ \ \text{for
any } \ \  \mathbf{v}\in K^u_p, \eneq and if
$Df^N(\mathbf{v})=\mathbf{u}=\mathbf{u}_{xy}+\mathbf{u}_z$, then
\beq\label{e.u} |\mathbf{u}_{xy}|<2\delta^{1/2}|\mathbf{u}_{z}|.
\eneq

Moreover, if $|\mathbf{v}_z|\ge \eta|\mathbf{v}_{xy}|$, then
\beq\label{e.expansionfork} |Df^k(\mathbf{v})|\ge
C\lambda^{\frac{k}{2}(1-4\varepsilon)}|\mathbf{v}| \ \ \ \text{for
each } \ \  k=1, 2, \ldots, N. \eneq

In particular, \beq\label{e.length} \mathrm{length} (f^N(\gamma))\ge
\lambda^{\frac{N}{2}(1-4\varepsilon)} \, \mathrm{length} (\gamma).
\eneq\enprop

In order to establish the distortion property we need better control over the expansion rates. In Proposition \ref{p.vectors} we improve the estimates given by \eqref{e.expansionforN} and \eqref{e.length}. As a first step we show that, roughly speaking, if a point stays for $N$ iterates in a neighborhood where normalizing coordinates are defined then it must be $\lambda^{-N}$-close to the center-stable manifold of a curve of fixed points.

\bprop\label{p.dist} Given $C_1>0, C_2>0, \lambda>1$,
there exist $\delta_0=\delta_0(C_1, C_2,  \lambda)$, $N_0=N_0(C_1, C_2,  \lambda,  \delta_0)\in
\mathbb{N}$, and $C^{**}>C^*>0$ such that for any $\delta\in (0, \delta_0)$, the
following holds.

Let $f:\mathbb{R}^3\to \mathbb{R}^3$ be a $C^2$-diffeomorphism
such that
\begin{itemize}
\item[{\rm (i)}] $\|f\|_{C^2}\le C_1$; \item[{\rm (ii)}] The
planes $\{z=0\}$ and $\{x=0\}$ are invariant under iterates of
$f$; \item[{\rm (iii)}]  Every point of the line $\{z=0, x=0\}$ is
a fixed point of $f$;

\item[{\rm (iv)}] At a point $Q\in \{z=0, x=0\}$ we have
$$
Df(Q)=\begin{pmatrix}
    \lambda^{-1} & 0 & 0 \\
    0 & 1 & 0 \\
    0 & 0 & \lambda \\
  \end{pmatrix}.
$$

 \item[{\rm (v)}]
$\|Df(p)-A\|<\delta$ for every $p\in \mathbb{R}^3$, where
$$
A=Df(Q)=\begin{pmatrix}
    \lambda^{-1} & 0 & 0 \\
    0 & 1 & 0 \\
    0 & 0 & \lambda \\
  \end{pmatrix}.
$$
\end{itemize}
Introduce the following cone fields in $\mathbb{R}^3$:
\begin{align}
\label{e.coneformula1} K_p^u & = \{\mathbf{v}\in T_p\mathbb{R}^3, \
\mathbf{v}=\mathbf{v}_{xy}+\mathbf{v}_z : |\mathbf{v}_z|\ge C_2
\sqrt{|z_p|}|\mathbf{v}_{xy}|\}, \\
\label{e.cuconeformula} K_p^{cu} & = \{\mathbf{v}\in
T_p\mathbb{R}^3, \ \mathbf{v}=\mathbf{v}_{x}+\mathbf{v}_{yz} :
|\mathbf{v}_x|<0.01\lambda^{-1} |\mathbf{v}_{yz}| \}, \\
\label{e.coneformulas} K_p^s & = \{\mathbf{v}\in T_p\mathbb{R}^3,
\ \mathbf{v}=\mathbf{v}_{x}+\mathbf{v}_{yz} : |\mathbf{v}_x|\ge
C_2 \sqrt{|x_p|}|\mathbf{v}_{yz}|\}, \\
\label{e.csconeformula} K_p^{cs} & = \{\mathbf{v}\in
T_p\mathbb{R}^3, \ \mathbf{v}=\mathbf{v}_{z}+\mathbf{v}_{xy} :
|\mathbf{v}_z|<0.01\lambda^{-1} |\mathbf{v}_{xy}| \}.
\end{align}

Suppose that for a finite orbit $p_0, p_1, p_3, \ldots, p_N$ we
have
$$
(p_0)_x\ge 1, \ (p_1)_x<1, \ (p_N)_z\ge 1, \ (p_{N-1})_z<1,
$$
and there are curves $\gamma_0$ and $\gamma_N$ such that
$\gamma_0$ connects $p_0$ with $W^{ss}(Q)$
and is tangent to both cone fields $K^u$ and $K^{cu}$, and
$\gamma_N$ connects $p_N$ with $W^{uu}(Q)$ and is tangent to both
cone fields $K^s$ and $K^{cs}$.

Then
$$
C^*\lambda^{-N}\le |(p_0)_z|\le C^{**}\lambda^{-N}, \ \ \text{\rm
and}
$$
$$
C^*\lambda^{-N}\le |(p_N)_x|\le C^{**}\lambda^{-N}.
$$
 \enprop
 \begin{proof}
 Consider an orthogonal from $p_0$ to the plane $\{z=0\}$, and denote its base by $p^*_0$. There is a unique point
 $Q_0$ on the line $\{z=0, x=0\}$ such that $p_0^*\in W^{ss}(Q_0)$. Denote the line segment connecting
 $p_0$ and $p_0^*$ by $\sigma_0$ and set $\sigma_i=f^i(\sigma_0), \ i=1, 2, \ldots, N$.

 Similarly, consider an orthogonal from $p_N$ to the plane $\{x=0\}$, and denote its base by $p^*_N$. There is a unique
 point $Q_N$ on the line $\{z=0, x=0\}$ such that $p_N^*\in W^{uu}(Q_N)$. Denote the line segment connecting
 $p_N$ and $p_N^*$ by $\rho_N$, and set $\rho_i=f^{-N+i}(\rho_N), \ i=0, 1, 2, \ldots, N-1$.

 We have
 $$
 0<|\sigma_0|<|\sigma_1|<\ldots<|\sigma_{N-1}|<|\sigma_N|, \ 1\le |\sigma_N|\le \lambda(1+\delta),
 $$
 $$
 0<|\rho_N|<|\rho_{N-1}|<\ldots<|\rho_{1}|<|\rho_0|, \ 1\le |\rho_0|\le \lambda(1+\delta).
 $$

 Denote $b_k=\text{\rm dist}(p_k, Q)$. Then we have
 \begin{align*}
(\lambda^{-1}-\min(\delta, C_1b_k))|\sigma_k| & \le
|\sigma_{k-1}|\le (\lambda^{-1}+\min(\delta, C_1b_k))|\sigma_k|, \
\ k= 1, 2, \ldots, N, \\
(\lambda^{-1}-\min(\delta, C_1b_k))|\rho_k| & \le |\rho_{k+1}|\le
(\lambda^{-1}+\min(\delta, C_1b_k))|\rho_k|, \ \ k= 0, 1, 2,
\ldots, N-1.
\end{align*}

Now we have
\begin{align*}
 b_k & =\text{\rm dist}(p_k, Q) \\
 & \le \text{\rm dist}(Q, Q_0)+|\rho_k|+\text{\rm dist}(Q, Q_N)+|\sigma_k| \\
 & \le |\sigma_k|+|\rho_k|+C_3C_2(\sqrt{|\rho_N|}+\sqrt{|\sigma_0|}),
\end{align*}
where $C_3$ does not depend on $N$. Indeed, the distance between
$p_0^*$ and $W^{ss}(Q)$ is bounded above by the length of the
curve $\gamma_0$, and since $\gamma_0$ is tangent to the cone
fields $K^u$ and $K^{cu}$, its length is not greater than
$C_2\sqrt{|\sigma_0|}$. On the other hand, $\text{\rm dist}(Q,
Q_0)$ is of the same order as that distance since the strong
stable manifolds of fixed points form a $C^1$-foliation of the
plane $\{z=0\}$. In the same way one
gets an estimate $\text{\rm dist}(Q, Q_N)\le
C_3C_2\sqrt{|\rho_N|}$.

Since we have the a priori estimates $|\sigma_k|\le
(1+\delta)(\lambda-\delta)^{-N+k}$ and $|\rho_k|\le
(1+\delta)(\lambda-\delta)^{-k}$, we also have
$$
b_k\le |\sigma_k|+|\rho_k|+2C_2(1+\delta)(\lambda-\delta)^{-N/2}.
$$
If $k<N/2$, then
\begin{align*}
b_k & \le
(\lambda-\delta)^{-k}(1+(\lambda-\delta)^{-N+2k}+2C(\lambda-\delta)^{-N/2+k}) \\
& \le C'(\lambda-\delta)^{-k}.
\end{align*}
If $k\ge N/2$, then
\begin{align*}
b_k & \le
(\lambda-\delta)^{-N+k}(1+(\lambda-\delta)^{N-2k}+2C(\lambda-\delta)^{N/2-k})
\\
& \le C'(\lambda-\delta)^{-N+k}.
\end{align*}
Therefore we have
\begin{align*}
|\sigma_0| & \le |\sigma_N|\prod_{k=1}^N(\lambda^{-1}+\min(\delta, C_1b_k)) \\
& \le \lambda(1+\delta) \lambda^{-N} \prod_{k=1}^{[N/2]}
(1+C_1C'(\lambda-\delta)^{-k}) \cdot
\prod_{k=[N/2]+1}^{N}(1+C_1C'(\lambda-\delta)^{N/2-k}) \\
& \le C^{**}\lambda^{-N}.
\end{align*}
Also,
\begin{align*}
|\sigma_0| & \ge |\sigma_N|\prod_{k=1}^N(\lambda^{-1}-\min(\delta,
C_1b_k)) \\
& \ge \lambda^{-N} \prod_{k=1}^{[N/2]}
(1-C_1C'(\lambda-\delta)^{-k}) \cdot \prod_{k=[N/2]+1}^{N}
(1-C_1C'(\lambda-\delta)^{N/2-k}) \\
& \ge C^{*}\lambda^{-N}.
\end{align*}
In the same way we get estimates for $\rho_N$.
\end{proof}

\bprop\label{p.vectors} Given $C_1>0, C_2>0, \lambda>1$,
there exist $\delta_0=\delta_0(C_1, C_2,  \lambda)$, $N_0\in
\mathbb{N}, N_0=N_0(C_1, C_2,  \lambda,  \delta_0)\in
\mathbb{N}$, and
$\widetilde{C}>0$ such that for any $\delta\in (0, \delta_0)$, the
following holds.

Under the conditions of and with the notation from
Proposition~\ref{p.dist}, suppose that $\mathbf{v}\in
T_{p_0}\mathbb{R}^3, \mathbf{v}\in K^u_{p_0}$. Then
$|Df_{p_0}^N(\mathbf{v})|\ge
\widetilde{C}\lambda^{N/2}|\mathbf{v}|$. \enprop

\begin{proof}
We will use the notation from Proposition \ref{p.dist} and its
proof.

Let us denote $\mathbf{v}_k=Df^k(\mathbf{v}), k=0, 1, \ldots, N,$
and $D_k=|(\mathbf{v}_k)_z|, d_k=|(\mathbf{v}_k)_{xy}|$. Let us
normalize $\mathbf{v}$ in such a way that $d_0=1$. Since
$\mathbf{v}\in K^u_{p_0}$ and $|\sigma_0| \ge C^{*}\lambda^{-N}$,
we have $D_0\ge C_5\lambda^{-N/2}$, where $C_5$ is independent of
$N$.

Denote
$$
Df(p)=\begin{pmatrix}
        \nu(p) & m_1(p) & t_1(p) \\
        m_2(p) & e(p) & t_2(p) \\
        s_1(p) & s_2(p) & \lambda(p) \\
      \end{pmatrix}.
$$
We have
$$
Df(p)(\mathbf{v})=\begin{pmatrix}
        \nu(p) & m_1(p) & t_1(p) \\
        m_2(p) & e(p) & t_2(p) \\
        s_1(p) & s_2(p) & \lambda(p) \\
      \end{pmatrix}\begin{pmatrix}
                     \mathbf{v}_x \\
                     \mathbf{v}_y \\
                     \mathbf{v}_z \\
                   \end{pmatrix}=\begin{pmatrix}
                     \nu(p)\mathbf{v}_x+m_1(p)\mathbf{v}_y+t_1(p)\mathbf{v}_z \\
                     m_2(p)\mathbf{v}_x+e(p)\mathbf{v}_y+t_2(p)\mathbf{v}_z \\
                     s_1(p)\mathbf{v}_x+s_2(p)\mathbf{v}_y+\lambda(p)\mathbf{v}_z \\
                   \end{pmatrix}.
$$
Since $\|f\|_{C^2}\le C_1$, we also have $|\nu(p)|\le \lambda^{-1}
+ C_1\mathrm{dist}(Q, p)$, $|m_1(p)|$, $|m_2(p)|$, $|t_1(p)|$,
$|t_2(p)| \le C_1 \mathrm{dist}(Q, p)$, and $|\lambda(p)|\ge
\lambda-C_1\mathrm{dist}(Q, p)$. Furthermore, if $p$ belongs to the plane
$\{z=0\}$, then $s_1(p)=s_2(p)=0$. Therefore, for arbitrary $p$,
we have $|s_1(p)|, |s_2(p)|\le C_1z_p$. This implies that we have
the following estimates:

\beq\label{e.recurrence}
\left\{
  \begin{array}{ll}
    d_{k+1}\le (1+\min(\delta, C_1b_k))d_k+\min(\delta, C_1b_k) D_k \\
    D_{k+1}\ge (\lambda-\min(\delta, C_1b_k))D_k- \min(\delta, C_1|\sigma_k|) d_k
  \end{array}
\right.
\eneq

\blm\label{l.kstar}
There exists $k^*$ such that $d_k\ge D_k$ for all $k\le k^*$, and $d_k<D_k$ for all $k>k^*$.
\elm
\begin{proof}
Indeed, if $D_k>d_k$, then
$$
d_{k+1} \le (1+\delta)d_k + \delta D_k \le (1+2\delta)D_k
$$
and
$$
D_{k+1}\ge (\lambda-\delta)D_k-\delta d_k\ge
(\lambda-2\delta)D_k.
$$
Since $\lambda-2\delta>1+2\delta$, we have $D_{k+1}>d_{k+1}$.
\end{proof}

We have the following preliminary estimates. If $k<N/2$, then
$b_k\le C'(\lambda-\delta)^{-k}$; if $k\ge N/2$, then $b_k\le
C'(\lambda-\delta)^{-N+k}$. Also, $|\sigma_k|\le
(\lambda-\delta)^{-N+k}$ for each $k=0, 1, \ldots, N$. Notice that
this implies that $\prod_{i=1}^N(1+C_1b_i)$ is bounded by a
constant that is independent of $N$. And, finally, $D_k\ge
C_5\lambda^{-N/2}(\lambda-\delta)^k$; see \cite[Lemma~6.1]{DG09a}.

If $D_k\le d_k$ (i.e., $k\le k^*$), then
\begin{align*}
d_{k+1} & \le (1+C_1b_k)d_k+C_1b_kD_k \\
& \le (1+2C_1b_k)d_k \\
& \le \left[\prod_{i=1}^k(1+2C_1b_k)\right]d_0 \\
& \le C_6,
\end{align*}
where $C_6$ does not depend on $k$ or $N$.
Moreover, we have
\begin{align*}
D_{k+1} & \ge (\lambda-C_1b_k)D_k - C_1 |\sigma_k|d_k \\
& \ge (\lambda-C_1b_k)D_k - C_1 C_6(\lambda-\delta)^{-N+k} \\
& \ge (\lambda - C_1 b_k) D_k \left(1 - \frac{C_1 C_6}{\lambda -
C_1b_k} \cdot \frac{(\lambda - \delta)^{-N+k}}{D_k}\right) \\
& \ge (\lambda - C_1 b_k) D_k \left(1 - \frac{C_1 C_6}{\lambda -
C_1 b_k} \cdot \frac{(\lambda - \delta)^{-N+k}}{C_5
\lambda^{-N/2}(\lambda -
\delta)^k}\right) \\
& \ge (\lambda - C_1 b_k) D_k \left(1 - \left(\frac{C_1
C_6}{C_5(\lambda - C_1 b_k)}\right)(\lambda^{1/2}(\lambda -
\delta)^{-1})^{N}\right) \\
& \ge \lambda^{k+1} D_0 \left[\prod_{i=0}^{k}(1 -
(C_1\lambda^{-1}) b_i)\right] \cdot \left(1 - \left(\frac{C_1
C_6}{C_5(\lambda - C_1
b_k)}\right)(\lambda^{1/2}(\lambda - \delta)^{-1})^{N}\right)^k \\
& \ge \lambda^{k+1} D_0 \left[\prod_{i=0}^{k}(1 -
(C_1\lambda^{-1}) b_i) \right] \cdot \left(1 - \left(\frac{C_1
C_6}{C_5(\lambda - C_1
b_k)}\right)(\lambda^{1/2}(\lambda - \delta)^{-1})^{N}\right)^N \\
& \ge C_7 \lambda^{k+1}D_0,
\end{align*}
since for any $C>0$ and $\xi\in(0,1)$, one has $\lim_{N\to
\infty}(1-C\xi^N)^N=1$.

If $d_k<D_k$ (i.e., $k>k^*$), then
\begin{align*}
D_{k+1} & \ge (\lambda-C_1b_k)D_k - C_1 |\sigma_k| d_k \\
& \ge (\lambda - C_1 b_k - C_1 |\sigma_k|)D_k \\
& \ge \lambda D_k(1 - \lambda^{-1} C_1 b_k - \lambda^{-1} C_1 |\sigma_k|) \\
& \ge \lambda^{k+1} C_7 D_0 \prod_{i=k^*}^{k}(1 - \lambda^{-1} C_1
b_k - \lambda^{-1} C_1 |\sigma_k|) \\
& \ge C_8\lambda^{k+1}D_0,
\end{align*}
where $C_8$ does not depend on $N$ or $k$.

Finally, $|Df_{p_0}^N(\mathbf{v})|\ge D_N\ge C_8\lambda^ND_0\ge
\widetilde{C}\lambda^{N/2}|\mathbf{v}|$.
\end{proof}

Below (in the proof of Proposition \ref{p.distances}) we will also need an estimate on $k^*$ provided by Lemma \ref{l.kstar}. Namely, we claim that $k^*$ cannot be much larger than $N/2$. The formal statement is the following.

\blm\label{l.onkstarboundabove}
There is a constant $C_9$ independent of $N$ such that
$$
\lambda^{k^*}\le C_9\lambda^{N/2}.
$$
\elm
\begin{proof}
We know that $D_{k^*-1}\le d_{k^*-1}\le C_6$. Therefore $C_6\ge
D_{k^*-1}\ge C_7\lambda^{k^*}D_0\ge C_7\lambda^{k^*}\cdot
C_5\lambda^{-N/2}$, so $\lambda^{k^*}\le
(C_6C_7^{-1}C_5^{-1})\lambda^{N/2}$.
\end{proof}

Now we are ready to formulate the statement that will be used to
check the distortion property of the trace map.

\bprop\label{p.distances} Given $C_1>0, C_2>0, C_3>0, \lambda>1$,
there exist $\delta_0=\delta_0(C_1, C_2, C_3, \lambda)$, $N_0=N_0(C_1, C_2, C_3, \lambda,  \delta_0)\in
\mathbb{N}$, and $C>0$ such that for
any $\delta\in (0, \delta_0)$ and any $\Delta>0$, 
the following holds.

Under the conditions of and with the notation from
Proposition~\ref{p.dist}, suppose that the curve $\gamma_0$ has a
curvature bounded by $C_3$. Suppose also that for the points
$p=(x_p, y_p, z_p)$ and $q=(x_q, y_q, z_q)$, the following holds:

\vspace{3pt}

1. $p, q\in \gamma_0$;

\vspace{3pt}

2. For some $N\ge N_0$ both $f^N(p)$ and $f^N(q)$  have
$z$-coordinates larger than 1, and both $f^{N-1}(p)$ and
$f^{N-1}(q)$ have $z$-coordinates not greater than 1;

\vspace{3pt}

3. $\text{\rm dist}(f^N(p), f^N(q))=\Delta$.

\vspace{3pt}

Denote $p_k=f^k(p)$, $q_k=f^k(q)$, $k=0, \ldots, N$. 

\vspace{3pt}

Let  $\mathbf{v}\in T_p\mathbb{R}^3$ and  $\mathbf{w}\in
T_q\mathbb{R}^3$ be vectors tangent to $\gamma_0$.

\vspace{3pt}

Denote $\mathbf{v}_k=Df^k(v)$ and $\mathbf{w}_k=Df^k(w)$, $k=0, \ldots, N$. Let
$\alpha_k$ be the angle between $\mathbf{v}_k$ and $\mathbf{w}_k$.

 \vspace{3pt}

 Then,
 \beq\label{e.sumsarebounded}
 \sum_{k=0}^{N}\alpha_k<C\Delta, \ and \ \ \sum_{k=0}^{N} \text{\rm dist}(p_k, q_k)<C\Delta.
 \eneq

 \enprop

\begin{proof}
First of all, notice that it is enough to prove
Proposition~\ref{p.distances} in the case when the points $p$ and
$q$ are arbitrarily close to each other. Indeed, otherwise split
the piece of the curve $\gamma_0$ between the points $p$ and $q$
into a large number of extremely small pieces. If for each of them
the statement of Proposition~\ref{p.distances} holds, then by
subadditivity of the inequalities \eqref{e.sumsarebounded} it
holds in general.

Denote by $\Gamma$ the piece of the curve $\gamma_0$ between the
points $p_0$ and $q_0$, and set $\Gamma_k=f^k(\Gamma)$, $k=0, 1,
2, \ldots, N$. Denote $\mu_k=|\Gamma_k|$. Due to the remark above
we can assume that for any vector tangent to $\Gamma$, the value
of $k^*$ is the same. From the proof of
Proposition~\ref{p.vectors} we see that for $k=0, 1, \ldots, k^*$,
we have $\mu_k\le C_6\mu_0$, and $\Delta\approx \mu_N\ge
\widetilde{C}\lambda^{N/2}\mu_0\ge
\widetilde{C}\lambda^{N/2}C_6^{-1} \mu_{k^*}$, so $\mu_{k^*}\le
(\widetilde{C}^{-1}C_6)\lambda^{-N/2}\mu_N\le
C_{11}\lambda^{-N/2}\Delta$.

On the other hand, if $k>k^*$, then
$$
\Delta\approx \mu_N\ge D_N\ge C_7\lambda^{N-k}D_{k^*}\ge
C_7\lambda^{N-k}\frac{1}{2}\mu_k,
$$
where we denote by $D_k$ the length of the projection of
$\Gamma_k$ to the $z$-axis (slightly abusing the notation). Therefore,
$\mu_k\le (2C_7^{-1})\lambda^{-N+k}\Delta$.

It follows that we have
\begin{align*}
\sum_{k=0}^{N} \text{\rm dist}(p_k, q_k) & \le \sum_{k=0}^{N}
\mu_k \\
& = \sum_{k=0}^{k^*} \mu_k +\sum_{k=k*+1}^{N} \mu_k \\
& \le k^*\cdot C_6 \mu_0 + \sum_{k=k^*+1}^{N} (2C_7^{-1})
\lambda^{-N+k} \Delta \\
& \le k^*\cdot \widetilde{C}^{-1} \lambda^{-N/2} C_6 \Delta +
C_{12} \Delta \\
& \le C\Delta.
\end{align*}
 Notice that for any two vectors $\mathbf{v}, \mathbf{w}\in K^{cu}$
 $$
 \angle(A\mathbf{v}, A\mathbf{w})\le \lambda \angle (\mathbf{v}, \mathbf{w}), \ \text{\rm where}\ \ A=\begin{pmatrix}
    \lambda^{-1} & 0 & 0 \\
    0 & 1 & 0 \\
    0 & 0 & \lambda \\
  \end{pmatrix},
 $$
 and if a linear operator $B$ is $\xi$-close to $A$, then
\begin{align*}
 \angle(B\mathbf{v}, B\mathbf{w}) & \le (\lambda+\xi)(1+\xi)\angle (\mathbf{v}, \mathbf{w}) \\
 & = (\lambda+\xi(\lambda+1+\xi))\angle (\mathbf{v}, \mathbf{w}) \\
 & < \lambda(1+2\xi)\angle (\mathbf{v}, \mathbf{w}).
\end{align*}
Therefore we have
$$
\alpha_0\le C_3\mu_0
$$
and
$$
\alpha_{k+1}\le \lambda(1+2C_1b_k)\alpha_k+C_1\mu_k, \ \ k=0,1,
\ldots, k^*.
$$
Since $\prod_{k=0}^{k^*}(1+2C_1b_k)\le C_{13}$ for some $C_{13}$
that is independent of $k^*$ and $N$, we have
$$
\alpha_k\le (\lambda^k+\lambda^{k-1}+\ldots +\lambda+1)\cdot
(C_{13}C_1C_6C_3)\mu_0\le C_{14}\lambda^{k}\mu_0,
$$
where $C_{14}$ is also independent of $k$ and $N$. In particular,
$$
\alpha_{k^*}\le C_{14}\lambda^{k^*}\mu_0\le
(C_{14}C_9)\lambda^{N/2}\widetilde{C}^{-1}\lambda^{-N/2}\Delta\le
C_{15}\Delta
$$
and
$$
\sum_{k=0}^{k^*}\alpha_k\le
\sum_{k=0}^{k^*}C_{14}\lambda^{k}\mu_0\le
C_{16}\lambda^{k^*}\mu_0\le C_{17}\Delta.
$$
Now denote
$$
K^{uu}_p=\{\mathbf{v}\in T_p\mathbb{R}^3 : |\mathbf{v}_z|>100
\lambda |\mathbf{v}_{xy}|\}.
$$
If $\mathbf{v}, \mathbf{w}\in K^{uu}$, then $\angle(A\mathbf{v},
A\mathbf{w})\le \lambda^{-1/2} \angle(\mathbf{v},\mathbf{w})$, and
the same holds for any linear operator $B$ which is $\delta$-close
to $A$.

There exists $m\in \mathbb{N}$ independent of $N$ such that if for
a vector $\mathbf{v}$ we have $|\mathbf{v}_z|>|\mathbf{v}_{xy}|$,
then $Df^m(\mathbf{v})\in K^{uu}$.

Also
$$
\alpha_{k^*+m}=\angle(Df^m(v_{k^*}), Df^m(w_{k^*}))\le
C_{15}(\lambda+2\delta)^m\Delta=C_{16}\Delta,
$$
and for $k\ge k^*+m$, we have $\alpha_{k+1}\le
\lambda^{-1/2}\alpha_k+C_1\mu_k$.

Denote $\nu=\lambda^{-1/3}-\lambda^{-1/2}$.

If $C_1\mu_k<\nu\alpha_k$, then
$$
\alpha_{k+1}\le
\lambda^{-1/2}\alpha_k+\nu\alpha_k=(\lambda^{-1/2}+\nu)\alpha_k=\lambda^{-1/3}\alpha_k.
$$

If $C_1\mu_k\ge \nu\alpha_k$, then
$$
\alpha_{k+1}\le \lambda^{-1/2}\alpha_k+C\mu_k\le
(\lambda^{-1/2}\nu^{-1}+1)C_1\mu_k.
$$
Since $\sum_{k=k^*+m}^N\mu_k\le \sum_{k=k^*+m}^N
(2C_7^{-1})\lambda^{-N+k}\Delta\le C_{12}\Delta$, this implies
that
$$
\sum_{k=k^*+m}^{N}\alpha_k\le C_{18}\Delta
$$
and
$$
\sum_{k=0}^{N}\alpha_k\le \sum_{k=0}^{k^*}\alpha_k
+\left(\alpha_{k^*+1}+\ldots
+\alpha_{k^*+m-1}\right)+\sum_{k=k^*+m}^{N}\alpha_k\le C\Delta,
$$
concluding the proof.
\end{proof}

\subsection{Distortion Property: Preliminary Estimates}\label{ss.distprem}

The main result of this subsection is the following statement:

\blm\label{l.prelim} There are constants $R>0$, $V_0>0$, and
$\kappa>0$ such that for any $V\in (0,V_0)$ and $N\in \mathbb{N}$,
the following holds. Suppose that $\gamma\subset
T^N(\mathbf{l}_V)\backslash O_{r_1}$ is a connected curve of
length not greater than $\kappa$. Let the points $p,q\in
\mathbf{l}_V$ be such that $T^N_V(p)\in \gamma$ and $T^N_V(q)\in
\gamma$, and $v_p$ and $v_q$ be unit vectors tangent to $\gamma$
at points $p$ and $q$. Then
$$
\sum_{i=0}^{N}\left(\angle(DT_V^i(v_p), DT_V^i(v_q))+\text{\rm dist}(T_V^i(p), T_V^i(q))\right)<R.
$$
\elm

Notice that $F^{-1}(\mathbb{S}\backslash O_{r_2})$ is a torus
without small neighborhoods of the preimages of the singularities, and we
can define the following map
\begin{equation}\label{e.Ttilde}
\widetilde{T}_V: F^{-1}(\mathbb{S}\backslash O_{r_2}) \to
\mathbb{T}^2, \ \ \ \widetilde{T}_V=F^{-1}\circ \pi_V\circ
T_V\circ \pi_V^{-1}\circ F\equiv F_V^{-1}\circ T_V\circ F_V.
\end{equation}

If $V$ is small, $\widetilde{T}_V$ is $C^2$-close to the linear
automorphism $\mathcal{A}$ on its domain.

\blm\label{l.angels} For $V_0>0$ small enough, there exists $t\in (0,1)$ such that
for $V \in [0,V_0]$, $p,q \in \mathbb{T}^2\backslash
F^{-1}(O_{r_2})$ and unit vectors $v_p\in K_p^u$, $v_q\in K_q^u$,
we have
$$
\angle (D\widetilde{T}_{V, p}(v_p), D\widetilde{T}_{V, q}(v_q))\le
t\angle(v_p, v_q)+2\|\widetilde{T}_V\|_{C^2}\text{\rm dist}(p,q).
$$
\elm
\begin{proof}
If $V_0$ is small, then $\widetilde{T}_V$ is $C^2$-close to the
linear automorphism $\mathcal{A}$. In particular, for any point
$p\in \mathbb{T}^2\backslash F^{-1}(O_{r_2})$ and any vectors $v_1,
v_2\in K^u_p$,
$$
\angle (D\widetilde{T}_{V, p}(v_1), D\widetilde{T}_{V, p}(v_2))\le
t\angle(v_1, v_2),
$$
where $t\in (0,1)$ can be chosen uniformly for all $V\in [0,V_0]$
and $p\in \mathbb{T}^2\backslash F^{-1}(O_{r_2})$. Therefore we
have
\begin{align*}
\angle (D\widetilde{T}_{V, p}(v_p), D\widetilde{T}_{V, q}(v_q)) & \le
\angle (D\widetilde{T}_{V, p}(v_p), D\widetilde{T}_{V, p}(v_q)) + \angle
(D\widetilde{T}_{V, p}(v_q), D \widetilde{T}_{V, q}(v_q)) \\
& \le t\angle(v_p, v_q)+2\|D\widetilde{T}_{V,p}(v_q)-D\widetilde{T}_{V,q}(v_q)\| \\
& \le t\angle(v_p, v_q)+2\|\widetilde{T}_V\|_{C^2}\text{\rm
dist}(p,q),
\end{align*}
as claimed.
\end{proof}

\bdef
For any points $p, q$ and any vectors $v_p, v_q$ define
\begin{equation}\label{e.ffunction}
\mathfrak{F}(p, q, v_p, v_q) \equiv \frac{\angle (v_p, v_q)}{\text{\rm dist}(p,q)}.
\end{equation}
\endef

\blm\label{l.inequality} For $p,q\in \mathbb{T}^2\backslash
F^{-1}(O_{r_2})$, $p\ne q$, and vectors $v_p\in K_p^u$, $v_q\in
K_q^u$, consider the function $\mathfrak{F}(p, q, v_p, v_q)$ defined by \eqref{e.ffunction}.
Suppose that $p$ and $q$ belong to a curve that is tangent to the
unstable cone field. Then
$$
\frak{F}(\widetilde{T}_V(p),\widetilde{T}_V(q), D\widetilde{T}_{V,
p}(v_p), D\widetilde{T}_{V, q}(v_q))\le t\mathfrak{F}(p, q, v_p,
v_q) +2\|\widetilde{T}_V\|_{C^2}.
$$
In particular, if $\mathfrak{F}(p, q, v_p, v_q) >
\frac{4\|\widetilde{T}_V\|_{C^2}}{1-t}$, then
$$
\frak{F}(\widetilde{T}_V(p),\widetilde{T}_V(q), D\widetilde{T}_{V,
p}(v_p), D\widetilde{T}_{V, q}(v_q))\le
\frac{1+t}{2}\mathfrak{F}(p, q, v_p, v_q).
$$
\elm
\begin{proof}
We have
\begin{align*}
\frak{F}(\widetilde{T}_V(p),\widetilde{T}_V(q), D\widetilde{T}_{V, p}(v_p),
D\widetilde{T}_{V, q}(v_q)) & =\frac{\angle(D\widetilde{T}_{V, p}(v_p),
D\widetilde{T}_{V, q}(v_q))}{\text{\rm dist}(\widetilde{T}_V(p),\widetilde{T}_V(q))} \\
& \le \frac{t\angle(v_p, v_q)+2\|\widetilde{T}_V\|_{C^2}\text{\rm dist}(p,q)}{\text{\rm dist}(p,q)} \\
& = t\mathfrak{F}(p, q, v_p, v_q) +2\|\widetilde{T}_V\|_{C^2}.
\end{align*}
If we also have $\mathfrak{F}(p, q, v_p, v_q) > \frac{4\|\widetilde{T}_V\|_{C^2}}{1-t}$, then
\begin{align*}
t\mathfrak{F}(p, q, v_p, v_q) + 2 \|\widetilde{T}_V\|_{C^2} & \le
t\mathfrak{F}(p, q, v_p, v_q) +\frac{1-t}{2}\mathfrak{F}(p, q, v_p, v_q) \\
& = \frac{1+t}{2}\mathfrak{F}(p, q, v_p, v_q).
\end{align*}
\end{proof}

Lemma~\ref{l.inequalities} immediately implies the following statement.

\blm\label{l.changeofcoordinates} Fix a small $V\ge 0$. Suppose
that $a, b \in \mathbb{T}^2\backslash F^{-1}(O_{r_2})$ are such
that $\Psi_V(a)$ and $\Psi_V(b)$ are defined, and $v_a\in
T_a\mathbb{T}^2\backslash F^{-1}(O_{r_2})$, $v_b\in
T_b\mathbb{T}^2\backslash F^{-1}(O_{r_2})$. Then,
$$
\frak{F}(a, b, v_a, v_b)\le \widetilde{C}^2(\frak{F}(\Psi_V(a),
\Psi_V(b), D\Psi_V(v_a), D\Psi_V(v_b))+1)
$$
and
$$
\frak{F}(\Psi_V(a), \Psi_V(b), D\Psi_V(v_a), D\Psi_V(v_b))\le \widetilde{C}^2(\frak{F}(a, b, v_a, v_b)+1).
$$
\elm

Since $\widetilde{T}_V$ is $C^2$-close to the linear automorphism
$\mathcal{A}$ on its domain for small $V\ge 0$, we can assume that
for all $V\in [0, V_0]$, we have $\|\widetilde{T}_V\|_{C^2}\le
10$.

Let $C$ be the constant from Proposition~\ref{p.distances}, where
$C_3=(\frac{40}{1-t}+1)\widetilde{C}^2$ was taken. Fix a small
$\tau>0$.  Take $n^*\in \mathbb{N}$ such that
\begin{equation}\label{e.nstarchoice}
\left(\frac{1+t}{2}\right)^{n^*}\widetilde{C}^2(C+1)\le
\frac{40}{1-t}, \ \ \text{\rm and}\ \ \
\widetilde{C}^2\mu^{n^*}\ge \mu^{n^*(1-\frac{\tau}{8})}.
\end{equation}

Now we are going to choose a neighborhood $U$ of the set of singularities $\{P_1, P_2, P_3, P_4\}$ in such a way that if an orbit of a point leaves $U$, then it does not enter $U$ for the next $n^*$ iterates. Also, we will choose a smaller neighborhood $U^*\subset U$ such that if a small curve is tangent to an unstable cone field (see Definition \ref{d.tangenttounstable}) and intersects $U^*$, then either it is entirely inside of $U$, or its iterates will continue to intersect $U$ until they reach the opposite rectangle. Here is how we do that.

For small $r_3\in (0, r_2)$, we denote $O_{r_3}^0=O_{r_3},
O_{r_3}^1=T(O_{r_3})\cap O_{r_1}$, $O_{r_3}^i=T(O_{r_3}^{i-1})\cap
O_{r_1}$ for each $i>1$, $O_{r_3}^i=T^{-1}(O_{r_3}^{i+1})\cap
O_{r_1}$ for each $i<0$, and
$U=\cup_{i=-\infty}^{\infty}(O_{r_3}^i)$. We will take $r_3$ so
small that the following property holds. If $p\in S_V$ is such
that $T^{-1}(p)\in U$ but $p\not \in U$, then $T^n(p)\not \in U$
for every $n \in \N$ with $n\le n^*$.

For small $r_4\in (0, r_3)$, we denote $O_{r_4}^0=O_{r_4},
O_{r_4}^1=T(O_{r_4})\cap O_{r_1}$, $O_{r_4}^i=T(O_{r_4}^{i-1})\cap
O_{r_1}$ for each $i>1$, $O_{r_4}^i=T^{-1}(O_{r_4}^{i+1})\cap
O_{r_1}$ for each $i<0$, and
\begin{equation}\label{e.ustarchoice}
U^*=\bigcup_{i=-\infty}^{\infty}(O_{r_4}^i).
\end{equation}
We will take $r_4$ so small that the following property holds.
Suppose $\gamma$ is a curve on $\mathbb{T}^2$ such that
$\Psi_V(\gamma)$ is defined, $F_V(\gamma)\subset O_{r_1}$,
$\gamma$ is tangent to the unstable cone field,  $F_V(\gamma)\cap
O_{r_1}\backslash U\ne \emptyset$, and $F_V(\gamma)\cap U^*\ne
\emptyset$. Then there is $k\in \mathbb{N}$ such that
$T_V^n(F_V(\gamma))\cap U\ne \emptyset$ for all natural $n\le k$,
and $T_V^k(F_V(\gamma))$ intersects the opposite rectangle of the
Markov partition.

\blm\label{l.vector} There are $V_0>0$ and $R_1>0$ such that the
following holds. Suppose that $v$ is a non-zero vector tangent to
the line $\mathbf{l}_V$ at some point $p\in \mathbf{l}_V$. Let
$N\in \mathbb{N}$ be such that $T_V(p)$ belongs to the bounded
component of $S_V\backslash O_{r_1}$. Then,
$$
\sum_{i=0}^{N}\|DT^i_V(v)\|\le R_1\|DT_V^N(v)\|.
$$
\elm

\begin{proof}
If $V$ is small enough, the vector $DT^{m_0}(v)$ is tangent to the
unstable cone field. Let us split the orbit $\{T^{m_0}(p),
T^{m_0+1}(p), \ldots, T^N(p)\}$ into several intervals
$$
 \{T^{m_0}(q), T^{m_0+1}(q), \ldots, T^{k_1-1}(q)\}, \{T^{k_1}(q),  \ldots, T^{k_2-1}(q)\}, \ldots, \{T^{k_s}(q), \ldots, T^N(q)\}
 $$
 in such a way that the following properties hold:
 \begin{enumerate}
 \item
 for each $  i=1, 2, \ldots, s$, the points $T^{k_i-1}(q)$ and $T^{k_i}(q)$ are outside of $O_{r_2}$;
 \item
 if $\{T^{k_i}(q),  \ldots, T^{k_{i+1}-1}(q)\}\cap U^*\ne \emptyset$, then $\{T^{k_i}(q),  \ldots, T^{k_{i+1}-1}(q)\}\subset U^*$;
 \item
 for each $  i=1, 2, \ldots, s-1$, we have either $k_{i+1}-k_i\ge n^*$ (where $n^*$ is chosen due to (\ref{e.nstarchoice})) or $\{T^{k_i}(q),  \ldots, T^{k_{i+1}-1}(q)\}\cap U^*\ne \emptyset$.
 \end{enumerate}
Such a splitting exists due to the choice of $U^*\subset U$ in \eqref{e.ustarchoice} above.

Apply Proposition \ref{p.distances}   to those intervals in the
splitting that are contained in $U^*$. The choice of $n^*$ in
\eqref{e.nstarchoice} above guarantees for the intervals that do
not intersect $U^*$ uniform expansion of the vector. The first and
the last interval may have length greater than $n^*$, and then we
have uniform expansion that "kills" the distortion added by the
change of coordinates, or smaller than $n^*$, but then they do not
add more than a constant to the sum. As a result, the required sum
is bounded above by a geometrical progression.
\end{proof}

Lemma \ref{l.vector} implies the following statement.

\blm\label{l.sum} There are constants $R_1>0$, $V_0>0$, and
$\kappa_1>0$ such that for any $V\in (0,V_0)$ and $N\in
\mathbb{N}$, the following holds. Suppose that $\gamma\subset
T^N(\mathbf{l}_V)\backslash O_{r_1}$ is a connected curve of
length not greater than $\kappa_1$. Let the points $p,q\in
\mathbf{l}_V$ be such that $T^N_V(p)\in \gamma$ and $T^N_V(q)\in
\gamma$. Then,
$$
\sum_{i=0}^{N}\text{\rm dist}(T_V^i(p), T_V^i(q))<R_1.
$$
\elm

Finally, the choice of $n^*$ and Proposition~\ref{p.distances}
imply that the function $\frak{F}(T^i_V(p), T^i_V(q), DT^i_V(v_p),
DT^i_V(v_q))$ is uniformly bounded, and together with
Lemma~\ref{l.sum} this proves Lemma~\ref{l.prelim}.

\subsection{Proof of the Distortion Property}\label{ss.proofdist}

\begin{proof}[Proof of Proposition \ref{p.distortion}]

Notice that we need to prove that
$$
\left|\log \left(\frac{|\Gamma_G||\gamma_B|}{|\Gamma_B||\gamma_G|}\right)\right|
$$
is bounded by some constant independent of the choice of the gap
and of $V$. There are points $p_G\in \gamma_G$ and $p_B\in
\gamma_B$ such that if $v_G$ is a unit  vector tangent to the
curve $\gamma_G$ at $p_G$, and $v_B$ is a unit vector tangent to
the curve $\gamma_B$ at $p_B$, then
\begin{align*}
\left|\log \left(\frac{|\Gamma_G||\gamma_B|}{|\Gamma_B||\gamma_G|}\right)\right| & =
\left|\log \left(\frac{|T_V^{n+2}(\gamma_G)||\gamma_B|}{|T_V^{n+2}(\gamma_B)||\gamma_G|}\right)\right| \\
& = \left|\log \left(\frac{|DT_V^{n+2}(v_G)|}{|DT_V^{n+2}(v_B)|}\right)\right| \\
& = \left|\sum_{i=0}^{n+1}\left(\log |DT_V|_{DT_V^i(v_G)}(T_V^i(p_G))| - \log |DT_V|_{DT_V^i(v_B)}(T_V^i(p_B))|\right)\right| \\
& \le \sum_{i=0}^{n+1}\left|\log |DT_V|_{DT_V^i(v_G)}(T_V^i(p_G))| - \log |DT_V|_{DT_V^i(v_B)}(T_V^i(p_B))|\right| \\
& \le \sum_{i=0}^{n+1}\left||DT_V|_{DT_V^i(v_G)}(T_V^i(p_G))| -  |DT_V|_{DT_V^i(v_B)}(T_V^i(p_B))|\right|
\end{align*}

We estimate each of the terms in this sum using

\blm\label{l.lastest} Suppose $f:\mathbb{R}^n\to \mathbb{R}^n$ is a smooth map, $a,
b\in \mathbb{R}^n$, and $v_a\in T_a\mathbb{R}^n, v_b\in
T_b\mathbb{R}^n$ are unit vectors. Then,
$$
\left||Df|_{v_a}(a)|-|Df|_{v_b}(b)|\right|\le \|f\|_{C^2}(\angle(v_a,v_b)+|a-b|).
$$
\elm
\begin{proof}[Proof of Lemma \ref{l.lastest}]
\begin{align*}
\left||Df|_{v_a}(a)|-|Df|_{v_b}(b)|\right| & \le \left||Df|_{v_a}(a)|-|Df|_{v_b}(a)|\right| + \left||Df|_{v_b}(a)|-|Df|_{v_b}(b)|\right| \\
& \le \|Df(a)\|\cdot |v_a-v_b|+\|f\|_{C^2}\cdot |a-b| \\
& \le \|f\|_{C^2}(|v_a-v_b|+|a-b|) \\
& \le \|f\|_{C^2}(\angle(v_a,v_b)+|a-b|).
\end{align*}
\end{proof}

Now,  Proposition \ref{p.distortion} follows from Lemma \ref{l.prelim}.
\end{proof}

\section{The Integrated Density of States}\label{s.ids}

\subsection{Definition and Basic Properties}

Recall the definition of $N(E,V)$ given in \eqref{e.idsdef},
$$
N(E,V) = \lim_{n \to \infty} \frac{1}{n} N(E, \omega , V , [1,n] ).
$$

\begin{Prop}[Hof, see \cite{Ho}]
For every $(E,V) \in \R^2$, the limit in \eqref{e.idsdef} exists
for every $\omega \in \T$ and its value does not depend on it.
\end{Prop}


The following proposition collects some well-known properties of
the integrated density of states.

\begin{Prop}
{\rm (a)} The map
$$
\R \times \R \ni (E,V) \mapsto N(E,V) \in [0,1]
$$
is continuous.
\\[1mm]
{\rm (b)} For every $V \in \R$, there is a Borel measure on $\R$,
called the density of states measure and denoted by $dN_V$, such
that
$$
N(E,V) = \int_\R \chi_{(-\infty,E]} \, dN_V.
$$
{\rm (c)} The topological support of the measure $dN_V$ is equal to $\Sigma_V$.
\\[1mm]
{\rm (d)} The density of states measure is the $\omega$-average of
the spectral measure associated with $H_{V , \omega}$ and the
vector $\delta_0 \in \ell^2(\Z)$. That is, for every $V \in \R$
and every bounded measurable $g : \R \to \R$,
\begin{equation}\label{e.idsformula}
\int_\R g \, dN_V = \int_{\omega \in \T} \langle \delta_0 , g(H_{V , \omega}) \delta_0 \rangle \, d\omega.
\end{equation}
{\rm (e)} We have
$$
N(E,0) = \begin{cases} 0 & E \le -2 \\ \frac{1}{\pi} \arccos \left( - \frac{E}{2} \right) & -2 < E < 2 \\ 1 & E \ge 2. \end{cases}
$$
\end{Prop}

\begin{proof}
(a) This follows from (the proof of) Lemma~3.1 and Theorem~3.2 in \cite{AS}.\\
(b) For every $V \in \R$, the map
$$
\R \ni E \mapsto N(E,V) \in [0,1]
$$
is continuous by (a) and non-decreasing by construction, and hence it is the distribution function of a Borel measure on $\R$.\\
(c) and (d) See \cite[Section~9.2]{CFKS}.\\
(e) This is folklore; see, for example, \cite[Theorem~1.1]{LS} and
its discussion there for a simple derivation.
\end{proof}

\subsection{Complete Gap Labeling}\label{ss.cgl}

Here we prove the following result, which implies
Theorem~\ref{t.completegaplabeling} since the transversality
assumption holds for $V_0 > 0$ sufficiently small.

\begin{Thm}\label{t.completegaplabeling2}
Suppose $V_0 > 0$ is such that for every $V \in (0,V_0]$ and every
point in $\Omega_V$, its stable manifold intersects $\ell_V$
transversally. Then, for every $V \in (0,V_0]$, all gaps allowed
by the gap labeling theorem are open. That is,
$$
\{ N(E,V) : E \in \R \setminus \Sigma_V \} = \{ \{ k \alpha \} : k \in \Z \} \cup \{ 1 \}.
$$
\end{Thm}

\begin{proof}
Consider the preimages of the singularities of the trace map $F^{-1}(P_i),
i=1, 2, 3, 4,$ on the torus. They form a set of 4 periodic points
of the hyperbolic automorphism $\mathcal{A}:\mathbb{T}^2\to
\mathbb{T}^2$, and the stable manifolds of those periodic points
intersect the line $\{\phi=0\}$ transversally at the points
$\{k\alpha\ (\mod 1\ )\}$,  $\{k\alpha+\frac{1}{2} \ (\mod 1\
)\}$,  $\{k\alpha +\frac{\alpha}{2}\ (\mod 1\ )\}$, and $\{k\alpha
+\frac{1}{2}+\frac{\alpha}{2}\ (\mod 1\ )\}$.

The images of these points under the semiconjugacy $F$ form the
set of points on $\ell_0$ of the form $(\pm\cos(\pi m\alpha),
\pm\cos(\pi m\alpha), 1 )$, and they correspond to the energies
$E\in \{\pm\cos(\pi m\alpha), m\in \mathbb{Z}\}$. The integrated
density of states for the free Laplacian, $N(E,0)$, takes the values
$\{ \{ m\alpha \} : m\in \mathbb{Z}\}$ at these energies.

After we increase the value of the coupling constant, each
singularity splits into two periodic points, and each
of the stable manifolds of the singularities splits into two strong
stable manifolds of the periodic points. Every point between the
stable manifolds has an unbounded positive semiorbit, and
therefore the interval that those manifolds cut in the line
$\ell_V$ corresponds to a gap in the spectrum.

Due to the continuous dependence of $N(E,V)$ on the coupling
constant and the local constancy of $N(\cdot,V)$  in the
complement of $\Sigma_V$, the integrated density of states takes
the same value in the formed gap as at the energy that corresponds
to the initial point of intersection of the stable manifold of
singularity with $\ell_0$.
\end{proof}

\subsection{More on the Asymptotic Gap Lengths}\label{ss.agl}

\begin{proof}[Proof of Theorem~\ref{t.gaplimit}.]

Fix any $m\in \mathbb{Z}\backslash \{0\}$. The integrated density
of states of the free Laplacian takes the values $\{\pm m\alpha\}$ at
the energies $\{\pm \cos \pi m \alpha\}$.

If $m=2k$, then these energies correspond to points with
$\theta$-coordinates $k\alpha (\mod 1)$ and $k\alpha
+\frac{1}{2}(\mod 1)$ on $F^{-1}(\ell_0)$.

If $m=2k+1$, then these energies correspond to points with
$\theta$-coordinates $k\alpha+\frac{\alpha}{2} (\mod 1)$ and
$k\alpha +\frac{\alpha}{2}+\frac{1}{2}(\mod 1)$ on
$F^{-1}(\ell_0)$.

Take one of these points, $Q_k\in F^{-1}(\ell_0)$. Let $P^*$ be the
singularity such that $Q_k\in F^{-1}(W^{ss}(P^*))$. Denote
$\Gamma=F^{-1}(W^{ss}_1(P^*))$. Let $M\in \mathbb{N}$ be the
smallest number such that $\mathcal{A}^{-M}(\Gamma)$ contains $Q_k$.

Then $\mathcal{A}^{M}(F^{-1}(\ell_0))$ intersects $\Gamma$ at some point
$Z$, and the distance from $Z$ to the set of singularities is
uniformly (in $|m|$) bounded from zero.

Denote by $P'(V)$ and $P''(V)$ the periodic points on $S_V$ near
$P^*$ (of period $2$ or of period $6$, depending on $P^*$).
Denote by $\Gamma'=F_V^{-1}(W_1^{ss}(P'(V))$ and
$\Gamma''=F_V^{-1}(W_1^{ss}(P''(V))$.

Then $\widetilde{T}_V^{M}(F_V^{-1}(\ell_V))$ (recall that $\widetilde{T}_V$ was defined by (\ref{e.Ttilde})) intersects $\Gamma'$
and $\Gamma''$ transversally near the point $Z$ and the curves
$\Gamma'$ and $\Gamma''$ cut in
$\widetilde{T}_V^{M}(F_V^{-1}(l_V))$ an interval $I_V$ whose
length is of order $V$. In other words, $\lim_{V\to
0}\frac{|I_V|}{|V|}$ exists and is uniformly bounded from zero and
from above.

We also have
$$
\lim_{V\to
0}\frac{\widetilde{T}_V^{-M}(I_V)}{|I_V|}=\left|D\mathcal{A}^M|_{F^{-1}(l_0)}\right|^{-1}=\alpha^M\cdot
C(Q_k),
$$
where $C(Q_k)$ is bounded from zero and from above.

Notice that if $\mathcal{A}^{-M}(\Gamma)$ intersects $F^{-1}(\ell_0)$
at $Q_k$, then $\mathcal{A}^{-M}(\Gamma)$ must be of length at least
$|k|\alpha^{-1}$. On the other hand,
$|\mathcal{A}^{-M}(\Gamma)|=\alpha^{-M}|\Gamma|$. Therefore,
$\alpha^{-M}|\Gamma|\ge |k|\alpha^{-1}$. Hence
$$
\lim_{V\to 0}\frac{|\widetilde{T}_V^{-M}(I_V)|}{|V|}=\lim_{V\to
0}\frac{|\widetilde{T}_V^{-M}(I_V)|}{|I_V|}\frac{|I_V|}{|V|}=\alpha^MC(Q_k)\lim_{V\to
0}\frac{|I_V|}{|V|}\le \overline{C}(Q_k)\frac{1}{|k|}.
$$
On the other hand,
$$
\lim_{V\to
0}\frac{|U_m(V)|}{|\widetilde{T}_V^{-M}(I_V)|}=|DF|_{F^{-1}(l_0)}(Q_k)|
$$
is bounded from above (since $\|DF\|$ is bounded). Therefore we have
$$
\lim_{V\to 0}\frac{|U_m{V}|}{|V|}=\lim_{V\to
0}\frac{|\widetilde{T}_V^{-M}(I_V)|}{|V|}\frac{|U_m(V)|}{|\widetilde{T}_V^{-M}(I_V)|}\le
\overline{\overline{C}}(Q_k)\frac{1}{|k|},
$$
where $\overline{\overline{C}}(Q_k)$ is uniformly bounded from above.
\end{proof}

\brm Notice that one can actually claim a bit more. Namely, for
those gaps that appear away from the endpoints of the free spectrum, the
corresponding constant $C_m$ is uniformly bounded away from zero.
This follows from the fact that away from the singularities, the
differential of $F$ has norm which is bounded away from zero. \erm

\section{Spectral Measures and Transport Exponents}\label{s.smte}

\subsection{Solution Estimates}

In this subsection we study solutions to the difference equation and
prove upper and lower bounds for them. Results of this kind are
known (we provide the references below), but our purpose here is to obtain explicit quantitative
estimates, as functions of the coupling constant, as they enter
explicitly in the bounds on fractal dimensions of spectral
measures and transport exponents and we wish to prove the best
bounds possible for the latter quantities to get an idea about
their behavior at weak coupling. These applications will be
discussed in the next two subsections.

Denote the largest root of the polynomial $x^3 - (2+V) x - 1$ by
$a_V$. For small $V > 0$, we have $a_V \approx
\frac{\sqrt{5}+1}{2} + cV$ with a suitable constant $c$. Our goal
is to prove the following pair of theorems (recall that
$\sigma(H_{V,\omega}) = \Sigma_V$ for every $\omega$).

Let  $M(n,m,\omega,E)$ be the standard transfer matrix, that is, the $\mathrm{SL}(2,\R)$ matrix that maps $(u(m+1),u(m))^T$ to $(u(n+1),u(n))^T$ for every solution $u$ of the difference equation $H_{V,\omega} u = Eu$.

\begin{Thm}\label{t.tmest}
For every $V > 0$ and every
\begin{equation}\label{e.zetadef}
\zeta > \frac{\log [(5 + 2V)^{1/2} (3 + V) a_V]}{\log
\frac{\sqrt{5} + 1}{2}},
\end{equation}
there is a constant $C > 0$ such that for every $\omega$ and every $E \in \Sigma_V$, we have
\begin{equation}\label{f.tmest}
\max_{0 \le |n-m| \le N} \|M(n,m,\omega,E)\| \le C N^\zeta.
\end{equation}
\end{Thm}

\begin{Thm}\label{t.solest}
For every $V > 0$,
$$
\gamma_\ell < \frac{\log \left( 1 + \frac{1}{(2+2V)^2} \right)}{16
\cdot \log \left( \frac{\sqrt{5} + 1}{2} \right)}, \quad \text{
and } \quad \gamma_u > 1 + \frac{\log [(5 + 2V)^{1/2} (3 + V)
a_V]}{\log \frac{\sqrt{5} + 1}{2}},
$$
there are constants $C_\ell, C_u > 0$ such that for every $\omega$, $L \ge 1$,
and $E \in \Sigma_V$, we have that every solution to the
difference equation $H_{V,\omega} u = E u$, which is normalized by
$|u(0)|^2 + |u(1)|^2 = 1$, obeys the estimates
\begin{equation}\label{f.solest}
C_\ell L^{\gamma_\ell} \le \|u\|_L \le C_u L^{\gamma_u},
\end{equation}
where the local $\ell^2$ norm $\| \cdot \|_L$ is defined by
$$
\|u\|_L^2 = \sum_{n=1}^{[L]} |u(n)|^2 + (L-[L]) |u([L]+1)|^2.
$$
for $L \ge 1$.
\end{Thm}

Theorem~\ref{t.tmest} is a quantitative version of
\cite[Theorem~3]{DL99a}, which in turn was a generalization of
\cite[Theorem~1]{IT}. Theorem~\ref{t.solest}, on the other hand,
is a quantitative version of \cite[Propositions~5.1 and 5.2]{DKL}.
The upper bound in Theorem~\ref{t.solest} is a consequence of
Theorem~\ref{t.tmest}, while the lower bound in
Theorem~\ref{t.solest} will be extracted from the details of the
proof of \cite[Proposition~5.1]{DKL}.

Let us recall some notation from \cite{DL99a}. For $V > 0$, $E \in
\R$, and $a \in \{0,1\}$, we write
$$
T(V,E,a) = \begin{pmatrix} E - V a & - 1 \\ 1 & 0 \end{pmatrix}.
$$
For a word $w = w_1 \ldots w_n \in \{ 0,1 \}^*$, we then set
$$
M(V,E,w) = T(V,E,w_n) \times \cdots \times T(V,E,w_1).
$$
Next, define the words $\{s_n\}_{n \ge 0}$ inductively by
$$
s_0 = 0 , \quad s_1 = 1 , \quad s_n = s_{n-1} s_{n-2} \text{ for }
n \ge 2.
$$
By the definition of these words, there is a unique $u \in \{0,1\}^{\Z_+}$ (the fixed point of the Fibonacci substitution) that has $s_n$ as a prefix for every $n \ge 1$; namely $u = 1011010110110\ldots$. We denote the set of finite subwords of $u$ by $\mathcal{W}_u$, that is, $\mathcal{W}_u = \{ 0 , 1 , 01 , 10 , 11 , 010 , 011 , 101 , 110 , \ldots \}$. By uniform recurrence of $u$ it suffices to consider the matrices $M(V,E,w)$ for $E \in \Sigma_V$ and $w \in \mathcal{W}_u$ when proving Theorems~\ref{t.tmest} and \ref{t.solest}.

With these quantities, we have
$$
M_n = M_n(V,E) = M(V,E,s_n)
$$
and consequently
$$
x_n = x_n(V,E) = \frac12 \mathrm{Tr} M(V,E,s_n)
$$
for $n \ge 1$.

Finding upper bounds for the norm of transfer matrices consists of
three steps. The first step is to bound the special matrices
$M_n$, that is, $M(F_n,0,\omega=0,E)$. This is the objective of
Lemma~\ref{l.b1}. Then, in Lemma~\ref{l.b2}, we use interpolation
to bound the norm of the matrices $M(n,0,\omega=0,E)$. Finally,
the case of general matrices $M(n,m,\omega,E)$ is treated using
partition. This will then complete the proof of
Theorem~\ref{t.tmest}.

We begin with the first step. Part~(a) of the following lemma is
due to S\"ut\H{o} (see \cite[Lemma~2]{S87}) and part~(b) is a
relative of \cite[Lemma~4.(ii)]{IT}:

\begin{Lm}\label{l.b1}
{\rm (a)} We have $E \in \Sigma_V$ if and only if the sequence
$\{x_n\}$ is bounded. Moreover, we have
$$
|x_n| \le 1 + \frac{V}{2}
$$
for every $E \in \Sigma_V$ and $n \ge 1$.
\\
{\rm (b)} With some positive $V$-dependent constant $C$, we have
$$
\|M_n\| \le C a_V^n
$$
for every $E \in \Sigma_V$ and $n \ge 1$.
\end{Lm}

\begin{proof}
As pointed out above, part~(a) is known. Let us prove part~(b). By the Cayley-Hamilton Theorem, we have that
$M_n^2 - 2 x_n M_n + I = 0$,
that is,
$$
M_n = 2x_n I - M_n^{-1}.
$$
The recursion for the matrices $M_n$, $M_n = M_{n-2} M_{n-1}$, gives
$$
M_{n-2}^{-1} = M_{n-1} M_n^{-1}.
$$
Putting these things together, we find
$$
M_n = M_{n-2} M_{n-1} = M_{n-2}(2x_{n-1} I - M_{n-1}^{-1}) = 2 x_{n-1} M_{n-2} - M_{n-3}^{-1}.
$$
Since we also have $\|M_n\| = \|M_n^{-1}\|$, we obtain from this identity along with part~(a) the estimate
$$
\|M_n\| \le (2 + V) \|M_{n-2}\| + \|M_{n-3}\|
$$
for every $E \in \Sigma_V$.

Consider for comparison the recursion
$$
m_n = (2 + V) m_{n-2} + m_{n-3}.
$$
It is clear that we have $\|M_n\| \le m_n$ if we consider the
sequence $\{ m_n \}$ generated by the recursion and the initial
conditions $m_1 = \|M_1\|$, $m_2 = \|M_2\|$, $m_3 = \|M_3\|$.

Any solution of the recursion is of the form
$$
m_n = c_1 x_1^n + c_2 x_2^n + c_3 x_3^n,
$$
where the $x_j$ are the roots of the characteristic polynomial
$x^3 - (2+V) x - 1$. Thus, the definition of $a_V$ implies that
the estimate in (b) holds.
\end{proof}

\noindent\textit{Remark.} Since $F_n \sim
(\frac{\sqrt{5}+1}{2})^n$, we can infer that $\|M_n\| \le C F_n^{1
+ \delta}$ with $\delta \to 0$ as $V \to 0$. Since for $V = 0$,
the norm of the transfer matrices grows linearly at the energies
$E = \pm 2$, we cannot expect a better result in general. However,
since the transfer matrices are bounded when $V = 0$ and $E \in
(-2,2)$ it is reasonable to expect that there is a better bound
for small values of $V$ for most energies in the spectrum.

\medskip

Let us now turn to the second step, which is the interpolation of
the estimates from the previous lemma to non-Fibonacci sites in
the $\omega = 0$ sequence. The following lemma is a relative of
\cite[Lemma~5]{IT}. Even though the proof is closely related to
that in \cite{IT}, we give the details to clearly show where the
improved estimate comes from.

\begin{Lm}\label{l.b2prep}
For $n \ge 1$ and $k \ge 0$, we may write
$$
M_n M_{n+k} = P_k^{(1)} M_{n+k} + P_k^{(2)} M_{n + k - 1} + P_k^{(3)} M_{n + k - 2} + P_k^{(4)} I
$$
with coefficients $P_k^{(j)}$ that also depend on $n$ and $E$. Moreover, for every $n \ge 1$ and $E \in \Sigma_V$, we have that
$$
\sum_{j = 1}^4 |P_k^{(j)}| \le (5 + 2V) (3 + V)^{\lfloor k/2 \rfloor}.
$$
\end{Lm}

\begin{proof}
Consider the case $k = 0$. Then,
$$
M_n M_n = M_n^2 = 2 x_n M_n - I
$$
by Cayley-Hamilton and hence we can set
$$
P_0^{(1)} = 2x_n, \quad P_0^{(2)} = P_0^{(3)} = 0 , \quad P_0^{(4)} = -1.
$$
It follows that for $E \in \Sigma_V$, we have the estimate
$$
\sum_{j = 1}^4 |P_0^{(j)}| \le 3 + V.
$$

Consider now the case $k = 1$. Then,
$$
M_n M_{n+1} = M_{n+2} = 2 x_{n+1} M_n + M_{n-1} - 2 x_{n-1} I
$$
by \cite[Lemma~6]{IT} and hence we can set
$$
P_1^{(1)} = 0, \quad P_1^{(2)} = 2 x_{n+1} , \quad P_1^{(3)} = 1 , \quad P_1^{(4)} = - 2 x_{n-1}.
$$
For $E \in \Sigma_V$, we therefore have
$$
\sum_{j = 1}^4 |P_1^{(j)}| \le 5 + 2V.
$$

Next, assume that the lemma holds for $k \in \{ 1 , \ldots , \ell
\}$ and consider the case $k = \ell + 1$. Then, using
\cite[Lemma~6]{IT} one more time, we find
\begin{align*}
M_n M_{n + \ell + 1} & = ( M_n M_{n + \ell - 1}) M_{n + \ell} \\
& = \left( P_{\ell - 1}^{(1)} M_{n+\ell - 1} + P_{\ell - 1}^{(2)}
M_{n + \ell - 2} + P_{\ell - 1}^{(3)} M_{n + \ell - 3} + P_{\ell - 1}^{(4)}
I \right) M_{n + \ell} \\
& = P_{\ell - 1}^{(1)} M_{n+\ell + 1} + P_{\ell - 1}^{(2)} (2 x_{n + \ell - 2}
M_{n + \ell} - M_{n + \ell - 1}) + \\ & \qquad P_{\ell - 1}^{(3)} ( 2x_{n + \ell - 1}
M_{n + \ell - 1} - I) + P_{\ell - 1}^{(4)} M_{n + \ell} \\
& = P_{\ell + 1}^{(1)} M_{n+\ell + 1} + P_{\ell + 1}^{(2)} M_{n + \ell} + P_{\ell + 1}^{(3)}
M_{n + \ell - 1} + P_k^{(4)} I,
\end{align*}
where we set
\begin{align*}
P_{\ell + 1}^{(1)} & = P_{\ell - 1}^{(1)} \\
P_{\ell + 1}^{(2)} & = 2 x_{n + \ell - 2} P_{\ell - 1}^{(2)} + P_{\ell - 1}^{(4)}\\
P_{\ell + 1}^{(3)} & = - P_{\ell - 1}^{(2)} + 2x_{n + \ell - 1} P_{\ell - 1}^{(3)}\\
P_{\ell + 1}^{(4)} & = - P_{\ell - 1}^{(3)}.
\end{align*}
It follows that
$$
\sum_{j = 1}^4 |P_{\ell + 1}^{(j)}| \le (3 + V) \sum_{j = 1}^4 |P_{\ell - 1}^{(j)}|.
$$
By induction hypothesis, we obtain the desired estimate.
\end{proof}

\begin{Lm}\label{l.b2}
For every $V > 0$ and every $\zeta$ obeying \eqref{e.zetadef}, there is a constant $\tilde C$ so that
$$
\| M(n,0,\omega=0,E) \| \le \tilde C n^{\zeta}
$$
for every $n \ge 1$, $E \in \Sigma_V$.
\end{Lm}

Before we give the proof of Lemma~\ref{l.b2}, let us recall some
facts about the minimal representation of positive integers in
terms of Fibonacci numbers; compare, for example, \cite{HCB,Le,Z}.
Given $n \ge 1$, there is a unique representation
$$
n = \sum_{k = 0}^K F_{n_k}
$$
such that $n_k \in \Z_+$ and $n_{k+1} - n_k \ge 2$. This is the
so-called Zeckendorf representation of $n$ (which, incidentally,
was published by Lekkerkerker some twenty years before Zeckendorf). It
is found by the greedy algorithm, that is, $F_{n_K}$ is set to be
the largest Fibonacci number less than or equal to $n$, $n$ is
replaced by $n - F_{n_K}$, and the process is repeated until we
have a zero remainder. The property $n_{k+1} - n_k \ge 2$ follows
from the way the algorithm works together with the recursion for
the Fibonacci numbers. Uniqueness of this representation in turn
follows from $n_{k+1} - n_k \ge 2$.

By construction, we have
\begin{equation}\label{e.zecken1}
F_{n_K} \le n < F_{n_K + 1} \le 2 F_{n_K}.
\end{equation}
We will also need a relation between $n$ and $K$, that is, the
number of terms in the Zeckendorf representation of $n$. The
algorithm to compute $K = K(n)$ is the following. Start with two
words over the positive integers, $1$ and $1$. The next word is
obtained by writing down the previous word, following by the word
preceding the previous word, but this time with every symbol
increased by one. Now iterate this procedure:
$$
1 , \quad 1 , \quad 12 , \quad 122 , \quad 12223 , \quad 12223233 , \quad 1222323323334 , \ldots
$$
Concatenate all the words to obtain a semi-infinite word over the positive integers:
$$
111212212223122232331222323323334 \ldots
$$
The $n$-th symbol in this semi-infinite sequence is $K(n)$. Notice
that the lengths of these words follow the Fibonacci sequence.
Moreover, only every other word has an increase in the maximum
value of its entries, relative to the previous one, and the
increase is by one whenever it happens. This shows that
\begin{equation}\label{e.zecken2}
F_{2(K - 1)} \le n
\end{equation}

\begin{proof}[Proof of Lemma~\ref{l.b2}.]
From the Zeckendorf representation of $n \ge 1$, one finds that
$$
M(n,0,\omega=0,E) = M_{n_0} M_{n_1} \cdots M_{n_K};
$$
compare the beginning of the proof of \cite[Theorem~1]{IT}.

Let us prove by induction that with
$$
b_V = (3 + V)^{1/2},
$$
we have for every $k \ge 1$,
\begin{equation}\label{e.est}
\| M_{n_0} M_{n_1} \cdots M_{n_k} \| \le C (5 + 2V)^k b_V^{-n_0 +
2(k-1)} \left(a_Vb_V\right)^{n_k}.
\end{equation}

Consider first the case $k = 1$. From Lemmas~\ref{l.b1} and \ref{l.b2prep}, we may infer that
\begin{align*}
\|M_{n_0} M_{n_1}\| & \le |P_{n_1 - n_0}^{(1)}| \cdot \|M_{n_1}\| + |P_{n_1 - n_0}^{(2)}| \cdot
\|M_{n_1 - 1}\| + |P_{n_1 - n_0}^{(3)}| \cdot \|M_{n_1 - 2}\| + |P_{n_1 - n_0}^{(4)}| \\
& \le \sum_{j = 1}^4 |P_{n_1 - n _0}^{(j)}| C a_V^{n_1} \\
& \le C (5 + 2V) (3 + V)^{\lfloor (n_1 - n_0)/2 \rfloor} a_V^{n_1} \\
& \le C (5 + 2V) b_V^{-n_0} \left(a_Vb_V\right)^{n_1}.
\end{align*}
Thus, the estimate \eqref{e.est} holds for $k = 1$.

Now assume that \eqref{e.est} holds for $k \in \{ 1 , \ldots ,
\ell \}$ and consider the case $k = \ell + 1$. Using
Lemma~\ref{l.b2prep} again, we see that
\begin{align*}
\| M_{n_0} M_{n_1} \cdots M_{n_\ell} M_{n_{\ell + 1}} \| & = \| (M_{n_0} M_{n_1}) M_{n_2} \cdots M_{n_\ell} M_{n_{\ell + 1}} \| \\
& \le |P_{n_1 - n_0}^{(1)}| \cdot \|M_{n_1} M_{n_2} \cdots M_{n_\ell} M_{n_{\ell + 1}}\| \\
& \quad + |P_{n_1 - n_0}^{(2)}| \cdot \|M_{n_1 - 1} M_{n_2} \cdots M_{n_\ell} M_{n_{\ell + 1}}\| \\
& \quad + |P_{n_1 - n_0}^{(3)}| \cdot \|M_{n_1 - 2} M_{n_2} \cdots M_{n_\ell} M_{n_{\ell + 1}}\| \\
& \quad + |P_{n_1 - n_0}^{(4)}| \cdot \|M_{n_2} \cdots M_{n_\ell} M_{n_{\ell + 1}}\|\\
& \le |P_{n_1 - n_0}^{(1)}| \cdot C (5 + 2V)^\ell b_V^{-n_1 + 2(\ell-1)} (ab)^{n_{\ell + 1}}\\
& \quad + |P_{n_1 - n_0}^{(2)}| \cdot C (5 + 2V)^\ell b_V^{-n_1 + 1 + 2(\ell-1)} \left(a_Vb_V\right)^{n_{\ell + 1}}\\
& \quad + |P_{n_1 - n_0}^{(3)}| \cdot C (5 + 2V)^\ell b_V^{-n_1 + 2 + 2(\ell-1)} \left(a_Vb_V\right)^{n_{\ell + 1}}\\
& \quad + |P_{n_1 - n_0}^{(4)}| \cdot C (5 + 2V)^\ell b_V^{-n_2 + 2(\ell-2)} \left(a_Vb_V\right)^{n_{\ell + 1}}\\
& \le \left( \sum_{j = 1}^4 |P_{n_1 - n _0}^{(j)}| \right) C (5 + 2V)^\ell b_V^{-n_1 + 2\ell} \left(a_Vb_V\right)^{n_{\ell + 1}}\\
& \le (5 + 2V) b^{n_1 - n_0} C (5 + 2V)^\ell b_V^{-n_1 + 2\ell} \left(a_Vb_V\right)^{n_{\ell + 1}} \\
& = C (5 + 2V)^{\ell + 1} b_V^{-n_0 + 2((\ell + 1) - 1)}
\left(a_Vb_V\right)^{n_{\ell + 1}}.
\end{align*}
We conclude that \eqref{e.est} holds for $k = \ell + 1$.

Therefore,
$$
\|M(n,0,\omega=0,E)\| \le C (5 + 2V)^K b_V^{-n_0 + 2(K-1)}
\left(a_Vb_V\right)^{n_K}.
$$
It follows that
\begin{align*}
\limsup_{n \to \infty} & \frac{\log \|M(n,0,\omega=0,E)\|}{\log n} \le \limsup_{n \to \infty}
\frac{\log \left( C (5 + 2V)^K b_V^{-n_0 + 2(K-1)} \left(a_Vb_V\right)^{n_K} \right)}{\log n} \\
& \le \limsup_{n \to \infty} \frac{K \log (5 + 2V) + 2K \log b_V + n_K \log \left(a_Vb_V\right)}{\log n} \\
& \le \limsup_{n \to \infty} \frac{\frac{\log n}{2 \log
\frac{\sqrt{5} + 1}{2}} \log (5 + 2V) + 2 \frac{\log n}{2 \log
\frac{\sqrt{5} + 1}{2}} \log b_V + \frac{\log n}{\log
\frac{\sqrt{5} + 1}{2}}
\log \left(a_Vb_V\right)}{\log n} \\
& = \frac{\log (5 + 2V)^{1/2} + \log b_V + \log \left(a_Vb_V\right)}{\log \frac{\sqrt{5} + 1}{2}} \\
& = \frac{\log [(5 + 2V)^{1/2} (3 + V) a_V]}{\log \frac{\sqrt{5} +
1}{2}},
\end{align*}
where we used \eqref{e.zecken1} and \eqref{e.zecken2} in the third step. Since these estimates are uniform in $E \in \Sigma_V$, the result follows.
\end{proof}

In the third step, we can now turn to the

\begin{proof}[Proof of Theorem~\ref{t.tmest}.]
Suppose $V > 0$ and $\zeta$ obeys \eqref{e.zetadef}. We will prove the following estimate,
\begin{equation}\label{f.tmest2}
\|M(V,E,w)\| \le C |w|^\zeta \quad \text{ for every } V > 0, \; E \in \Sigma_V, \text{ and } w \in \mathcal{W}_u,
\end{equation}
from which \eqref{f.tmest} immediately follows.

Let $V,E,w$ as in \eqref{f.tmest2} be given. As explained in \cite[Proof of Lemma~5.2]{DL99a}, we can write
$$
w = xyz,
$$
where $y \in \{a,b\}^*$ is a word of length $2$ and $x^R$ (the reversal of $x$) and $z$ are prefixes of $u$. Thus,
\begin{equation}\label{e.tmest3}
\|M(V,E,w)\| \le \|M(V,E,x)\| \cdot \|M(V,E,y)\| \cdot \|M(V,E,z)\|.
\end{equation}
It follows from \cite[Lemma~5.1]{DL99a} that
\begin{equation}\label{e.tmest4}
\|M(V,E,x^R)\| = \|M(V,E,x)\| .
\end{equation}
Moreover, Lemma~\ref{l.b2} yields
\begin{equation}\label{e.tmest5}
\|M(V,E,x^R)\| \le \tilde C |x|^\zeta
\end{equation}
and
\begin{equation}\label{e.tmest6}
\|M(V,E,z)\| \le \tilde C |z|^\zeta.
\end{equation}
Since $\zeta > 1$ and $y$ has length $2$, \eqref{f.tmest2} follows from \eqref{e.tmest3}--\eqref{e.tmest6}.
\end{proof}

\begin{proof}[Proof of Theorem~\ref{t.solest}.]
The upper bound in \eqref{f.solest} follows from
Theorem~\ref{t.tmest}. The lower bound in \eqref{f.solest} will be
extracted here from \cite{DKL} since it is not made explicit
there.

Denote $U(n) = (u(n+1) , u(n))^T$ and consider the associated
local $\ell^2$-norms $\|U\|_L$. Power-law bounds for $\|U\|_L$
correspond in a natural way to power-law bounds for $\|u\|_L$, so
let us discuss the former object. By considering squares adjacent
to the origin and Gordon's two-block method, Damanik, Killip, and
Lenz showed that
$$
\|U\|_{F_{8n}} \ge \left( 1 + \frac{1}{(2+2V)^2} \right)^{n/2};
$$
see \cite[Lemma~4.1]{DKL}. We find
$$
\liminf_{n \to \infty} \frac{\log \|U\|_{F_{8(n-1)}}}{\log F_{8n}}
\ge \frac{\log \left( 1 + \frac{1}{(2+2V)^2} \right)}{16 \cdot
\log \left( \frac{\sqrt{5} + 1}{2} \right)},
$$
uniformly in $\omega \in \T$ and $E \in \Sigma_V$. This, together
with monotonicity, yields the asserted lower bound.
\end{proof}

\subsection{Spectral Measures}

Given $V > 0$ and $\omega \in \T$, let us consider the operator
$H_{V,\omega}$. Since this operator is self-adjoint, the spectral
theorem associates with each $\psi \in \ell^2(\Z)$ a Borel measure
$\mu_\psi$ that is supported by $\Sigma_V$ and obeys
$$
\left\langle \psi , g\left(H_{V,\omega}\right) \psi \right\rangle
= \int g (E) \, d\mu_\psi(E)
$$
for every bounded measurable function $g$ on $\Sigma_V$. It
follows from Theorem~\ref{t.solest} that each of these spectral
measures gives zero weight to sets of small Hasudorff dimension.
More precisely, we have the following result.

\begin{Cor}\label{c.alphacont}
For every $V > 0$, every $\omega$, and every
$$
\alpha < \frac{2 \log \left( 1 + \frac{1}{(2+2V)^2} \right)}{\log
\left( 1 + \frac{1}{(2+2V)^2} \right) + 16 \cdot \log \left(
\frac{\sqrt{5} + 1}{2} \right) + 16 \log [(5 + 2V)^{1/2}(3+V)
a_V]},
$$
we have that $\mu_\psi(S) = 0$ for every $\psi \in \ell^2(\Z)$ and
every Borel set $S$ with $h^\alpha(S) = 0$.
\end{Cor}

\begin{proof}
This is an immediate consequence of Theorem~\ref{t.solest} and
\cite[Theorem~1]{DKL}.
\end{proof}

\subsection{Transport Exponents}

Absolute continuity of spectral measures with respect to Hausdorff
measures is important because it implies lower bounds for
transport exponents via the Guarneri-Combes-Last Theorem. Let us
recall this connection briefly. Given $H_{V,\omega}$ and $\psi$ as
above, consider the Schr\"odinger equation $i \partial_t \psi(t) =
H_{V,\omega} \psi(t)$ with initial condition $\psi(0) = \psi$.
Then, the unique solution of this initial-value problem is given
by $\psi(t) = e^{-itH_{V,\omega}} \psi$. To measure the spreading
of a wavepacket, one considers time-averaged moments of the
position operator. This is of course mainly of interest when
$\psi(0) = \psi$ is well-localized and in fact, in $\ell^2(\Z)$
one usually considers the canonical initial state $\psi =
\delta_0$.

One is interested in the growth of $\langle \langle |X|_\psi^p
\rangle \rangle(T)$ as $T \to \infty$ for $p > 0$, where
$$
\langle |X|_\psi^p \rangle(t) = \sum_{n \in \Z} |n|^p |\langle
\delta_n , e^{-itH_{V,\omega}} \psi \rangle |^2
$$
and the time average of a $t$-dependent function $f$ is given by
either
\begin{equation}\label{e.cesaro}
\langle f \rangle(T) = \frac{1}{T} \int_0^T f(t) \, dt
\end{equation}
or
\begin{equation}\label{e.abelian}
\langle f \rangle(T) = \frac{2}{T} \int_0^\infty e^{-2t/T} f(t) \,
dt.
\end{equation}
We will indicate in the results below which time-average is
involved. However, for compactly supported (or fast-decaying)
$\psi$, all the results mentioned below hold for both types of
time-average.

To measure the power-law growth of $\langle \langle |X|^p \rangle
\rangle(T)$, define
$$
\beta_\psi^+ (p) = \limsup_{T \to \infty} \frac{\log \langle
\langle |X|_\psi^p \rangle \rangle(T)}{p \log T}
$$
and
$$
\beta_\psi^- (p) = \liminf_{T \to \infty} \frac{\log \langle
\langle |X|_\psi^p \rangle \rangle(T)}{p \log T}.
$$

The Guarneri-Combes-Last Theorem (see \cite[Theorem~6.1]{La})
states that if $\mu_\psi$ is not supported by a Borel set $S$ with
$h^\alpha(S) = 0$, then $\beta^\pm(p) \ge \alpha$ for every $p >
0$, where the time-average is given by \eqref{e.cesaro}.

\begin{Cor}\label{c.betabound}
For every $V > 0$, $\omega \in \T$, $0 \not= \psi \in \ell^2(\Z)$,
and $p > 0$, we have
$$
\beta_\psi^\pm(p) \ge \frac{2 \log \left( 1 + \frac{1}{(2+2V)^2}
\right)}{\log \left( 1 + \frac{1}{(2+2V)^2} \right) + 16 \cdot
\log \left( \frac{\sqrt{5} + 1}{2} \right) + 16 \log [(5 +
2V)^{1/2}(3+V) a_V]},
$$
where the time-average is given by \eqref{e.cesaro}.
\end{Cor}

\begin{proof}
Immediate from the Guarneri-Combes-Last Theorem and
Corollary~\ref{c.alphacont}.
\end{proof}

We also have the following estimate, which holds for a special
initial state but which is better when $p$ is large:

\begin{Cor}\label{c.betabound2}
For every $V > 0$, $\omega \in \T$, and $p > 0$, we have
$$
\beta_{\delta_0}^\pm(p) \ge \frac{\log \frac{\sqrt{5} +
1}{2}}{\log \frac{\sqrt{5} + 1}{2} + \log [(5 + 2V)^{1/2} (3 + V)
a_V]} - \frac{3 \log [(5 + 2V)^{1/2} (3 + V) a_V]}{p\left( \log
\frac{\sqrt{5} + 1}{2} + \log [(5 + 2V)^{1/2} (3 + V) a_V] \right)
},
$$
where the time-average is given by \eqref{e.abelian}.
\end{Cor}

\begin{proof}
This is a consequence of Theorem~\ref{t.tmest} and results of
Damanik, S\"ut\H{o}, and Tcheremchantsev (see
\cite[Theorem~1]{DST} and compare also \cite[Theorem~1]{DT08}).
\end{proof}

It can be shown that for compactly supported (or fast-decaying)
$\psi$, the transport exponents $\beta^\pm_\psi(p)$ are
non-decreasing functions of $p$ taking values in $[0,1]$. Thus, in
this case it is natural to consider the limits of these quantities
as $p \downarrow 0$ and $p \uparrow \infty$ and denote them by
$\beta^\pm_\psi(0)$ and $\beta^\pm_\psi(\infty)$, respectively.
The bounds in the two previous corollaries imply estimates for
$\beta^\pm_\psi(0)$ and $\beta^\pm_\psi(\infty)$ as well. However,
at small coupling, a better estimate for
$\beta^\pm_{\delta_0}(\infty)$ is obtained via a different route:

\begin{Cor}
We have
$$
\lim_{V \downarrow 0} \beta^\pm_{\delta_0}(\infty) = \lim_{V
\downarrow 0} \lim_{p \uparrow \infty} \beta^\pm_{\delta_0}(p) =
1,
$$
uniformly in $\omega \in \T$. Here, the time-average is given by
\eqref{e.abelian}.
\end{Cor}

\begin{proof}
Damanik, Embree, Gorodetski, and Tcheremchantsev showed for the
Fibonacci Hamiltonian that $\beta^\pm_{\delta_0}(\infty) \ge
\dim_B^\pm(\Sigma_V)$ for every $V > 0$ and $\omega \in \T$; see
\cite[Theorem~3]{DEGT}. Since the right-hand side converges to one
as $V \downarrow 0$, the result follows.
\end{proof}

\noindent{Remarks.} (a) As pointed out above, while the papers
\cite{DEGT, DST, DT03} work with the time-average
\eqref{e.abelian}, the results above carry over to transport
exponents defined using \eqref{e.cesaro}.
\\[1mm]
(b) We find that $\beta^\pm_{\delta_0}(\infty)$, considered as a
function of $V$, extends continuously to zero. While we would
expect this also for other initial states $\psi$, it does not
follow from Corollary~\ref{c.betabound}. It would be interesting
to extend this continuity result to all fast-decaying initial
states or to exhibit one for which it fails.

\section{Consequences for Higher-Dimensional Models}\label{s.chdm}

\subsection{Lattice Schr\"odinger Operators with Separable Potentials}

Let $d \ge 1$ be an integer and assume that for $1 \le j \le d$,
we have bounded maps $V_j : \Z \to \R$. Consider the associated
Schr\"odinger operators on $\ell^2(\Z)$,
\begin{equation}\label{f.1doper}
[H_j \psi](n) = \psi(n+1) + \psi(n-1) + V_j(n) \psi(n).
\end{equation}
Furthermore, we let $V : \Z^d \to \R$ be given by
\begin{equation}\label{f.sumpot}
V(n) = V_1(n_1) + \cdots + V_d(n_d),
\end{equation}
where we express an element $n$ of $\Z^d$ as $n = (n_1,\ldots,n_d)$ with $n_j \in \Z$.

Finally, we introduce the Schr\"odinger operator on $\ell^2(\Z^d)$ with potential $V$, that is,
\begin{equation}\label{f.2doper}
[H \psi](n) = \Big( \sum_{j = 1}^d \psi(n+e_j) + \psi(n-e_j) \Big) + V(n) \psi(n).
\end{equation}
Here, $e_j$ denotes the element $n$ of $\Z^d$ that has $n_j = 1$ and $n_k = 0$ for $k \not= j$.

Potentials of the form \eqref{f.sumpot} and Schr\"odinger
operators of the form \eqref{f.2doper} are called separable. Let
us first state some known results for separable Schr\"odinger
operators.

\begin{Prop}\label{p.products}
{\rm (a)} The spectrum of $H$ is given by
$$
\sigma(H) = \sigma(H_1) + \cdots + \sigma(H_d).
$$
{\rm (b)} Given $\psi_1, \ldots, \psi_d \in \ell^2(\Z)$, denote by
$\mu_j$ the spectral measure corresponding to $H_j$ and $\psi_j$.
Furthermore, denote by $\mu$ the spectral measure corresponding to
$H$ and the element $\psi$ of $\ell^2(\Z^d)$ given by $\psi(n) =
\psi_1(n_1) \cdots \psi_d(n_d)$. Then,
$$
\mu = \mu_1 \ast \cdots \ast \mu_d.
$$
\end{Prop}

\begin{proof}
Recall the definition and properties of tensor products of Hilbert
spaces and operators on these spaces; see, for example,
\cite[Sections~II.4 and VIII.10]{RS}. It follows from
\cite[Theorem~II.10]{RS} that there is a unique unitary map $U$
from $\ell^2(\Z) \otimes \cdots \otimes \ell^2(\Z)$ ($d$ factors)
to $\ell^2(\Z^d)$ so that for $\psi_j \in \ell^2(\Z)$, the
elementary tensor $\psi_1 \otimes \cdots \otimes \psi_d$ is mapped
to the element $\psi$ of $\ell^2(\Z^d)$ given by $\psi(n) =
\psi_1(n_1) \cdots \psi_d(n_d)$. With this unitary map $U$, we
have
$$
U^* H U = \sum_{j = 1}^d \mathrm{Id} \otimes \cdots \otimes \mathrm{Id} \otimes H_j \otimes \mathrm{Id} \otimes \cdots \otimes \mathrm{Id},
$$
with $H_j$ being the $j$-th factor. Given this representation,
part~(a) now follows from \cite[Theorem~VIII.33]{RS} (see also the
example on \cite[p.~302]{RS}). Part~(b) follows from the proof of
\cite[Theorem~VIII.33]{RS}; see also \cite{BS} and \cite{S}.
\end{proof}

\subsection{A Consequence of the Newhouse Gap Lemma}\label{ss.gaplemma}

We have seen above that the spectrum of a product model is given
by the sum of the individual spectra. If these individual spectra
are Cantor sets and we want to show that their sum is not a Cantor
set, the following consequence of the Gap Lemma is useful.

\begin{Lm}\label{p.sumofcs} Suppose $C, K\subset \mathbb{R}^1$ are
Cantor sets with $\tau(C)\cdot \tau(K)>1$. Assume also that the size of the largest gap of $C$ is not greater than the diameter of $K$, and the size of the largest gap
of $K$ is not greater than the diameter of $C$. Then the sum $C+K$
is a closed interval.
\end{Lm}

\begin{proof}
Denote $\min C=c_1$, $\max C=c_2$, $\min K=k_1$, $\max K=k_2$. Let us prove that
$$
C+K = [c_1+k_1, c_2+k_2].
$$
The inclusion ``$\subseteq$'' is obvious, so let us prove the inclusion ``$\supseteq$.''
Take an arbitrary point $x\in [c_1+k_1, c_2+k_2]$. Then, $x\in C+K$ if and only if $0\in
C+K-x=C-(x-K)$. Therefore,
$$
x\in C+K  \  \Leftrightarrow \ C\cap (x-K)\ne \emptyset.
$$
Since $\tau(C)\cdot \tau(x-K)=\tau(C)\cdot \tau(K)>1$, the Gap Lemma
implies that {\it a priori} there are only four possibilities:
\begin{enumerate}

\item the intervals $[c_1, c_2]$ and $[x-k_2, x-k_1]$ are disjoint;

\item the set $C$ is contained in a finite gap of the set $(x-K)$;

\item the set $(x-K)$  is contained in a finite gap of the set
$C$;

\item $C\cap (x-K)\ne \emptyset$.

\end{enumerate}
But the case (1) contradicts the assumption $x\in [c_1+k_1,
c_2+k_2]$, and the cases (2) and (3) are impossible due to our
assumption on the sizes of gaps and diameters of $C$ and $K$.
Therefore, we must have $C\cap (x-K)\ne \emptyset$ and hence $x\in C+K$.
\end{proof}

\subsection{The Square Fibonacci Hamiltonian}\label{ss.sfh}

Let us now discuss the (diagonal version of the) model studied by Even-Dar Mandel and
Lifshitz, namely the operator
\begin{align*}
[H^{(2)}_V \psi] (m,n) = & \psi(m+1,n) + \psi(m-1,n) + \psi(m,n+1)
+ \psi(m,n-1) + \\
& + V \left( \chi_{[1-\alpha , 1)}(m\alpha \!\!\! \mod 1) +
\chi_{[1-\alpha , 1)}(n\alpha \!\!\! \mod 1) \right) \psi(m,n)
\end{align*}
in $\ell^2(\Z^2)$. By Proposition~\ref{p.products}, we have
\begin{equation}\label{e.squarefib}
\sigma(H^{(2)}_V) = \Sigma_V + \Sigma_V.
\end{equation}

\begin{proof}[Proof of Theorem~\ref{t.4}.]
By \eqref{e.squarefib}, it suffices to show that $\Sigma_V +
\Sigma_V$ is an interval for $V > 0$ sufficiently small. By
Theorem~\ref{t.2}, there is $V_0 > 0$ such that $\tau(\Sigma_V) >
1$ for $V \in (0,V_0)$. For such $V$'s, Lemma~\ref{p.sumofcs} applies with $C = K = \Sigma_V$ and yields
$$
\Sigma_V + \Sigma_V = [ \min \Sigma_V + \min \Sigma_V , \max
\Sigma_V + \max \Sigma_V ],
$$
as desired.
\end{proof}

\begin{appendix}

\section{The Off-Diagonal Fibonacci Hamiltonian}\label{a.odfh}

The purpose of this appendix is to give complete proofs of the
basic spectral properties of the off-diagonal Fibonacci
Hamiltonian. This operator has been considered in many physics
papers and is the basic building block for the higher-dimensional
product models studied by Even-Dar Mandel and Lifshitz. The
mathematics literature on the Fibonacci model provides an
exhaustive study of the diagonal model, and it was always
understood that ``analogous results hold for the off-diagonal
model.'' For the reader's convenience, we make these analogous
results explicit here and hence provide in this paper a complete
treatment of the spectrum of the Even-Dar Mandel-Lifshitz product
model at weak coupling.

\subsection{Model and Results}\label{s.s1}

Let $a,b$ be two positive real numbers and consider the Fibonacci
substitution,
$$
S(a) = ab , \quad S(b) = a.
$$
This substitution rule extends to finite and one-sided infinite words by concatenation. For example,
$S(aba) = abaab$. Since $S(a)$ begins with $a$, one obtains a one-sided infinite sequence that is
invariant under $S$ by iterating the substitution rule on $a$ and taking a limit. Indeed, we have
\begin{equation}\label{f.fibrec}
S^k(a) = S^{k-1}(S(a)) = S^{k-1}(ab) = S^{k-1}(a) S^{k-1}(b) = S^{k-1}(a) S^{k-2}(a).
\end{equation}
In particular, $S^k(a)$ starts with $S^{k-1}(a)$ and hence there is
a unique one-sided infinite sequence $u$, the so-called Fibonacci
substitution sequence, that starts with $S^k(a)$ for every $k$. The
hull $\Omega_{a,b}$ is then obtained by considering all two-sided
infinite sequences that locally look like $u$,
$$
\Omega_{a,b} = \{ \omega \in \{a,b\}^\Z : \text{ every subword of
$\omega$ is a subword of } u \}.
$$
It is known that, conversely, every subword of $u$ is a subword of
every $\omega \in \Omega_{a,b}$. In this sense, $u$ and all elements
of the hull $\omega$ look exactly the same locally.

We wish to single out a special element of $\Omega_{a,b}$. Notice
that $ba$ occurs in $u$ and that $S^2(a) = aba$ begins with $a$
and $S^2(b) = ab$ ends with $b$. Thus, iterating $S^2$ on $b | a$,
where $|$ denotes the eventual origin, we obtain as a limit a
two-sided infinite sequence which belongs to $\Omega_{a,b}$ and
coincides with $u$ to the right of the origin. This element of
$\Omega_{a,b}$ will be denoted by $\omega_s$.

Each $\omega \in \Omega_{a,b}$ generates a Jacobi matrix $H_\omega$
acting in $\ell^2(\Z)$,
$$
(H_\omega \psi)_n = \omega_{n+1} \psi_{n+1} + \omega_{n}
\psi_{n-1}.
$$
With respect to the standard orthonormal basis $\{ \delta_n \}_{n
\in \Z}$ of $\ell^2(\Z)$, where $\delta_n$ is one at $n$ and
vanishes otherwise, this operator has the matrix
$$
\begin{pmatrix}
\ddots & \ddots & \ddots & {} & {} & {} \\
\ddots & 0 & \omega_{-1} & 0 & {} & {} \\
\ddots & \omega_{-1} & 0 & \omega_0 & 0 & \vphantom{\ddots} \\
{} & 0 & \omega_0 & 0 & \omega_1 & \ddots  \\
{} & {} & 0 & \omega_1 & 0 & \ddots   \\
{} & {} & {} & \ddots & \ddots & \ddots
\end{pmatrix}
$$
and it is clearly self-adjoint.

This family of operators, $\{H_\omega\}_{\omega \in \Omega_{a,b}}$,
is called the off-diagonal Fibonacci model. Of course, the structure
of the Fibonacci sequence disappears when $a = b$. In this case, the
hull consists of a single element, the constant two-sided infinite
sequence taking the value $a = b$, and the spectrum and the spectral
measures of the associated operator $H_\omega$ are well understood.
For this reason, we will below always assume that $a \not= b$.
Nevertheless, the limiting case, where we fix $a$, say, and let $b$
tend to $a$ is of definite interest.

Our first result concerns general properties of the spectrum of $H_\omega$. For $S \subset
\R$, we denote by $\text{dim}_H S$ the Hausdorff dimension of $S$
and by $\text{dim}_B S$ the box counting dimension of $S$ (which is
then implicitly claimed to exist).

\begin{Thm}\label{t.app1}
Suppose $a,b > 0$ and $a \not= b$. Then, there exists a compact
set $\Sigma_{a,b} \subset \R$ such that $\sigma(H_\omega) =
\Sigma_{a,b}$ for every $\omega \in \Omega_{a,b}$, and
\begin{itemize}

\item[{\rm (i)}] $\Sigma_{a,b}$ has zero Lebesgue measure.

\item[{\rm (ii)}] The Hausdorff dimension $\dim_H \Sigma_{a,b}$ is
an analytic function of $a$ and $b$.

\item[{\rm (iii)}]  $0<\dim_H \Sigma_{a,b}<1$.
\end{itemize}
\end{Thm}

More can be said about the spectrum when $a$ and $b$ are close to each other:

\begin{Thm}\label{t.app2}
There exists $\varepsilon_0>0$ such that if $a,b > 0$, $a \not= b$, and $\frac{a^2 +
b^2}{2ab} < 1 + \varepsilon_0$ {\rm (}in other words, if $a$ and
$b$ are close enough{\rm )}, then

\begin{itemize}

\item[{\rm (iv)}] The spectrum $\Sigma_{a,b}$ is a Cantor set that depends
continuously on $a$ and $b$ in the Hausdorff metric.

\item[{\rm (v)}] For every small $\delta>0$ and every $E \in
\Sigma_{a,b}$, we have
\begin{align*}
\dim_H \left( (E - \delta, E + \delta) \cap \Sigma_{a,b} \right) & =
\dim_B \left( (E - \delta, E + \delta) \cap \Sigma_{a,b} \right) \\
& = \dim_H \Sigma_{a,b}  =\dim_B \Sigma_{a,b}.
\end{align*}

\item[{\rm (vi)}] Denote $\alpha = \dim_H \Sigma_{a,b}$, then the
Hausdorff $\alpha$-measure of $\Sigma_{a,b}$ is positive and
finite.

\item[{\rm (vii)}] We have that
$$
\Sigma_{a,b} + \Sigma_{a,b} = [\min \Sigma_{a,b} + \min
\Sigma_{a,b} , \max \Sigma_{a,b} + \max \Sigma_{a,b}].
$$
\end{itemize}
\end{Thm}

Given these results, and especially Theorem~\ref{t.app2}.(vii), we can confirm rigorously the observation made by Even-Dar Mandel and Lifshitz in \cite{EL06,EL07} that the square Fibonacci Hamiltonian (based on the off-diagonal one-dimensional model) has no gaps in its spectrum for sufficiently small coupling.

Next, we turn to the spectral type of $H_\omega$.

\begin{Thm}\label{t.app3}
Suppose $a,b > 0$ and $a \not= b$. Then, for every $\omega \in
\Omega_{a,b}$, $H_\omega$ has purely singular continuous spectrum.
\end{Thm}

Throughout the rest of this Appendix we will only consider $a,b > 0$
with $a \not= b$. Theorems~\ref{t.app1} and \ref{t.app2} are proved in
Subsection~\ref{s.s2} and Theorem~\ref{t.app3} is proved in
Subsection~\ref{s.s3}.

\subsection{The Trace Map and its Relation to the Spectrum}\label{s.s2}

The spectral properties of $H_\omega$ are closely related to the
behavior of the solutions to the difference equation
\begin{equation}\label{f.eve}
\omega_{n+1} u_{n+1} + \omega_{n} u_{n-1} = E u_n.
\end{equation}
Denote
$$
U_n = \begin{pmatrix} u_n \\ \omega_{n} u_{n-1} \end{pmatrix}.
$$
Then $u$ solves \eqref{f.eve} (for every $n \in \Z$) if and only if $U$ solves
\begin{equation}\label{f.tmatrix}
U_{n} = T_\omega(n,E) U_{n-1},
\end{equation}
(for every $n \in \Z$), where
$$
T_\omega(n,E) = \frac{1}{\omega_{n}} \begin{pmatrix} E & -1 \\
\omega_{n}^2 & 0 \end{pmatrix}.
$$
Note that $\det T_\omega(n,E) = 1$. Iterating \eqref{f.tmatrix}, we find
$$
U_{n} = M_\omega(n,E) U_0,
$$
where
$$
M_\omega(n,E) = T_\omega(n,E) \times \cdots T_\omega(1,E).
$$
With the Fibonacci numbers $\{F_k\}$, generated by $F_0 = F_1 = 1$, $F_{k+1} = F_k + F_{k-1}$
for $k \ge 1$, we define
$$
x_k = x_k(E) = \frac12 \mathrm{Tr} M_{\omega_s}(F_k,E).
$$
For example, we have
\begin{align*}
M_{\omega_s}(F_1,E) & = \frac{1}{a} \begin{pmatrix} E & -1 \\ a^2 & 0 \end{pmatrix} \\
M_{\omega_s}(F_2,E) & = \frac{1}{b} \begin{pmatrix} E & -1 \\ b^2 & 0 \end{pmatrix} \frac{1}{a}
\begin{pmatrix} E & -1 \\ a^2 & 0 \end{pmatrix} = \frac{1}{ab} \begin{pmatrix} E^2 - a^2 & -E
\\ E b^2 & - b^2 \end{pmatrix} \\
M_{\omega_s}(F_3,E) & = \frac{1}{a} \begin{pmatrix} E & -1 \\ a^2 & 0 \end{pmatrix} \frac{1}{b}
\begin{pmatrix} E & -1 \\ b^2 & 0 \end{pmatrix} \frac{1}{a} \begin{pmatrix} E & -1 \\ a^2 & 0
\end{pmatrix} = \frac{1}{a^2 b} \begin{pmatrix} E^3 - E a^2 - E b^2 & -E^2 + b^2 \\ E^2 a^2 -
a^4 & - E a^2 \end{pmatrix}
\end{align*}
and hence
\begin{equation}\label{f.tmic}
x_1 = \frac{E}{2a} , \quad x_2 = \frac{E^2 - a^2 - b^2}{2ab},
\quad x_3 = \frac{E^3 - 2 E a^2 - E b^2}{2a^2 b}.
\end{equation}

\begin{Lm}\label{l.l1}
We have
\begin{equation}\label{f.tmrec}
x_{k+1} = 2 x_k x_{k-1} - x_{k-2}
\end{equation}
for $k \ge 2$. Moreover, the quantity
\begin{equation}\label{f.tminv}
I_k = x_{k+1}^2 + x_k^2 + x_{k-1}^2 - 2 x_{k+1} x_k x_{k-1} - 1
\end{equation}
is independent of both $k$ and $E$ and it is given by
$$
I = \frac{(a^2 + b^2)^2}{4a^2b^2} - 1.
$$
\end{Lm}

\begin{proof}
Since $\omega_s$ restricted to $\{n \ge 1\}$ coincides with $u$
and the prefixes $s_k$ of $u$ of length $F_k$ obey $s_{k+1} = s_k
s_{k-1}$ for $k \ge 2$ by construction (cf.~\eqref{f.fibrec}), the
recursion \eqref{f.tmrec} follows as in the diagonal case; compare
\cite{D00,D07a,S87}. This recursion in turn implies readily that
\eqref{f.tminv} is $k$-independent.

In particular, the $x_k$'s are again generated by the trace map
$$
T(x,y,z) = (2xy - z , x , y)
$$
and the preserved quantity is again
$$
I(x,y,z) = x^2 + y^2 + z^2 - 2xyz - 1.
$$
The only difference between the diagonal and the off-diagonal model can be found in the initial
conditions. How are $x_1, x_0, x_{-1}$ obtained? Observe that the trace map is invertible and
hence we can apply its inverse twice to the already defined quantity $(x_3,x_2,x_1)$. We have
$$
T^{-1}(x,y,z) = (y,z, 2yz - x)
$$
and hence, using \eqref{f.tmic},
\begin{align*}
(x_1,x_0,x_{-1}) & = T^{-2} ( x_3 , x_2 , x_1 ) \\
& = T^{-2} \left( \frac{E^3 - 2 E a^2 - E b^2}{2a^2 b} , \frac{E^2 - a^2 - b^2}{2ab}, \frac{E}{2a}
\right) \\
& = T^{-1} \left( \frac{E^2 - a^2 - b^2}{2ab}, \frac{E}{2a} , 2 \frac{(E^2 - a^2 - b^2)E}{4a^2 b} -
\frac{E^3 - 2 E a^2 - E b^2}{2a^2 b} \right) \\
& = T^{-1} \left( \frac{E^2 - a^2 - b^2}{2ab}, \frac{E}{2a} , \frac{E}{2b} \right) \\
& = \left( \frac{E}{2a} , \frac{E}{2b} , 2 \frac{E^2}{4ab} - \frac{E^2 - a^2 - b^2}{2ab} \right) \\
& = \left( \frac{E}{2a} , \frac{E}{2b} , \frac{a^2 + b^2}{2ab} \right)
\end{align*}
It follows that
\begin{align*}
I(x_{k+1} , x_k , x_{k-1}) & = I(x_1 , x_0 , x_{-1}) \\
& = \frac{E^2}{4a^2} + \frac{E^2}{4b^2} + \frac{(a^2 + b^2)^2}{4a^2b^2} - 2 \frac{E^2(a^2 + b^2)}{8a^2b^2}
- 1 \\
& = \frac{(a^2 + b^2)^2}{4a^2b^2} - 1
\end{align*}
for every $k \ge 0$.
\end{proof}

It is of crucial importance for the spectral analysis that, as in
the diagonal case, the invariant is energy-independent and
strictly positive when $a \not= b$!

\begin{Lm}\label{l.l2}
The spectrum of $H_\omega$ is independent of $\omega$ and may be
denoted by $\Sigma_{a,b}$. With
$$
\sigma_k = \{ E \in \R : |x_k| \le 1 \},
$$
we have
\begin{equation}\label{f.descspec}
\Sigma_{a,b} = \bigcap_{k \ge 1} \sigma_k \cup \sigma_{k+1}.
\end{equation}
Moreover, for every $E \in \Sigma_{a,b}$ and $k \ge 2$,
\begin{equation}\label{t.tmbound}
|x_k| \le 1 + \left( \frac{(a^2 + b^2)^2}{4a^2b^2} - 1 \right)^{1/2}
\end{equation}
and for $E \not\in \Sigma_{a,b}$, $|x_k|$ diverges
super-exponentially.
\end{Lm}

\begin{proof}
It is well known that the hull $\Omega_{a,b}$ together with the
standard shift transformation is minimal. In particular, every
$\omega \in \Omega_{a,b}$ may be approximated pointwise by a
sequence of shifts of any other $\tilde \omega \in \Omega_{a,b}$.
The associated operators then converge strongly and we get
$\sigma(H_\omega) \subseteq \sigma(H_{\tilde \omega})$. Reversing
the roles of $\omega$ and $\tilde \omega$, the first claim follows.

So let $\Sigma_{a,b}$ denote the common spectrum of the operators
$H_\omega$, $\omega \in \Omega_{a,b}$. We have $\|H_\omega\| \le
\max \{ 2a,2b \}$. Thus, $\Sigma_{a,b} \subseteq [- \max \{ 2a,2b \}
, \max \{ 2a,2b \} ] =: I_{a,b}$. For $E \in I_{a,b}$, we have that
at least one of $x_1$, $x_0$ belongs to $[-1,1]$. This observation
allows us to use the exact same arguments S\"ut\H{o} used to prove
\eqref{f.descspec} for the diagonal model in \cite{S87}.

The only point where care needs to be taken is the claim
that $\sigma_k$ is the spectrum of the periodic Jacobi
matrix obtained by repeating the values $\omega_s$ takes on $\{ 1
\le n \le F_k \}$ periodically on the off-diagonals. This,
however, follows from the general theory of periodic Jacobi
matrices, which relies on the diagonalization of the monodromy
matrix (which is $M_{\omega_s}(F_k,E)$ in this case) in order to
obtain Floquet solutions and in particular discriminate between
those energies that permit exponentially growing solutions and
those that do not. This distinction works just as well here, but
one needs to use that the $\omega_n$'s that enter in the $U_n$'s
are uniformly bounded away from zero and infinity.

Thus, after paying attention to this fact, we may now proceed along
the lines of S\"ut\H{o}. Let us describe the main steps of the
argument. Since at least one of $x_1$, $x_0$ belongs to $[-1,1]$, we
have a result analogous to \cite[Lemma~2]{S87} with the same proof
as given there. Namely, the sequence $\{x_k\}_{k \ge 0}$ is
unbounded if and only if there exists $k$ such that $|x_{k}|> 1$ and
$|x_{k+1}| > 1$. Moreover, we then have $|x_{k + l}| > c^{F_l}$ for
some $c > 1$ and all $l \ge 0$. This shows
$$
\sigma_k \cup \sigma_{k+1} = \bigcup_{l \ge 0} \sigma_{k + l}.
$$
Using now the fact that the $F_k$ periodic Jacobi matrices with
spectrum $\sigma_k$ converge strongly to $H_{\omega_s}$, we obtain
$$
\Sigma_{a,b} \subseteq \bigcap_{k \ge 1} \overline{\bigcup_{l \ge 0}
\sigma_{k + l}} = \bigcap_{k \ge 1} \overline{\sigma_k \cup
\sigma_{k+1}} = \bigcap_{k \ge 1} \sigma_k \cup \sigma_{k+1},
$$
since the spectra $\sigma_k$ and $\sigma_{k+1}$ are closed sets.
Thus, we have one inclusion in \eqref{f.descspec}.

Next, suppose $E \in \bigcap_{k \ge 1} \sigma_k \cup \sigma_{k+1}$
If $k \ge 1$ is such that $|x_k| > 1$, then $|x_{k-1}| \le 1$ and
$|x_{k+1}| \le 1$. Since we have
$$
x_{k+1}^2 + x_k^2 + x_{k-1}^2 - 2 x_{k+1} x_k x_{k-1} -1 =
\frac{(a^2 + b^2)^2}{4a^2b^2} - 1 ,
$$
this implies
$$
x_k = x_{k+1} x_{k-1} \pm \left( 1 - x_{k+1}^2 - x_{k-1}^2 +
x_{k+1}^2 x_{k-1}^2 + \frac{(a^2 + b^2)^2}{4a^2b^2} - 1
\right)^{1/2}
$$
and hence
$$
|x_k| \le |x_{k+1} x_{k-1}| + \left( (1 - x_{k+1}^2) (1 -
x_{k-1}^2) + \left( \frac{(a^2 + b^2)^2}{4a^2b^2} - 1 \right)
\right)^{1/2}
$$
which, using $|x_{k-1}| \le 1$ and $|x_{k+1}| \le 1$ again, implies
the estimate \eqref{t.tmbound} for $E \in \bigcap_{k \ge 1} \sigma_k
\cup \sigma_{k+1}$. We will show in the next subsection that the
boundedness of the sequence $\{x_k\}_{k \ge 0}$ implies that, for
arbitrary $\omega \in \Omega_{a,b}$, no solution of the difference
equation \eqref{f.eve} is square-summable at $+\infty$.
Consequently, such $E$'s belong to $\Sigma_{a,b}$.\footnote{This
follows by a standard argument: If $E \not\in \Sigma_{a,b}$, then
$(H_\omega - E)^{-1}$ exists and hence $(H_\omega-E)^{-1} \delta_0$
is an $\ell^2(\Z)$ vector that solves \eqref{f.eve} away from the
origin. Choosing its values for $n \ge 1$, say, and then using
\eqref{f.eve} to extend it to all of $\Z$, we obtain a solution that
is square-summable at $+\infty$.} This shows the other inclusion in
\eqref{f.descspec} and hence establishes it. Moreover, it follows
that \eqref{t.tmbound} holds for every $E \in \Sigma_{a,b}$.

Finally, from the representation \eqref{f.descspec} of
$\Sigma_{a,b}$ and our observation above about unbounded sequences
$\{x_k\}_{k \ge 0}$, we find that $|x_k|$ diverges
super-exponentially for $E \not\in \Sigma_{a,b}$. This concludes
the proof of the lemma.
\end{proof}

\begin{Lm}\label{l.l3}
For every $E \in \R$, there is $\gamma(E) \ge 0$ such that
$$
\lim_{n \to \infty} \frac{1}{n} \log \| M_\omega(n,E) \| =
\gamma(E),
$$
uniformly in $\omega \in \Omega_{a,b}$.
\end{Lm}

\begin{proof}
This follows directly from the uniform subadditive ergodic theorem;
compare \cite{DL99b,Ho}.
\end{proof}

\begin{Lm}\label{l.l4}
The set $\mathcal{Z}_{a,b} := \{ E \in \R : \gamma(E) = 0 \}$ has
zero Lebesgue measure.
\end{Lm}

\begin{proof}
This is one of the central results of Kotani theory; see
\cite{Ko} and also \cite{D07b}. Note that these papers only
discuss the diagonal model. Kotani theory for Jacobi matrices is
discussed in Carmona-Lacroix \cite{CL} and the result needed can
be deduced from what is presented there. For a recent reference
that states a result sufficient for our purpose explicitly, see
Remling \cite{Re}.
\end{proof}

\begin{Lm}\label{l.l5}
We have $\Sigma_{a,b} = \mathcal{Z}_{a,b}$.
\end{Lm}

\begin{proof}
The inclusion $\Sigma_{a,b} \supseteq \mathcal{Z}_{a,b}$ holds by
general principles. For example, one can construct Weyl sequences by
truncation when $\gamma(E) = 0$. The inclusion $\Sigma_{a,b}
\subseteq \mathcal{Z}_{a,b}$ can be proved in two ways. Either one
uses the boundedness of $x_k$ for energies $E \in \Sigma_{a,b}$ to
prove explicit polynomial upper bounds for $\|M_\omega(n,E)\|$ (as
in \cite{IT} for $\omega = \omega_s$ or in \cite{DL99b} for general
$\omega \in \Omega_{a,b}$), or one combines the proof of the absence
of decaying solutions at $+\infty$ for $E \in \Sigma_{a,b}$ given in
the next subsection with Osceledec's Theorem, which states that
$\gamma(E) > 0$ would imply the existence of an exponentially
decaying solution at $+\infty$. Here we use one more time that $U_n$
is comparable in norm to $(u_n,u_{n-1})^T$.
\end{proof}

\begin{proof}[Proof of Theorems~\ref{t.app1} and \ref{t.app2}.]
The existence of the uniform spectrum $\Sigma_{a,b}$ was shown in
Lemma~\ref{l.l3} and the fact that $\Sigma_{a,b}$ has zero
Lebesgue measure follows from Lemmas~\ref{l.l4} and \ref{l.l5}.
The set of bounded orbits of the restriction of the trace map
$T:\mathbb{R}^3\to \mathbb{R}^3$ to the invariant surface
$I(x,y,z)=C\equiv\frac{(a^2 + b^2)^2}{4a^2b^2} - 1, C>0,$ is
hyperbolic; see \cite{Can} (and also \cite{DG09a} for $C$
sufficiently small and \cite{Cas} for $C$ sufficiently large). Due
to Lemma~\ref{l.l2}, the points of the spectrum correspond to the
points of the intersection of the line of the initial conditions
$$
\ell_{a,b}\equiv\left\{ \left( \frac{E}{2a} , \frac{E}{2b} , \frac{a^2
+ b^2}{2ab} \right) : E \in \mathbb{R}\right\}
$$
with the stable manifolds of the hyperbolic set of bounded orbits.
Properties~(ii) and (iii) can be proved in exactly the same way as
Theorem~6.5 in \cite{Can}. The line $\ell_{a,b}$ intersects the
stable lamination of the hyperbolic set transversally for
sufficiently small $C > 0$, as can be shown in the same way as for
the diagonal Fibonacci Hamiltonian with a small coupling constant;
see \cite{DG09a}. Therefore the spectrum $\Sigma_{a,b}$ for close
enough $a$ and $b$ is a dynamically defined Cantor set, and the
properties (iv)--(vi) follow; see \cite{DEGT, DG09a, Ma, MM, P, PT}
and references therein. The statement (vii) follows as in the diagonal case since the thickness of $\Sigma_{a,b}$ tends to infinity as $\frac{a^2 +
b^2}{2ab}$ approaches $1$.
\end{proof}

Notice that a proof of the transversality of the line $\ell_{a,b}$ to
the stable lamination of the hyperbolic set of bounded orbits for
arbitrary $a\ne b$ would imply the properties (iv)--(vi) for these
values of $a$ and $b$.

\subsection{Singular Continuous Spectrum}\label{s.s3}

In this subsection we prove Theorem~\ref{t.app3}. Given the results from the previous subsection, we can
follow the proofs from the diagonal case quite closely.

\begin{proof}[Proof of Theorem~\ref{t.app3}.]
Since the absence of absolutely continuous spectrum follows from
zero measure spectrum, we only need to show the absence of point
spectrum. It was shown by Damanik and Lenz \cite{DL99a} that, given any
$\omega \in \Omega_{a,b}$ and $k \ge 1$, the restriction of $\omega$
to $\{ n \ge 1 \}$ begins with a square
$$
\omega_1 \ldots \omega_{2F_k} \ldots = \omega_1 \ldots \omega_{F_k} \omega_1 \ldots \omega_{F_k} \ldots
$$
such that $\omega_1 \ldots \omega_{F_k}$ is a cyclic permutation
of $S^k(a)$. By cyclic invariance of the trace, it follows that
$\mathrm{Tr} M_\omega(F_k,E) = 2 x_k(E)$ for every $E$.

The Cayley-Hamilton Theorem, applied to $M_\omega(F_k,E)$, says
that
$$
M_\omega(F_k,E)^2 - \left( \mathrm{Tr} M_\omega(F_k,E) \right)
M_\omega(F_k,E) + I = 0,
$$
which, by the observations above, translates to
$$
M_\omega(2F_k,E) - 2x_k M_\omega(F_k,E) + I = 0.
$$

If $E \in \Sigma_{a,b}$ and $u$ is a solution of the difference
equation \eqref{f.eve}, it therefore follows that
$$
U(2F_k + 1) - 2x_k U(F_k + 1) + U(1) = 0.
$$
If $u$ does not vanish identically, this shows that $u_n \not\to
0$ as $n \to \infty$ since the $x_k$'s are bounded above and the
$\omega_n$'s are bounded below away from zero. In particular, if
$E \in \Sigma_{a,b}$, then no non-trivial solution of
\eqref{f.eve} is square-summable at $+\infty$ and hence $E$ is not
an eigenvalue. It follows that the point spectrum of $H_\omega$ is
empty.
\end{proof}

\end{appendix}


\begin{thebibliography}{OPRSS}
\def\bi#1{\bibitem[#1]{#1}}


\bi{As00} S.\ Astels, Cantor sets and numbers with restricted partial quotients, \textit{Trans.\ Amer.\ Math.\ Soc.}\  \textbf{352}  (2000), 133--170.

\bi{As01} S.\ Astels, Sums of numbers with small partial quotients. II,  \textit{J.\ Number Theory} \textbf{91} (2001), 187--205.

\bi{As02} S.\ Astels, Sums of numbers with small partial quotients, \textit{Proc.\ Amer.\ Math.\ Soc.}\  \textbf{130} (2002), 637--642.

\bi{AJ} A.\ Avila, S.\ Jitomirskaya, The Ten Martini Problem, \textit{Ann.\ of Math.}\ \textbf{170} (2009), 303--342.

\bi{AS} J.\ Avron, B.\ Simon, Almost periodic Schr\"odinger operators. II.~The integrated density of states,
\textit{Duke Math.\ J.}\ \textbf{50} (1983), 369--391.

\bi{BGJ} M.\ Baake, U.\ Grimm, D.\ Joseph, Trace maps, invariants,
and some of their applications, \textit{Internat.\ J.\ Modern
Phys.~B} \textbf{7} (1993), 1527--1550.

\bi{BR} M.\ Baake, J.\ Roberts, The dynamics of trace maps, in
\textit{Hamiltonian Mechanics} (\textit{Toru\'n, 1993}), 275--285,
NATO Adv.\ Sci.\ Inst.\ Ser.~B Phys.\ \textbf{331}, Plenum, New
York, 1994.

\bi{B} J.\ Bellissard, Spectral properties of Schr\"odinger's
operator with a Thue-Morse potential, \textit{Number Theory and
Physics} (\textit{Les Houches, 1989}), 140--150, Springer Proc.\
Phys.\ \textbf{47}, Springer, Berlin, 1990.

\bi{BBG91} J.\ Bellissard, A.\ Bovier, J.-M.\ Ghez, Spectral
properties of a tight binding Hamiltonian with period doubling
potential, \textit{Commun.\ Math.\ Phys.}\ \textbf{135} (1991),
379--399.

\bi{BBG92} J.\ Bellissard, A.\ Bovier, J.-M.\ Ghez, Gap labelling
theorems for one-dimensional discrete Schr\"odinger operators,
\textit{Rev.\ Math.\ Phys.} \textbf{4} (1992), 1--37.

\bi{BS} J.\ Bellissard, H.\ Schulz-Baldes, Subdiffusive quantum
transport for 3D Hamiltonians with absolutely continuous spectra,
\textit{J.\ Statist.\ Phys.}\ \textbf{99} (2000), 587--594.

\bi{BG} A.\ Bovier, J.-M.\ Ghez, Remarks on the spectral
properties of tight-binding and Kronig-Penney models with
substitution sequences, \textit{J.\ Phys.\ A} \textbf{28} (1995),
2313--2324.

\bi{Can} S.\ Cantat, Bers and H\'enon, Painlev\'e and Schr\"odinger,
\textit{Duke Math.\ J.}\ \textbf{149} (2009), 411--460.

\bi{CL} R.\ Carmona, J.\ Lacroix, \textit{Spectral Theory of Random Schr\"odinger Operators},
Birkh\"auser, Boston, MA, 1990.

\bi{Cas} M.\ Casdagli, Symbolic dynamics for the renormalization
map of a quasiperiodic Schr\"odinger equation, \textit{Comm.\
Math.\ Phys.}\ \textbf{107} (1986), 295--318.

\bi{Cus} T.\ Cusick, On M. Hall's continued fraction theorem,
\textit{Proc. Amer. Math. Soc.}\  \textbf{38} (1973), 253--254.

\bi{CFKS} H.\ Cycon, R.\ Froese, W.\ Kirsch, B.\ Simon, \textit{Schr\"odinger Operators
with Application to Quantum Mechanics and Global Geometry}, Texts and Monographs in
Physics, Springer-Verlag, Berlin, 1987.

\bi{D00} D.\ Damanik, Gordon-type arguments in the spectral theory
of one-dimensional quasicrystals, in \textit{Directions in
Mathematical Quasicrystals}, 277--305, CRM Monogr.\ Ser.\
\textbf{13}, Amer. Math. Soc., Providence, RI, 2000.

\bi{D07a} D.\ Damanik, Strictly ergodic subshifts and associated
operators, in \textit{Spectral Theory and Mathematical Physics: a
Festschrift in Honor of Barry Simon's 60th Birthday},  505--538,
Proc.\ Sympos.\ Pure Math.\ \textbf{76}, Part 2, Amer. Math. Soc.,
Providence, RI, 2007.

\bi{D07b} D.\ Damanik, Lyapunov exponents and spectral analysis of
ergodic Schr\"odinger operators: a survey of Kotani theory and its
applications, in \textit{Spectral Theory and Mathematical Physics:
a Festschrift in Honor of Barry Simon's 60th Birthday}, 539--563,
Proc.\ Sympos.\ Pure Math.\ \textbf{76}, Part 2, Amer. Math. Soc.,
Providence, RI, 2007.

\bi{DEGT} D.\ Damanik, M.\ Embree, A.\ Gorodetski, S.\
Tcheremchantsev, The fractal dimension of the spectrum of the
Fibonacci Hamiltonian, \textit{Commun.\ Math.\ Phys.}\
\textbf{280} (2008), 499--516.

\bi{DG09a} D.\ Damanik, A.\ Gorodetski, Hyperbolicity of the Trace
Map for the Weakly Coupled Fibonacci Hamiltonian,
\textit{Nonlinearity}, \textbf{22} (2009), 123--143.

\bi{DG09b} D.\ Damanik, A.\ Gorodetski, The Spectrum of the Weakly Coupled Fibonacci Hamiltonian, \textit{Electronic Research Announcements in Mathematical Sciences}, \textbf{16} (2009), 23--29.

\bi{DKL} D.\ Damanik, R.\ Killip, D.\ Lenz, Uniform spectral
properties of one-dimensional quasicrystals.
III.~$\alpha$-continuity, \textit{Comm.\ Math.\ Phys.}\
\textbf{212} (2000), 191--204.

\bi{DL99a} D.\ Damanik, D.\ Lenz, Uniform spectral properties of
one-dimensional quasicrystals, I.~Absence of eigenvalues,
\textit{Commun.\ Math.\ Phys.}\ \textbf{207} (1999), 687--696.

\bi{DL99b} D.\ Damanik, D.\ Lenz, Uniform spectral properties of
one-dimensional quasicrystals. II.~The Lyapunov exponent,
\textit{Lett.\ Math.\ Phys.}\ \textbf{50} (1999), 245--257.

\bi{DST} D.\ Damanik, A.\ S\"ut\H{o}, S.\ Tcheremchantsev,
Power-Law bounds on transfer matrices and quantum dynamics in one
dimension II., \textit{J.\ Funct.\ Anal.} \textbf{216} (2004),
362--387.

\bi{DT03} D.\ Damanik, S.\ Tcheremchantsev, Power-law bounds on
transfer matrices and quantum dynamics in one dimension,
\textit{Commun.\ Math.\ Phys.}\ \textbf{236} (2003), 513--534.

\bi{DT07} D.\ Damanik, S.\ Tcheremchantsev, Upper bounds in quantum
dynamics, \textit{J.\ Amer.\ Math.\ Soc.}\ \textbf{20} (2007),
799--827.

\bi{DT08} D.\ Damanik and S.\ Tcheremchantsev, Quantum dynamics
via complex analysis methods: general upper bounds without
time-averaging and tight lower bounds for the strongly coupled
Fibonacci Hamiltonian, \textit{J.\ Funct.\ Anal.}\ \textbf{255}
(2008), 2872--2887.

\bi{EL06} S.\ Even-Dar Mandel, R.\ Lifshitz, Electronic energy
spectra and wave functions on the square Fibonacci tiling,
\textit{Phil.\ Mag.}\ \textbf{86} (2006), 759--764.

\bi{EL07} S.\ Even-Dar Mandel, R.\ Lifshitz, Electronic energy
spectra of square and cubic Fibonacci quasicrystals,
\textit{Phil.\ Mag.}\ \textbf{88} (2008), 2261--2273.

\bi{EL08} S.\ Even-Dar Mandel, R.\ Lifshitz, Bloch-like electronic
wave functions in two-dimensional quasicrystals, Preprint
(arXiv:0808.3659).

\bi{Ha} M.\ Hall, On the sum and product of continued fractions,
\textit{Ann.\ of Math.}\ \textbf{48} (1947), 966--993.

\bi{Hl} J.\ Hlavka, Results on sums of continued fractions,
\textit{Trans. Amer. Math. Soc.}\ \textbf{211} (1975), 123--134.

\bi{HPS} M.\ Hirsch, C.\ Pugh, M.\  Shub, \textit{Invariant Manifolds},
Lecture Notes in Mathematics \textbf{583}, Springer-Verlag, Berlin-New York, 1977.

\bi{Ho} A.\ Hof, Some remarks on discrete aperiodic Schr\"odinger operators, \textit{J.\ Statist.\ Phys.}\
\textbf{72} (1993), 1353--1374.

\bi{HCB} V.\ Hoggatt, N.\ Cox, M.\ Bicknell, A primer for the Fibonacci numbers.~XII., \textit{Fibonacci Quart.}\ \textbf{11} (1973), 317--331.

\bi{HM} S.\ Humphries, A.\ Manning, Curves of fixed points of trace maps, \textit{Ergod. Th. \& Dynam. Sys.}  \textbf{27} (2007), 1167--1198.

\bi{IT} B.\ Iochum, D.\ Testard, Power law growth for the
resistance in the Fibonacci model, \textit{J.\ Stat.\ Phys.}
\textbf{65} (1991), 715--723.

\bi{Ka} L.\ P.\ Kadanoff, Analysis of cycles for a volume
preserving map, unpublished manuscript.

\bi{KKT} M.\ Kohmoto, L.\ P.\ Kadanoff, C.\ Tang, Localization
problem in one dimension: mapping and escape, \textit{Phys.\ Rev.\
Lett.}\ \textbf{50} (1983), 1870--1872.

\bi{Ko} S.\ Kotani, Jacobi matrices with random potentials taking finitely many values,
\textit{Rev.\ Math.\ Phys.}\ \textbf{1} (1989), 129--133.

\bi{La} Y.\ Last, Quantum dynamics and decompositions of singular
continuous spectra, \textit{J.\ Funct.\ Anal.}\ \textbf{142}
(1996), 406--445.

\bi{LS} Y.\ Last, B.\ Simon, Fine structure of the zeros of
orthogonal polynomials. IV.~A priori bounds and clock behavior,
\textit{Comm.\ Pure Appl.\ Math.}\ \textbf{61} (2008), 486--538.

\bi{Le} C.\ Lekkerkerker, Representation of natural numbers as a sum of Fibonacci numbers, \textit{Simon Stevin} \textbf{29} (1952), 190--195.

\bi{LW} Q.-H.\ Liu, Z.-Y.\ Wen, Hausdorff dimension of spectrum of
one-dimensional Schr\"odinger operator with Sturmian potentials,
\textit{Potential Anal.}\ \textbf{20} (2004), 33--59.

\bi{Ma} R.\ Man\'e, The Hausdorff dimension of horseshoes of
diffeomorphisms of surfaces, \textit{Bol.\ Soc.\ Brasil.\ Mat.}\
(\textit{N.S.}) \textbf{20} (1990), 1--24.

\bi{Me} W.\ de Melo, Structural stability of diffeomorphisms on
two-manifolds, \textit{Invent.\ Math.}\ \textbf{21} (1973),
233--246.

\bi{MM} A.\ Manning, H.\ McCluskey, Hausdorff dimension for
horseshoes, \textit{Ergodic Theory Dynam.\ Systems} \textbf{3}
(1983), 251--261.


\bi{N79} S.\ Newhouse, The abundance of wild hyperbolic sets and
nonsmooth stable sets for diffeomorphisms,  \textit{Inst.\ Hautes
\'Etudes Sci.\ Publ.\ Math.}\ \textbf{50} (1979), 101--151.

\bi{N70} S.\ Newhouse, Nondensity of axiom ${\rm A}({\rm a})$ on
$S\sp{2}$, \textit{Global Analysis} (\textit{Proc.\ Sympos.\ Pure
Math., Vol.~XIV, Berkeley, Calif., 1968}), 191--202,
Amer.\ Math.\ Soc., Providence, 1970.

\bi{OK} S.\ Ostlund, S.-H.\ Kim, Renormalization of quasiperiodic
mappings, \textit{Physica Scripta} \textbf{T9} (1985), 193--198.


\bi{PT} J.\ Palis, F.\ Takens, \textit{Hyperbolicity and Sensitive
Chaotic Dynamics at Homoclinic Bifurcations}, Cambridge University
Press, 1993.



\bi{P} Ya.\ Pesin, \textit{Dimension Theory in Dynamical Systems},
Chicago Lectures in Mathematics Series, 1997.

\bi{PSW} C.\ Pugh, M.\  Shub, A.\ Wilkinson, H\"older foliations,
\textit{Duke Math.\ J.}\ \textbf{86} (1997), 517--546.

\bi{Ra} L.\ Raymond, A constructive gap labelling for the discrete
Schr\"odinger operator on a quasiperiodic chain, Preprint (1997).

\bi{RS} M.\ Reed, B.\ Simon, \textit{Methods of Modern
Mathematical Physics. I.~Functional Analysis}, 2nd edition,
Academic Press, New York, 1980.

\bi{Re} C.\ Remling, The absolutely continuous spectrum of Jacobi matrices, Preprint (2007).

\bi{Ro} J.\ Roberts, Escaping orbits in trace maps, \textit{Phys.\
A} \textbf{228} (1996), 295--325.

\bi{S} B.\ Simon, Operators with singular continuous spectrum.
VII.~Examples with borderline time decay, \textit{Comm.\ Math.\
Phys.}\ \textbf{176}  (1996), 713--722.

\bi{Si89} C.\ Sire, Electronic spectrum of a 2D quasi-crystal
related to the octagonal quasi-periodic tiling, \textit{Europhys.\
Lett.}\ \textbf{10} (1989), 483--488.

\bi{SM89} C.\ Sire, R.\ Mosseri, Spectrum of 1D quasicrystals near
the periodic chain, \textit{J.\ Phys.\ France} \textbf{50} (1989),
3447--3461.

\bi{SM90} C.\ Sire, R.\ Mosseri, Excitation spectrum, extended
states, gap closing: some exact results for codimension one
quasicrystals, \textit{J.\ Phys.\ France} \textbf{51} (1990),
1569--1583.

\bi{SMS} C.\ Sire, R.\ Mosseri, J.-F.\ Sadoc, Geometric study of a
2D tiling related to the octagonal quasiperiodic tiling,
\textit{J.\ Phys.\ France} \textbf{55} (1989), 3463--3476.

\bi{S87} A.\ S\"ut\H{o}, The spectrum of a quasiperiodic
Schr\"odinger operator, \textit{Commun.\ Math.\ Phys.}\
\textbf{111} (1987), 409--415.

\bi{S89} A.\ S\"ut\H{o}, Singular continuous spectrum on a Cantor
set of zero Lebesgue measure for the Fibonacci Hamiltonian,
\textit{J.\ Stat.\ Phys.}\ \textbf{56} (1989), 525--531.

\bi{S95} A.\ S\"ut\H{o}, Schr\"odinger difference equation with
deterministic ergodic potentials, in \textit{Beyond Quasicrystals}
(\textit{Les Houches, 1994}), 481--549, Springer, Berlin, 1995.

\bi{T} F.\ Takens,  Limit capacity and Hausdorff dimension of
dynamically defined Cantor sets, \textit{Dynamical Systems},
Lecture Notes in Mathematics \textbf{1331} (1988), 196--212.

\bi{Z} E.\ Zeckendorf, A generalized Fibonacci numeration, \textit{Fibonacci Quart.}\
\textbf{10} (1972), 365--372.

\end{thebibliography}
\end{document}